\documentclass[12pt, onecolumn]{IEEEtran}
\usepackage{mathptmx}       
\usepackage{helvet}         
\usepackage{courier}        
\usepackage{type1cm}        
\usepackage{makeidx}         
\usepackage{graphicx}        
\usepackage{multicol}        
\usepackage[bottom]{footmisc}
\usepackage{ifthen}
\usepackage{subfigure}
\usepackage[usenames,dvipsnames]{color}

\usepackage{setspace}
\makeindex             

\usepackage{graphics} 
\usepackage{graphicx} 
\usepackage{epsfig} 
\usepackage{mathptmx} 
\usepackage{times} 
\usepackage{amsmath} 
\usepackage{amssymb}  

\usepackage{amsfonts}

\usepackage{subfig}
\usepackage{verbatim}

\usepackage{comment}

\DeclareGraphicsExtensions{.pdf,.png,.jpg,.eps}

\def\ba{\begin{array}}
\def\ea{\end{array}}
\newcommand{\beq}{\begin{equation}}
\newcommand{\eeq}{\end{equation}}
\newcommand{\bq}{\begin{eqnarray}}
\newcommand{\eq}{\end{eqnarray}}
\newcommand{\bqn}{\begin{eqnarray*}}
\newcommand{\eqn}{\end{eqnarray*}}
\newcommand{\bee}{\begin{enumerate}}
\newcommand{\eee}{\end{enumerate}}
\newcommand{\bi}{\begin{itemize}}
\newcommand{\ei}{\end{itemize}}

\newcommand{\ii}{\textbf{i}}

\newboolean{showcomments}
\setboolean{showcomments}{true}
\newcommand{\slow}[1]{\ifthenelse{\boolean{showcomments}}
{ \textcolor{red}{(Steven says:  #1)}}{}}

\newtheorem{theorem}{Theorem}

\newtheorem{lemma}[theorem]{Lemma}
\newtheorem{corollary}[theorem]{Corollary}
\newtheorem{remark}{Remark}

\begin{document}

\markboth{IEEE Trans. on Control of Network Systems, June 2014
(with proofs)}
{IEEE Trans. on Control of Network Systems, June 2014
(with proofs)}

\title{Convex Relaxation of Optimal Power Flow\\
{\huge Part II: Exactness}
\thanks{
\hspace{-0.34in}
\textbf{Citation:
\emph{IEEE Transactions on Control of Network Systems, June 2014.}
}
This is an extended version with Appendex VI that proves the main results in this
tutorial.
All proofs can be found in their original papers.  We provide proofs here because
(i) it is convenient to have all proofs in one place and in a uniform notation, 
and (ii) some of the formulations and presentations here are slightly different
from those in the original papers.
\newline
A preliminary and abridged version has appeared in
Proceedings of the IREP Symposium - Bulk Power System Dynamics and Control - IX, Rethymnon, Greece, August 25-30, 2013. 
}
}
\author{Steven H. Low \\ Electrical Engineering, Computing+Mathematical Sciences \\
	Engineering and Applied Science, Caltech \\ 
	slow@caltech.edu \\
}
\date{April 29, 2014}

\maketitle

\vspace{-0.5in}
\begin{center}
May 1, 2014
\end{center}
\vspace{0.5in}

\begin{abstract}
This tutorial summarizes recent advances in the convex relaxation of the optimal
power flow (OPF) problem, focusing on structural properties rather than algorithms.
 Part I presents two power flow models, formulates OPF and their relaxations in each 
 model, and  proves equivalence relations among them.
 Part II presents sufficient conditions under which the convex relaxations are
 exact. 
\end{abstract}

\newpage
 \tableofcontents
 \vfill{
\noindent
\textbf{Acknowledgment.}
We thank the support of NSF through NetSE CNS 0911041, 
ARPA-E through GENI DE-AR0000226, 
Southern California Edison, 
the National Science Council of Taiwan through
NSC 103-3113-P-008-001,
the Los Alamos National Lab (DoE),
and Caltech's Resnick Institute.
}
 \newpage

\section{Introduction}

The optimal power flow (OPF) problem is fundamental in power systems as it underlies
many applications such as economic dispatch, unit commitment, state estimation,
stability and reliability assessment, volt/var control, demand response, etc.
OPF seeks to optimize a certain objective function, such as power loss, generation cost 
and/or user utilities,
subject to Kirchhoff's laws as well as capacity, stability and security
constraints on the voltages and power flows.  There has been a great deal of research 
on OPF since Carpentier's first formulation in 1962 \cite{Carpentier62}.
Recent surveys can be found in, 
e.g., \cite{Powerbk, Huneault91,Momoh99a,Momoh99b,Pandya08, Frank2012a, Frank2012b,
OPF-FERC-1, OPF-FERC-2, OPF-FERC-3, OPF-FERC-4, OPF-FERC-5}.

OPF is generally nonconvex and NP-hard, and a large
number of optimization algorithms and relaxations have been proposed.
To the best of our knowledge solving OPF through semidefinite relaxation is first proposed in 
\cite{Jabr2006} as a second-order
cone program (SOCP) for radial (tree) networks and in \cite{Bai2008} as a semidefinite program 
(SDP) for general networks in a bus injection model.  It is first proposed 
in \cite{Farivar2011-VAR-SGC, Farivar-2013-BFM-TPS} as an SOCP for radial 
networks in the branch flow model of \cite{Baran1989a, Baran1989b}.
While these convex relaxations have been illustrated numerically in \cite{Jabr2006} and
\cite{Bai2008}, whether or when they will turn out to be exact is first studied in \cite{Lavaei2012}.
Exploiting graph sparsity to simplify the SDP relaxation of OPF is first proposed in 
\cite{Bai2011, Jabr2012} and analyzed in \cite{MolzahnLesieutre2013, Bose-2014-BFMe-TAC}.

Solving OPF through convex relaxation offers several advantages, as discussed
in Part I of this tutorial \cite[Section I]{Low2014a}.  In particular it provides the ability 
to check if a solution is globally optimal. If it is not, the solution provides a lower 
bound on the minimum cost and hence a bound on how far any feasible solution 
is from optimality.    Unlike approximations, if a relaxed problem is infeasible, it is 
a certificate that the original OPF is infeasible.

This tutorial presents main results on convex relaxations of OPF developed in the
last few years.   In Part I \cite{Low2014a},
 we present the bus injection model (BIM) and the branch 
flow model (BFM), formulate OPF within each model, and prove their equivalence.
The complexity of OPF formulated here lies in the quadratic nature of power
flows, i.e., the nonconvex quadratic constraints on the feasible set of OPF.   
We characterize these feasible sets 
and design convex supersets that lead to three different convex relaxations based
on semidefinite programming (SDP), chordal extension, and second-order cone 
programming (SOCP). 
When a convex relaxation is exact, an optimal solution of
the original nonconvex OPF can be recovered from {every} optimal solution of the
relaxation.  
In Part II we summarize main sufficient conditions
that guarantee the exactness of these relaxations.  

Network topology turns out to play a critical role in determining whether a
relaxation is exact.   In Section \ref{sec:opf} we review the definitions of OPF 
and their convex relaxations developed in \cite{Low2014a}.    We also define
the notion of exactness adopted in this paper.
In Section \ref{sec:exact-radial} we present three types of sufficient conditions
for these relaxations to be exact for radial networks.   These conditions
are generally not necessary and they have implications on allowable
power injections, voltage magnitudes, or voltage angles:
\bee
\item[A] \emph{Power injections:} These conditions require that not both constraints on
	real and reactive power injections be binding at both ends of a line. 
\item[B] \emph{Voltages magnitudes:} These conditions require that the upper bounds on 
	voltage magnitudes not be binding.  They can be enforced through affine constraints
	on power injections.  
\item[C] \emph{Voltage angles:} These conditions require that the voltage angles across
	each line be sufficiently close.   This is needed also for stability reasons.
\eee
These conditions and their references are summarized in Tables
\ref{table:SummaryRadial} and \ref{table:SummaryMesh}.
\begin{table*}[htbp]
\centering
\begin{tabular}{|| c || l | c | c | l ||}	
  \hline
  \hline
type 	& \qquad condition   &  model & reference & \qquad\qquad \quad remark
\\
\hline\hline
A
	& power injections
	& BIM, BFM
	& \cite{Bose2011, Bose-2012-QCQPt, Zhang2011geometry, Zhang2013, Sojoudi2012PES}
	&
\\
	& 
	&
	& \cite{Sojoudi2013,    Farivar2011-VAR-SGC, Farivar-2013-BFM-TPS}
	& 
\\ \hline
B
	& voltage magnitudes
	& BFM
	& \cite{Li-2012-BFMt, Gan-2012-BFMt, Gan-2013-BFMt-CDC, Gan-2014-BFMt-TAC}
	& allows general injection region
\\ \hline
C
  	& voltage angles   
	& BIM
	& \cite{LavaeiTseZhang2012, LamZhang2012}
	& makes use of branch power flows
\\
\hline\hline
\end{tabular}
\caption{Sufficient conditions for radial (tree) networks.}
\label{table:SummaryRadial}

\vspace{0.2in}
\begin{tabular}{|| c || l | c | l ||}	
  \hline
  \hline
network & \ \  condition  & reference & \qquad\qquad \ remark
\\
\hline\hline
with phase shifters
	& type A, B, C
	& \cite[Part II]{Farivar-2013-BFM-TPS}, \cite{SojoudiLavaei2013}
	& equivalent to radial networks
\\ \hline
direct current
	& type A 
	& \cite[Part I]{Farivar-2013-BFM-TPS}, \cite{Lavaei2012, Rantzer2011}
	& assumes nonnegative voltages
\\ \cline{2-4}
	& type B 
	& \cite{Gan-2013-OPFDC, Gan-2014-OPFDC-TPS}
	& assumes nonnegative voltages
\\
\hline\hline
\end{tabular}
\caption{Sufficient conditions for mesh networks}
\label{table:SummaryMesh}
\end{table*}
Some of these sufficient conditions are proved using BIM and others using 
BFM.  Since these two models are equivalent (in the sense that there is a linear bijection
between their solution sets \cite{Bose-2014-BFMe-TAC, Low2014a}), these sufficient conditions
apply to both models.   
The proofs of these conditions typically do not require that the cost function be convex
(they focus on the feasible sets and usually only need the cost function to be monotonic).
Convexity is required however for efficient computation.
Moreover it is proved in \cite{Gan-2014-BFMt-TAC} using BFM that when
the  cost function is convex then exactness of the SOCP relaxation implies 
uniqueness of the optimal 
solution for radial networks.   Hence the equivalence of BIM and BFM implies that 
any of the three types of sufficient conditions guarantees that, for a radial network with
a convex cost function, there is
a unique optimal solution and it can be computed by solving an SOCP.
Since the SDP and chordal relaxations are equivalent
to the SOCP relaxation for radial networks \cite{Bose-2014-BFMe-TAC, Low2014a}, 
these results apply to all three types of relaxations.
Empirical evidences suggest some of these conditions are likely satisfied
in practice.  This is important as most power distribution systems are radial.

These conditions are insufficient for general mesh networks because
they cannot guarantee that an optimal solution of a relaxation satisfies the
cycle condition discussed in \cite{Low2014a}.  
In Section \ref{sec:exact-mesh} we show that these conditions are however sufficient 
 for mesh networks that have tunable phase shifters at strategic locations.
  The phase shifters effectively make a mesh network behave like a radial network
 as far as convex relaxation is concerned.
 The result can help determine if a network with a given set of
phase shifters can be convexified and, if not, where additional phase shifters
are needed for convexification.  
These  conditions are also sufficient for direct current (dc) mesh
networks where all variables are in the real rather than complex domain.
Counterexamples are known where SDP relaxation is not exact, especially for
AC mesh networks without tunable phase shifters 
\cite{Lesieutre-2011-OPFSDP-Allerton, Bukhsh2013}.
We discuss three recent
approaches for global optimization of OPF when the semidefinite relaxations 
discussed in this tutorial fail.

We conclude in Section \ref{sec:conc2}.
This extended version differs from the journal version only in the addition of 
Appendix VI that proves all main results covered in this tutorial.
Even though all proofs can be found in their original papers,
we provide proofs here because (i) it is convenient to have all proofs in one place
and in a uniform notation, 
and (ii) some of the formulations and presentations here are slightly different
from those in the original papers.

\section{OPF and its relaxations}
\label{sec:opf}

We use the notations and definitions from Part I of this paper.
In this section we summarize the OPF problems and their 
relaxations developed there; see \cite{Low2014a} for details.

We adopt in this paper a strong sense of ``exactness'' where we require
the optimal solution set of the OPF problem and that of its relaxation 
 be equivalent.
This implies that an optimal solution of the 
nonconvex OPF problem can be recovered from \emph{every}  
optimal solution of its  relaxation.
This is important because it ensures any algorithm that solves an
exact relaxation always produces a globally optimal solution to the 
OPF problem.  Indeed interior point methods for solving SDPs
 tend to produce a solution matrix with a maximum rank 
\cite{Louca2013}, so can miss a rank-1 solution if
the relaxation has non-rank-1 solutions as well.  It can be difficult to
recover an optimal solution of OPF from such a non-rank-1 solution, and
our definition of exactness avoids this complication.
See Section \ref{subsec:exact} for detailed justifications.

\subsection{Bus injection model}

The BIM  adopts an undirected graph $G$ \footnote{We will
use ``bus'' and ``node'' interchangeably and ``line'' and ``link'' interchangeably.}
and can be formulated in terms of just the complex voltage vector $V \in \mathbb C^{n+1}$.
The feasible set is described by the following constraints:
\begin{subequations}
\bq
\quad
\underline{s}_j  \ \ \leq  \sum_{k: (j, k) \in E} y_{jk}^H\, V_j (V_j^H - V_k^H)  \ \ 
\leq \ \ \overline{s}_j, 
\qquad  j \in N^+
\label{eq:bimopf.1}
\\
\qquad\qquad
\underline{v}_j \ \leq \ \  |V_j|^2 \ \  \leq\  \overline{v}_j, \qquad\qquad j \in N^+
  \label{eq:bimopf.2}
\eq
\label{eq:bimopf}
\end{subequations}
where $\underline{s}_j, \overline{s}_j, \underline{v}_j, \overline{v}_j$, possibly
$\pm\infty \pm \ii\infty$, are given 
bounds on power injections and voltage magnitudes.  
Note that the vector $V$ includes $V_0$ which is assumed given 
($\underline v_0 = \overline v_0$ and $\angle V_0=0^\circ$) 
unless otherwise specified.
The problem of interest is:
\\
\noindent
\textbf{OPF:}
\bq
\label{eq:OPFbim}
\underset{V \in\mathbb C^{n+1}}{\text{min}}   \ C(V) \ & \text{subject to} &   
		 V \text{ satisfies } \eqref{eq:bimopf}
\eq

For relaxations consider the partial matrix $W_G$ defined on the network graph $G$ that satisfies
\begin{subequations}
\bq
\underline{s}_j \  \leq\!\!  \sum_{k: (j,k)\in E}\!\! y_{jk}^H\, \left( [W_G]_{jj} - [W_G]_{jk} \right) 
 \leq  \ \overline{s}_j,  \qquad j \in N^+
\label{eq:opfW.1}
\\
\underline{v}_j \ \leq \ [W_G]_{jj} \ \leq \ \overline{v}_j,  \qquad\qquad j \in N^+
\label{eq:opfW.2}
\eq
\label{eq:opfW}
\end{subequations}
We say that $W_G$ satisfies the {\em cycle condition} if for every 
cycle $c$ in $G$
\bq
\sum_{(j,k)\in c}\ \angle [W_G]_{jk} & = & 0   \ \ \mod 2\pi
\label{eq:cyclecond.2}
\eq

We assume the cost function $C$ depends on $V$ only through $VV^H$ and
use the same symbol $C$ to denote the cost in terms of a full or partial matrix.
Moreover we assume $C$ depends on the matrix only through the submatrix
$W_G$ defined on the network graph $G$.  See \cite[Section IV]{Low2014a}
for more details including the definitions of $W_{c(G)}\succeq 0$ and
$W_G(j,k)\succeq 0$.  Define the convex relaxations:
\\
\noindent
\textbf{OPF-sdp:}
\bq
\label{eq:bimOPF-sdp}
\hspace{-0.315in}
\underset{W\in \mathbb S^{n+1}}{\text{min}}   \ C(W_G) \ & \text{subject to} &   
		W_G \text{ satisfies } \eqref{eq:opfW}, \  W\succeq 0 
\eq
\noindent
\textbf{OPF-ch:}
\bq
\label{eq:bimOPF-ch}
\underset{W_{c(G)}}{\text{min}}   \ C(W_G) \ & \text{subject to} &  
		W_G \text{ satisfies } \eqref{eq:opfW}, \  W_{c(G)} \succeq 0
\eq
\noindent
\textbf{OPF-socp:}
\bq
\label{eq:bimOPF-socp}
\hspace{-0.285in}
\ \ 
\underset{W_G}{\text{min}} \  C(W_G) \ & \text{subject to} &   
		W_G \text{ satisfies } \eqref{eq:opfW},   \
		W_G(j,k) \succeq 0, \ (j,k)\in E
\eq

For BIM, we say that OPF-sdp \eqref{eq:bimOPF-sdp} is \emph{exact} if every 
optimal solution $W^\text{sdp}$ of OPF-sdp is psd rank-1; 
OPF-ch \eqref{eq:bimOPF-ch} is \emph{exact} if every optimal solution 
$W_{c(G)}^\text{ch}$ of OPF-ch is psd rank-1 (i.e., the principal submatrices 
$W_{c(G)}^\text{ch}(q)$ of $W_{c(G)}^\text{ch}$ are psd rank-1 for
all maximal cliques $q$ of the chordal extension $c(G)$ of graph $G$);  
OPF-socp \eqref{eq:bimOPF-socp} is \emph{exact} if every optimal solution $W_G^\text{socp}$ 
of OPF-socp is $2\times 2$ psd rank-1 and satisfies the cycle condition 
\eqref{eq:cyclecond.2}.
To recover an optimal solution $V^\text{opt}$ of OPF \eqref{eq:OPFbim} from 
$W^\text{sdp}$ or $W_{c(G)}^\text{ch}$ or $W_G^\text{socp}$,
see \cite[Section IV-D]{Low2014a}.

\subsection{Branch flow model}

The BFM adopts a directed graph $\tilde G$
and is defined by the following set of equations:
\begin{subequations}
\bq
\!\!\!\!
\sum_{k: j\rightarrow k} \!\!\! S_{jk}  & \!\! = \!\! &\!\!\!\!   
	\sum_{i: i\rightarrow j} \!\! \left( S_{ij} - z_{ij} |I_{ij}|^2 \right) + s_j, 
	\qquad  j \in N^+
\label{eq:bfm.3}
\\
\!\!\!\!
I_{jk}  & \!\! = \!\! &  y_{jk} (V_j - V_k),  \qquad\qquad\, j \rightarrow k \in \tilde{E}
\label{eq:bfm.1}
\\
\!\!\!\!
S_{jk}  & \!\! = \!\! & V_j \, I_{jk}^H, \qquad\qquad\qquad\ \ \, j \rightarrow k \in \tilde{E}
\label{eq:bfm.2}
\eq
\label{eq:bfm}
\end{subequations}
Denote the variables in BFM \eqref{eq:bfm} by 
$\tilde{x} := (S, I, V, s) \in \mathbb{C}^{2(m+n+1)}$.  
Note that the vectors $V$ and $s$ include $V_0$ (given) and $s_0$ respectively.
Recall from \cite{Low2014a} the variables
$x := (S, \ell, v, s) \in \mathbb R^{3(m+n+1)}$ 
that is related to $\tilde x$ by the mapping $x = h(\tilde x)$
with $\ell_{jk} := |I_{jk}|^2$ and $v_j := |V_j|^2$.
The operational constraints are:
\begin{subequations}
 \bq
\underline{v}_j \ \leq & v_j &  \leq\  \overline{v}_j, \qquad j \in N^+
  \label{eq:bfmopf.v}
 \\
 \underline{s}_j  \  \leq & s_j  & \leq \ \overline{s}_j, \qquad  j \in N^+
 \label{eq:bfmopf.s}
\eq
\label{eq:bfmopf}
\end{subequations}
We assume the cost function depends on $\tilde x$ only through
$x= h(\tilde x)$.   Then the problem in BFM is:
\\
\noindent
\textbf{OPF:}
\bq
\label{eq:OPFbfm}
\underset{\tilde x } {\text{min}}   \ \  C(x) 
& \text{subject to} &  \tilde x \text{ satisfies } \eqref{eq:bfm},  \eqref{eq:bfmopf}
\eq

For SOCP relaxation consider:
\begin{subequations}
\bq
\!\!\!\!\!
\sum_{k: j\rightarrow k} S_{jk} &\!\!\! = \!\!\! & \!\! \sum_{i: i\rightarrow j} \left( S_{ij} - z_{ij} \ell_{ij} \right)
		 + s_j, \qquad\qquad\ \ \  j \in N^+
\label{eq:mdf.1}
\\
\!\!\!\!\!
v_j - v_k &\!\!\!  = \!\!\! & 2\, \text{Re} \left(z_{jk}^H S_{jk} \right) - |z_{jk}|^2 \ell_{jk}, 
	\qquad j\rightarrow k \in \tilde E
\label{eq:mdf.2}
\\
\!\!\!\!\!
v_j \ell_{jk} &\!\!\!  \geq \!\!\! & |S_{jk}|^2, 
		\qquad\qquad\qquad\qquad\quad\ \,  j\rightarrow k \in \tilde E
\label{eq:mdf.3}
\eq
\label{eq:mdf}
\end{subequations}
%
%
%
We say that $x$ satisfies the \emph{cycle condition} if 
\bq
\!\!\!\!\!\!
\exists \theta \in \mathbb R^n & \text{such that} & B\theta  =  \beta(x)  \mod 2\pi
\label{eq:cyclecond.1}
\eq
where $B$ is the $m\times n$ reduced incidence matrix and, 
given $x := (S, \ell, v, s)$, $\beta_{jk}(x) := \angle (v_j - z_{jk}^H S_{jk})$ 
can be interpreted as the voltage angle difference across line 
$j\rightarrow k$ implied by $x$  (See \cite[Section V]{Low2014a}).
The SOCP relaxation in BFM is
\\
\noindent
\textbf{OPF-socp:}
\bq
\label{eq:bfmOPF-socp}
\underset{x } {\text{min}}   \ \  C(x) 
& \text{subject to} &  x \text{ satisfies } \eqref{eq:mdf},  \eqref{eq:bfmopf}
\eq

For BFM, OPF-socp \eqref{eq:bfmOPF-socp} in BFM
is \emph{exact} if every optimal solution $x^\text{socp}$ attains equality 
in \eqref{eq:mdf.3} and satisfies the cycle condition \eqref{eq:cyclecond.1}.
See \cite[Section V-A]{Low2014a} for how to recover an optimal solution 
$\tilde x^\text{opt}$ of OPF \eqref{eq:OPFbfm} from any optimal solution 
$x^\text{socp}$ of its SOCP relaxation.

\subsection{Exactness}
\label{subsec:exact}

The definition of exactness adopted in this paper is more stringent than needed.
Consider SOCP relaxation in BIM as an illustration (the same applies to
the other relaxations in BIM and BFM).
For any sets $A$ and $B$, we say that $A$ is \emph{equivalent to} $B$, 
denoted by $A\equiv B$, if there is a bijection between these two sets.
Let $\mathbb M(A)$ denote the set of minimizers 
when a certain function is minimized over $A$. 

Let $\mathbb V$ and $\mathbb W_G^+$ denote the feasible sets of 
OPF \eqref{eq:OPFbim} and OPF-socp \eqref{eq:bimOPF-socp}
respectively:
\bqn
\mathbb V &\!\!\!\! := \!\!\!\! & \{ V \in \mathbb C^{n+1} \ | \ V \text{ satisfies } \eqref{eq:bimopf} \}
\\
\mathbb W_G^+ & \!\!\!\! :=  \!\!\!\! & \{ W_G \ | \ W_G \text{ satisfies } \eqref{eq:opfW}, 
			 \ W_G(j,k) \succeq 0, \ (j,k)\in E \}
\eqn
Consider the following subset of $\mathbb W_G^+$:
\bqn
\mathbb W_G & := & \{ W_G \ | \ W_G \text{ satisfies } \eqref{eq:opfW}, \eqref{eq:cyclecond.2}, \ 
			 W_G(j,k) \succeq 0, 
			\, \text{rank}\, W_G(j,k) = 1, \  (j,k)\in E \}
\eqn
Our definition of exact SOCP relaxation is that  
$\mathbb M(\mathbb W_G^+) \subseteq \mathbb W_G$.
In particular, \emph{all} optimal solutions of OPF-socp must be
$2\times 2$ psd rank-1 and satisfy the cycle condition \eqref{eq:cyclecond.2}.
Since $\mathbb W_G\equiv \mathbb V$ (see \cite{Low2014a}), 
exactness requires that the set of optimal solutions of 
OPF-socp \eqref{eq:bimOPF-socp} be equivalent to that of 
OPF \eqref{eq:OPFbim}, i.e.,
$\mathbb M(\mathbb W_G^+) = \mathbb M(\mathbb W_G) \equiv \mathbb M(\mathbb V)$.

If 
$\mathbb M(\mathbb W_G^+) \supsetneq \mathbb M(\mathbb W_G)
 \equiv \mathbb M(\mathbb V)$
then OPF-socp \eqref{eq:bimOPF-socp} is not exact according to our definition.
Even in this case, however, every sufficient condition in this paper guarantees
that an optimal solution of OPF can be easily recovered from an optimal 
solution of the relaxation that is outside $\mathbb W_G$.
The difference between $\mathbb M(\mathbb W_G^+) = \mathbb M(\mathbb W_G)$
and $\mathbb M(\mathbb W_G^+) \supsetneq \mathbb M(\mathbb W_G)$ is often 
minor, depending on the objective function; 
see Remarks \ref{remark:exact1} and \ref{remark:exact2} and comments after
Theorems \ref{thm:bfmsocp.2} and \ref{thm:pfOPF}  in Section
\ref{sec:exact-radial}.  Hence we adopt the more stringent 
definition of exactness for simplicity.

\section{Radial networks}
\label{sec:exact-radial}

In this section we summarize the three types of sufficient conditions
listed in Table \ref{table:SummaryRadial}
for semidefinite relaxations of OPF to be exact for radial (tree) networks.   
These results are
important as most distribution systems are radial.   

For radial networks, if SOCP relaxation is exact then SDP and chordal 
relaxations are also exact (see \cite[Theorems 5, 9]{Low2014a}). 
We hence focus in this section on the exactness of OPF-socp in both BIM
and BFM.   Since the cycle conditions \eqref{eq:cyclecond.2} and
\eqref{eq:cyclecond.1} are vacuous for radial networks, 
OPF-socp \eqref{eq:bimOPF-socp} is exact if all of its optimal solutions 
are $2\times 2$ rank-1 and OPF-socp \eqref{eq:bfmOPF-socp} is
exact if all of its optimal solutions attain equalities in \eqref{eq:mdf.3}.
We will freely use either BIM or BFM in discussing these results.
To avoid triviality we make the following assumption throughout the paper:
\bee
\item[ ]  The voltage lower bounds satisfy $\underline{v}_j>0$, $j\in N^+$.
The original problems OPF \eqref{eq:OPFbim} and \eqref{eq:OPFbfm}
are feasible.
\eee

\subsection{Linear separability}
\label{subsec:linsep}

We will first present a general result 
on the exactness of the SOCP relaxation of general QCQP and then apply it to OPF.
This result is first formulated and proved using a duality argument in \cite{Bose-2012-QCQPt}, 
generalizing the result of \cite{Bose2011}.  It is proved using a simpler argument in \cite{Sojoudi2013}.

Fix an undirected graph $G = (N^+, E)$ where $N^+:= \{0, 1, \dots, n\}$
and $E\subseteq N^+ \times N^+$.  Fix Hermitian matrices $C_l\in \mathbb S^{n+1}$, 
$l = 0, \dots, L$, defined on $G$, i.e., $[C_l]_{jk} = 0$ if $(j,k)\not\in E$.  Consider
QCQP:
\begin{subequations}
\bq
\underset{x \in \mathbb{C}^{n+1}}{\text{min}}  & \!\!\!\! & \!\!\!\!   x^{{H}} C_0  x
\\
\text{subject to}
& \!\!\!\!  & \!\!\!\!   x^{{H}} C_l x  \leq b_l, \  \ l =1, \dots, L 
\eq
\label{eq:qcqp}
\end{subequations}
where $C_0, C_l \in \mathbb C^{(n+1)\times (n+1)}$, $b_l\in \mathbb R$, $l = 1, \dots, L$,
and its SOCP relaxation where the optimization variable ranges over Hermitian
partial matrices $W_G$:
\begin{subequations}
\bq
\underset{W_G}{\text{min}}  & \!\!\!\!\! & \!\!\!\!\! \text{tr } C_0 W_G
\\
\text{subject to}
& \!\!\!\!\!  &   \!\!\!\!\!   \text{tr } C_l W_G   \leq b_l, \ \  l = 1, \dots, L
\\
& \!\!\!\!\!  &   \!\!\!\!\!   W_G(j,k) \succeq 0, \ \  (j,k)\in E
\eq
\label{eq:socp}
\end{subequations}
%
The following result is proved in \cite{Bose-2012-QCQPt, Sojoudi2013}.  It can be regarded
as an extension of \cite{KimKojima2003} on the SOCP relaxation of QCQP from the
real domain to the complex domain. 
Consider: \footnote{All 
angles should be interpreted as ``mod $2\pi$'', i.e., projected onto $(-\pi, \pi]$.}
\begin{enumerate}
\item [A1:] The cost matrix $C_0$ is positive definite.
\item [A2:] For each link $(j,k) \in E$ there exists an $\alpha_{jk}$ such that 
	$\angle \left[ C_l \right]_{jk}  \in  [\alpha_{ij}, \alpha_{ij} + \pi]$ for all $l = 0, \dots, L$.
\end{enumerate}
Let $C^\text{opt}$ and $C^\text{socp}$ denote the optimal values of QCQP \eqref{eq:qcqp} and
SOCP \eqref{eq:socp} respectively.
\begin{theorem} 
\label{thm:exact4qcqp}
Suppose $G$ is a tree and A2 holds.  Then $C^\text{opt} = C^\text{socp}$
and an optimal solution of QCQP \eqref{eq:qcqp} can be recovered from
every optimal solution of SOCP \eqref{eq:socp}.
\end{theorem}
\vspace{0.1in}

\begin{remark}
\label{remark:exact1}
The proof of Theorem \ref{thm:exact4qcqp} prescribes a simple procedure to recover an
optimal solution of QCQP \eqref{eq:qcqp} from any optimal solution of its SOCP 
relaxation \eqref{eq:socp}.   The construction does not need the optimal solution of
SOCP \eqref{eq:socp} to be $2\times 2$ rank-1.  Hence 
the SOCP relaxation may not be exact according to our definition of exactness, 
i.e., some optimal
solutions of \eqref{eq:socp} may be $2\times 2$ psd but not $2\times 2$ rank-1.
If the objective function is strictly convex however then the optimal solution sets
of QCQP \eqref{eq:qcqp} and SOCP \eqref{eq:socp} are indeed equivalent.
\end{remark}
\begin{corollary} 
\label{coro:exact4qcqp}
Suppose $G$ is a tree and A1--A2 hold.  Then SOCP \eqref{eq:socp} is exact.
\end{corollary}
\vspace{0.1in}

We now apply Theorem \ref{thm:exact4qcqp} to our OPF problem.
Recall that OPF \eqref{eq:OPFbim} in BIM can be written as a standard 
form QCQP \cite{Bose-2012-QCQPt}:
\begin{subequations}
\bq
\!\!\!\!\!\!    \!\!\!\!\!\!\!
\min_{x\in \mathbb C^n} & \!\!\!\! & \!\!\!\!  V^H C_0 V
\nonumber
\\
\!\!\!\!\!\!       \!\!\!\!\!\!\!
\text{s.t.}  & \!\!\!\! & \!\!\!\!
 	V^H \Phi_j V \leq \overline{p}_j, \  	V^H (-\Phi_j) V \leq -\underline{p}_j
\label{eq:bimOPF.b2}
\\
\!\!\!\!\!\!       \!\!\!\!\!\!\!   & \!\!\!\!  & \!\!\!\!
 		V^H \Psi_j V \leq \overline{q}_j, \ V^H (-\Psi_j) V \leq -\underline{q}_j
\label{eq:bimOPF.b3}
\\
\!\!\!\!\!\!       \!\!\!\!\!\!\!   & \!\!\!\!  & \!\!\!\!
		V^H J_j V \leq -\overline{v}_j, \ V^H (- J_j) V \leq -\underline{v}_j
\nonumber
\eq
\label{eq:OPFbim.2}
\end{subequations}
for some Hermitian matrices $C_0, \Phi_j, \Psi_j, J_j$ where $j\in N^+$.
A2 depends only on the off-diagonal entries of $C_0$, $\Phi_j$, $\Psi_j$ ($J_j$ are
diagonal matrices).    It implies a simple pattern
on the power injection constraints \eqref{eq:bimOPF.b2}--\eqref{eq:bimOPF.b3}.
Let $y_{jk} = g_{jk} - \ii b_{jk}$ with $g_{jk}>0, b_{jk}>0$.  
Then we have (from  \cite{Bose-2012-QCQPt}):   
\bqn
[\Phi_k]_{ij} &\!\!\!\! = \!\!\!\! &  \begin{cases}
		\frac{1}{2} Y_{ij}  \, =\, - \frac{1}{2} (g_{ij} - \ii b_{ij}) &  \text{ if }k = i\\
		\frac{1}{2} {Y}_{ij}^H  \,=\, - \frac{1}{2} (g_{ij} + \ii b_{ij}) &  \text{ if }k = j \\
		0 & \text{ if } k\not\in \{ i, j \}
		\end{cases}
\eqn
\bqn
[\Psi_k]_{ij} &\!\!\!\!  =  \!\!\!\!  &  \begin{cases}
		\frac{-1}{2\ii} Y_{ij}  \,=\, - \frac{1}{2}(b_{ij} + \ii g_{ij})  &  \text{if }k = i \\
		\frac{1}{2\ii} {Y}_{ij}^H \,=\, - \frac{1}{2} (b_{ij} - \ii g_{ij})  &  \text{if }k = j \\
		0 & \text{if } k\not\in \{ i, j\}
		\end{cases}
\eqn
Hence for each line $(j, k)\in E$ the relevant angles for A2 are those of $[C_0]_{jk}$
and 
\bqn
\left[\Phi_j\right]_{jk}  & = & - \frac{1}{2} (g_{jk} - \ii b_{jk})
\eqn
\bqn
[\Phi_k]_{jk} & = & - \frac{1}{2} (g_{jk} + \ii b_{jk}) 
\eqn
\bqn
[\Psi_j]_{jk}  & = &  - \frac{1}{2} (b_{jk} + \ii g_{jk})
\eqn
\bqn
[\Psi_k]_{jk} & = &  - \frac{1}{2} (b_{jk} - \ii g_{jk}) 
\eqn
as well as the angles of  $-[\Phi_j]_{jk}, -[\Phi_k]_{jk}$ and $-[\Psi_j]_{jk}, -[\Psi_k]_{jk}$.
These quantities are shown in Figure \ref{fig:SeparatingLine} with their magnitudes 
{normalized } to a common value and explained in the caption of the figure.
\begin{figure}[htbp]
\centering
\includegraphics[width=0.55\textwidth]{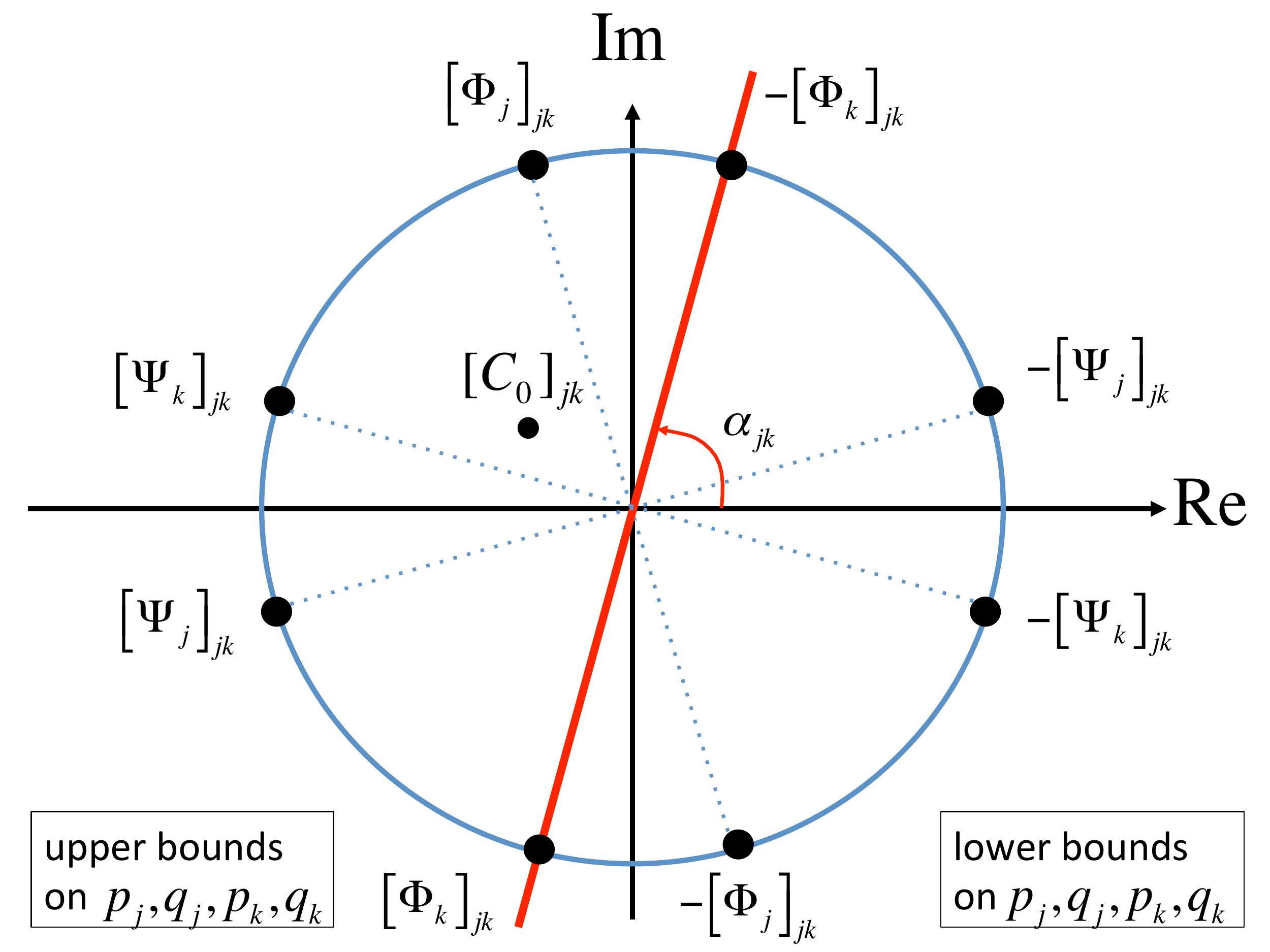}
\caption{Condition A2' on a line $(j,k)\in E$.  
The quantities $([\Phi_j]_{jk}, [\Phi_k]_{jk}, [\Psi_j]_{jk}, [\Psi_k]_{jk})$ on the left-half 
plane correspond to finite upper bounds on $(p_j, p_k, q_j, q_k)$ 
in \eqref{eq:bimOPF.b2}--\eqref{eq:bimOPF.b3};  
$(-[\Phi_j]_{jk}, -[\Phi_k]_{jk}, -[\Psi_j]_{jk}, -[\Psi_k]_{jk})$ on the right-half plane 
correspond to finite lower bounds on $(p_j, p_k, q_j, q_k)$.
A2' is satisfied if there is a line through the origin, specified by the angle $\alpha_{jk}$,
so that the quantities corresponding
to \emph{finite} upper or lower bounds on $(p_j, p_k, q_j, q_k)$
lie on one side of the line, possibly on the line itself.   The load over-satisfaction condition in 
\cite{Bose2011, Sojoudi2012PES} 
corresponds to the Im-axis that excludes all quantities
on the right-half plane.
The sufficient condition in \cite[Theorem 2]{Zhang2013}
corresponds to the red line in the figure that allows a finite lower 
bound on the real power at one end of the line, i.e., $p_j$ or $p_k$ but not both,
and no finite lower bounds on reactive powers $q_j$ and $q_k$.
}
\label{fig:SeparatingLine}
\end{figure}

Condition A2 applied to OPF \eqref{eq:OPFbim.2} takes the following form
(see Figure \ref{fig:SeparatingLine}):
\bee
\item [A2':] For each link $(j,k) \in E$ there is a line in the complex plane through the origin 
	such that $\left[ C_0 \right]_{jk}$ as well as those $\pm [\Phi_i]_{jk}$ and $\pm [\Psi_i]_{jk}$
	corresponding to \emph{finite} lower or upper bounds on $(p_i, q_i)$, for $i=j, k$, 
	are all on one side of the line, possibly on the line itself.
\eee
Let $C^\text{opt}$ and $C^\text{socp}$ denote the optimal values of OPF 
\eqref{eq:OPFbim} and OPF-socp \eqref{eq:bimOPF-socp} respectively.
\begin{corollary}
\label{coro: OPFSeparatingLine}
Suppose $G$ is a tree and A2' holds.  
\bee
\item  $C^\text{opt} = C^\text{socp}$.  Moreover 
	 an optimal solution $V^\text{opt}$ of OPF \eqref{eq:OPFbim} can be recovered from
	every optimal solution $W_G^\text{socp}$ of OPF-socp \eqref{eq:bimOPF-socp}.
\item If, in addition, A1 holds then OPF-socp \eqref{eq:bimOPF-socp} is exact.
\eee
\end{corollary}
\vspace{0.1in}

It is clear from Figure \ref{fig:SeparatingLine} that condition A2' cannot be
satisfied if there is a line where both the real and reactive power injections
at both ends are both lower and upper bounded (8 combinations as shown
in the figure).
A2' requires that some of them be unconstrained even though in practice they are
always bounded.  It should be interpreted as requiring that the optimal solutions obtained 
by ignoring these bounds  turn out to satisfy these bounds.  This is generally
different from solving the optimization \emph{with} these constraints but requiring that they be
inactive (strictly within these bounds) at optimality, unless the cost function is strictly convex.
The result proved in \cite{Bose-2012-QCQPt} also includes constraints on real branch
power flows and line losses.
Corollary \ref{coro: OPFSeparatingLine}
 includes several sufficient conditions in the literature for exact relaxation as special cases; see
 the caption of Figure \ref{fig:SeparatingLine}.

Corollary \ref{coro: OPFSeparatingLine} also implies a result first proved in 
\cite{Farivar2011-VAR-SGC}, using a different technique, that SOCP relaxation is 
exact in BFM for radial networks when there are no lower bounds on power 
injections $s_j$.  The argument in \cite{Farivar2011-VAR-SGC} is generalized in 
\cite[Part I]{Farivar-2013-BFM-TPS} to allow convex objective 
functions, shunt elements, and line limits in terms of upper bounds on $\ell_{jk}$.  
Assume 
\bee
\item[A3: ] The cost function $C(x)$ is convex, strictly increasing in $\ell$,
	nondecreasing in $s = (p, q)$, and independent of branch flows $S = (P, Q)$.
\item[A4: ] For $j\in N^+$,  $\underline{s}_j = -\infty - \ii \infty$.
\eee
Popular cost functions in the literature include active power loss over the network or
active power generations, both of which satisfy A3.
The next result is proved in \cite{Farivar2011-VAR-SGC, Farivar-2013-BFM-TPS}.
\begin{theorem} 
\label{thm:bfmsocp.1}
Suppose $\tilde G$ is a tree and A3--A4 hold.  Then OPF-socp \eqref{eq:bfmOPF-socp}
is exact.
\end{theorem}

\begin{remark}
\label{remark:exact2}
If the cost function $C(x)$ in A3 is only nondecreasing, rather than
strictly increasing, in $\ell$, then A3--A4 still guarantee that all optimal solutions of 
OPF \eqref{eq:OPFbfm} are (i.e., can be mapped to) optimal solutions of
OPF-socp \eqref{eq:bfmOPF-socp}, but OPF-socp 
may have an optimal solution that maintains strict inequalities in \eqref{eq:mdf.3}
and hence is infeasible for OPF.   Even though OPF-socp
is not exact in this case, the proof of Theorem \ref{thm:bfmsocp.1}
constructs from it an optimal solution of OPF (See the arXiv version of this paper). 
\end{remark}

\subsection{Voltage upper bounds}

While type A conditions (A2' and A4 in the last subsection) require that some
power injection constraints not be binding, type B conditions require non-binding
voltage upper bounds.   They are proved in 
\cite{Li-2012-BFMt, Gan-2012-BFMt, Gan-2013-BFMt-CDC, Gan-2014-BFMt-TAC} 
using BFM.  

For radial networks the model originally proposed in \cite{Baran1989a, Baran1989b},
which is \eqref{eq:mdf} with the inequalities in \eqref{eq:mdf.3} replaced
by equalities, is exact.  This is because 
the cycle condition \eqref{eq:cyclecond.1} is always satisfied
as the reduced incidence matrix $B$ is $n\times n$ and invertible for radial networks.
Following \cite{Gan-2014-BFMt-TAC} we adopt the graph orientation where every 
link points \emph{towards} node 0.
Then \eqref{eq:mdf} for a radial network reduces to:
\begin{subequations}
\bq
\!\!\!\!\!\!\!\!\!
{S}_{jk} & \!\!\!\! = \!\!\!\! &\!\! \sum_{i: i\rightarrow j} \!\! \left( {S}_{ij} - z_{ij} \ell_{ij} \right) + s_j, 
			\qquad\qquad\ \ \ \,  j \in N^+
\label{eq:gdf.1}
\\
\!\!\!\!\!\!\!\!\!
v_j - v_k & \!\!\!\! = \!\!\!\! & 2\, \text{Re} \left(z_{jk}^H S_{jk} \right)\! - |z_{jk}|^2 \ell_{jk}, 
			\qquad  j\rightarrow k \in \tilde E
\label{eq:gdf.2}
\\
\!\!\!\!\!\!\!\!\!
v_j \ell_{jk}   & \!\!\!\!  \geq \!\!\!\! & |S_{jk}|^2, 
			\qquad\qquad\qquad\qquad\quad\   j\rightarrow k \in \tilde E
\label{eq:gdf.3}
\eq
\label{eq:gdf}
\end{subequations}
where $v_0$ is given and  in \eqref{eq:gdf.1}, $k$  denotes the node on the unique path from node $j$ 
to node 0.   The boundary condition is: $S_{jk} := 0$ when $j=0$ in \eqref{eq:gdf.1} and 
$S_{ij} =0$, $\ell_{ij} = 0$ when $j$ is a leaf node.\footnote{A node $j\in N$ is a \emph{leaf node} if 
there is no $i$ such that $i\rightarrow j \in \tilde E$.}

As before the voltage magnitudes must satisfy:
\begin{subequations}
\bq
\underline{v}_j \ \leq & v_j &  \leq\  \overline{v}_j, \qquad j \in N
\label{eq:boxv}
\eq
We allow more general constraints on the power injections: 
for $j\in N$, $s_j$ can be in an \emph{arbitrary} set $\mathbb S_j$ that is bounded above:
\bq
s_j \ \in \ \mathbb S_j & \subseteq & \{ s_j \in \mathbb C \, |\, s_j \leq \overline{s}_j \}, 
\qquad  j\in N
\label{eq:constraints}
\eq
\label{eq:cnstrs}
\end{subequations}
for some given $\overline s_j$, $j\in N$.\footnote{We assume here that $s_0$ is 
unconstrained, and since $V_0:= 1\angle 0^\circ$ pu, the constraints \eqref{eq:cnstrs}
involve only $j$ in $N$, not $N^+$.}
Then the SOCP relaxation is
\\
\noindent
\textbf{OPF-socp:}
\bq
\underset{ x } {\text{min}}   \ \  C( x) & \text{subject to} &  
	\eqref{eq:gdf}, \eqref{eq:cnstrs}
\qquad\quad
\label{eq:SOCP-radialBFM}
\eq
As defined in Section \ref{subsec:exact},
OPF-socp \eqref{eq:SOCP-radialBFM} is  {exact} if 
every optimal solution $x^\text{socp}$ attains equality in \eqref{eq:gdf.3}.
In that case an optimal solution of BFM \eqref{eq:OPFbfm}
can be uniquely recovered from $x^\text{socp}$.

We make two comments on the constraint sets $\mathbb S_j$ in \eqref{eq:constraints}.  
First $\mathbb S_j$ need not be
convex nor even connected for convex relaxations to be exact.
They (only) need to be convex to be efficiently computable.   
Second such a general constraint on $s$ is useful in many applications.  It includes the case
where $s_j$ are subject to simple box constraints, but also allows constraints of the form
$|s_j|^2 \leq a$, $|\angle s_j| \leq \phi_j$ that is useful for volt/var control \cite{Turitsyn11},
or $q_j \in \{0, a\}$ for capacitor configurations.

\noindent
\emph{Geometric insight.}
To motivate our sufficient condition, we first explain 
a simple geometric intuition using a two-bus network on why relaxing voltage upper bounds 
guarantees exact SOCP relaxation.
Consider bus 0 and bus 1 connected by a line with impedance $z:= r+\ii x$.  Suppose
without loss of generality that $v_0=1$ pu.   Eliminating $S_{01} = s_0$
from \eqref{eq:gdf}, the model reduces to (dropping the subscript on $\ell_{01}$):
\bq
p_0 - r \ell \ = \ -p_1, \ \ \
q_0 - x \ell \ = \ -q_1, \ \ \
p_0^2 + q_0^2 \ = \ \ell
\label{eq:p0q0ell}
\eq
and
\bq
v_1 - v_0 & = & 2 (rp_0 + xq_0) - |z|^2 \ell
\label{eq:v}
\eq
Suppose $s_1$ is given (e.g., a constant power load).  Then the variables are
$(\ell, v_1, p_0, q_0)$ and the feasible set consists of solutions of \eqref{eq:p0q0ell}
and \eqref{eq:v} subject to additional constraints on $(\ell, v_1, p_0, q_0)$.   
\begin{figure}[htbp]
\centering
	\includegraphics[width=0.45\textwidth]{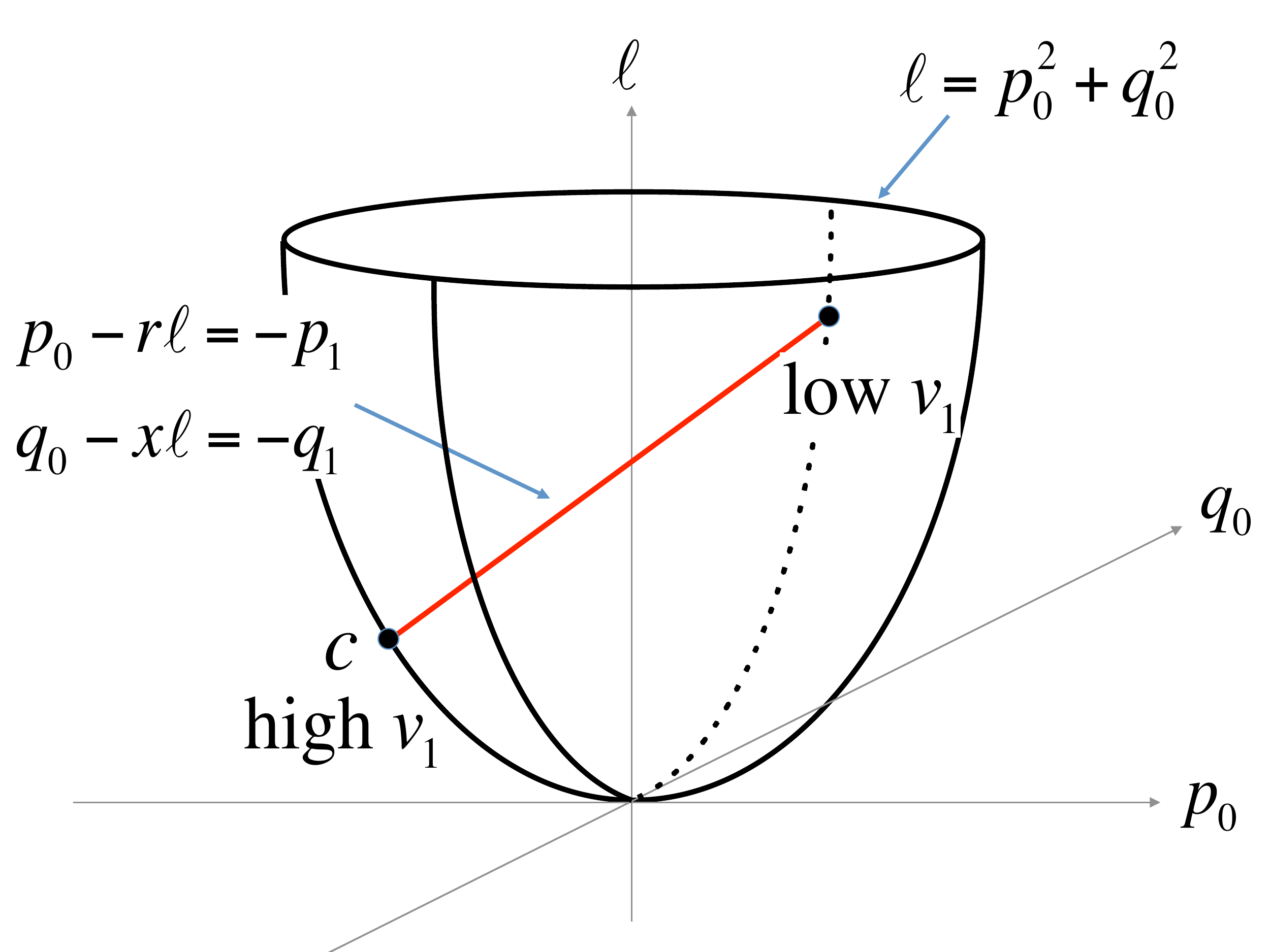}
\caption{Feasible set of OPF for a two-bus network without any constraint. 
	It consists of the (two) points of intersection of the line with 
	the convex surface (without the interior), and hence is nonconvex.  
	SOCP relaxation includes the interior of the convex surface and enlarges the feasible
	set to the line segment joining these two points.  If the cost function $C$ is increasing
	in $\ell$ or $(p_0, q_0)$ then the optimal point over the SOCP feasible set 
	(line segment) is the
	lower feasible point $c$, and hence the relaxation is exact.  No constraint
	on $\ell$ or $(p_0, q_0)$ will destroy exactness as long as the resulting feasible set
	contains $c$.  
}
\label{fig:2busNoC}
\end{figure}
The case without any constraint is instructive and shown in Figure \ref{fig:2busNoC}.
The point $c$ in the figure corresponds to a power flow solution with a
large $v_1$ (normal operation) whereas the other intersection corresponds to a 
solution with a small $v_1$ (fault condition).
As explained in the caption, SOCP relaxation is exact if there is no voltage constraint
and as long as constraints on $(\ell, p_0, q_0)$ does not remove the high-voltage
(normal) power flow solution $c$.  Only when the system is stressed to a point where
the high-voltage solution becomes infeasible will relaxation lose exactness.  This
agrees with conventional wisdom that power systems under normal operations are
well behaved.

Consider now the voltage constraint $\underline v_1 \leq v_1 \leq \overline v_1$.
Substituting \eqref{eq:p0q0ell} into \eqref{eq:v} we obtain
\bqn
v_1 & = &  (1 + rp_1 + xq_1) - |z|^2 \ell
\eqn
translating the constraint on $v_1$ into a box constraint on $\ell$:
\bqn
\frac{1}{|z|^2} \left( rp_1 + xq_1 + 1-\overline v_1 \right) \ \leq \ \ell \ \leq \
\frac{1}{|z|^2} \left( rp_1 + xq_1 + 1-\underline v_1 \right)
\eqn
Figure \ref{fig:2busNoC} shows that the lower bound $\underline v_1$
(corresponding to an upper bound on $\ell$) does not affect the exactness of SOCP 
relaxation.
The effect of upper bound $\overline v_1$ (corresponding to a lower bound on $\ell$)
is illustrated in Figure \ref{fig:2busC}.   As explained in the caption of the figure
SOCP relaxation 
is exact if the upper bound $\overline v_1$ does not exclude the high-voltage
power flow solution $c$ and is not exact otherwise.
\begin{figure}[htbp]
\centering
\subfigure [Voltage constraint not binding] {
	\includegraphics[width=0.4 \textwidth]{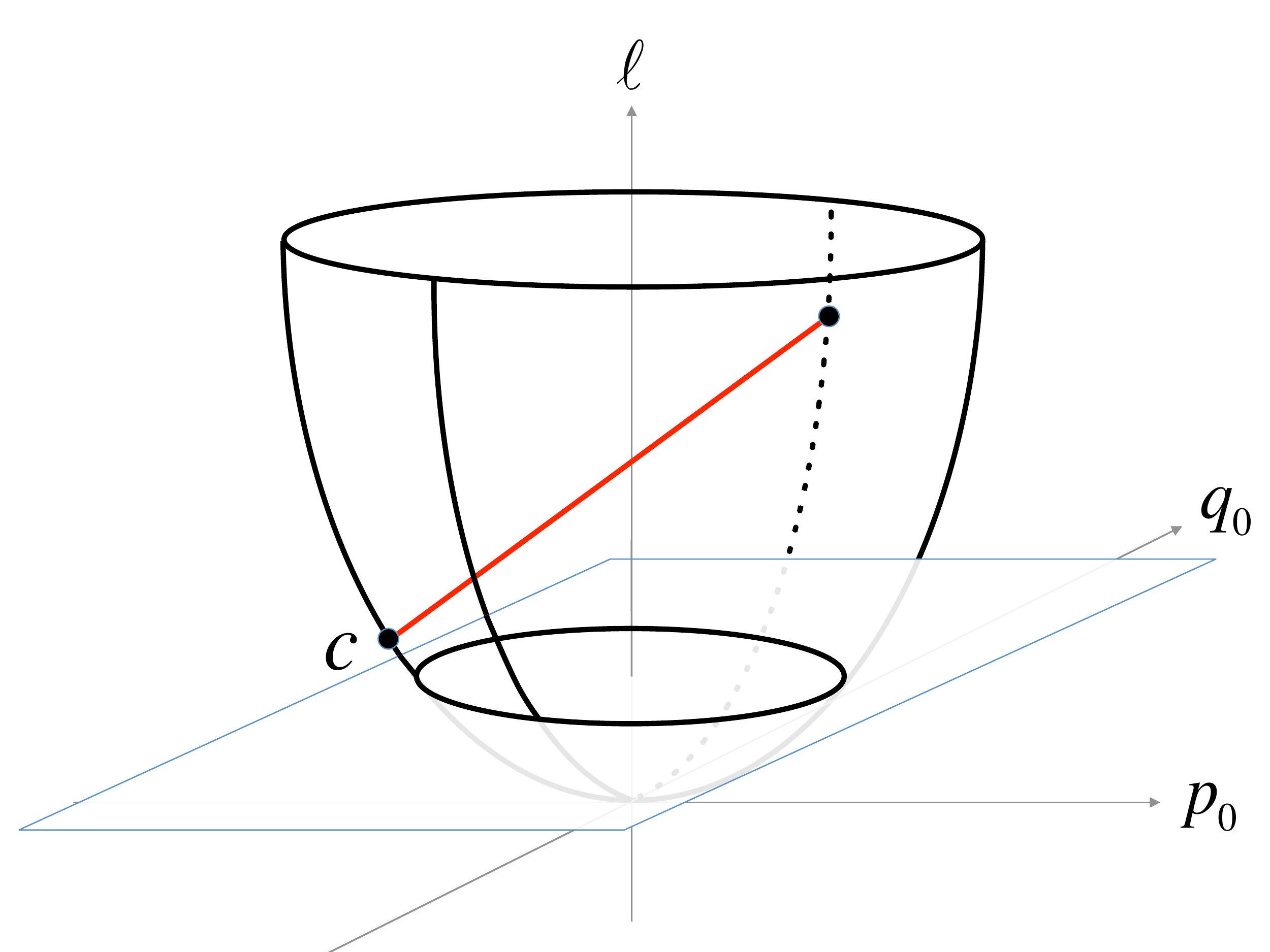}
	}
\subfigure [Voltage constraint binding] {
	\includegraphics[width=0.4 \textwidth]{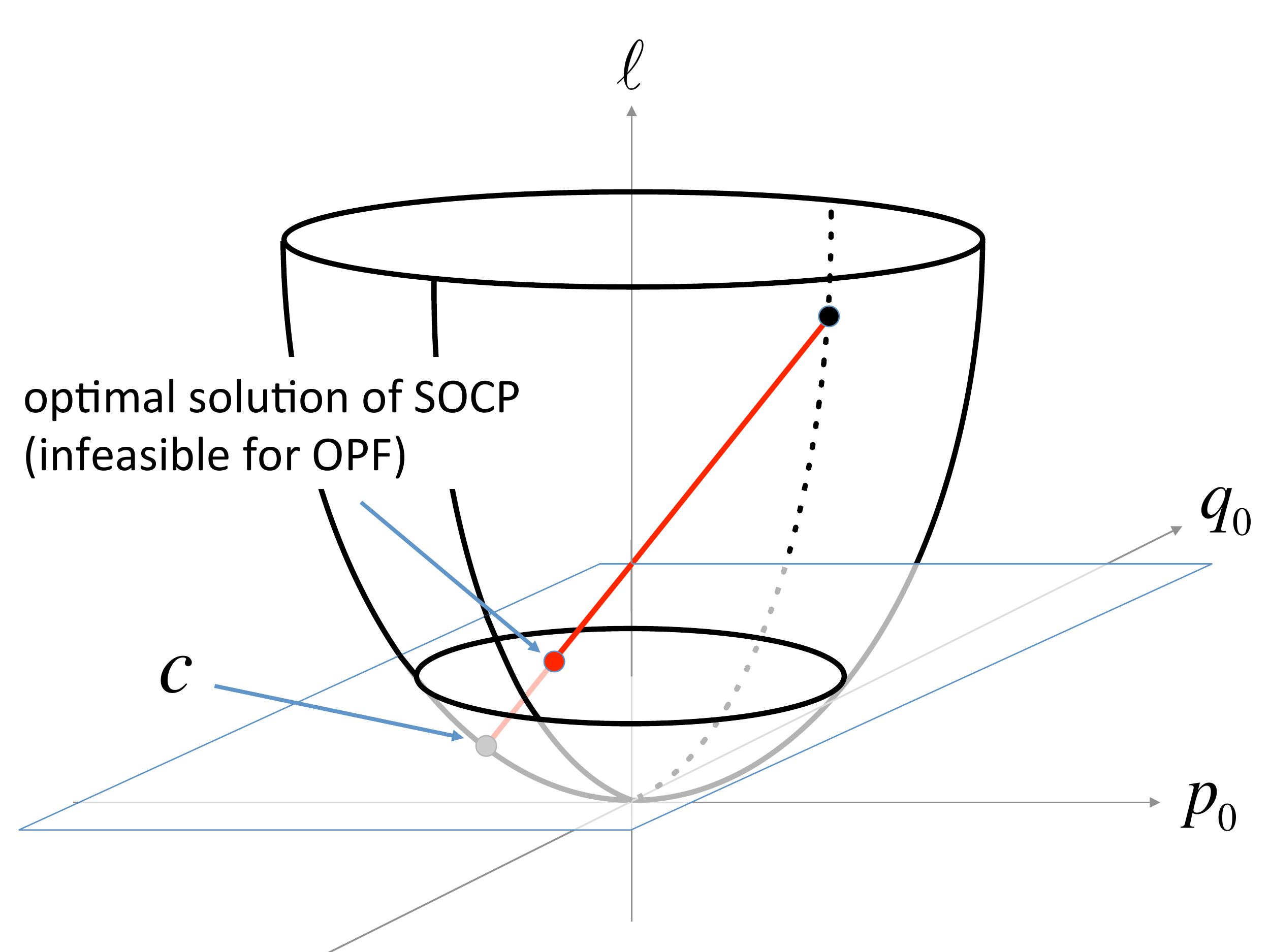}
	}
\caption{Impact of voltage upper bound $\overline v_1$ on exactness. 
(a) When $\overline v_1$ (corresponding to a lower bound on $\ell$) is
	not binding, the power flow solution $c$ is in the feasible set of SOCP and
	hence the relaxation is exact.
(b) When $\overline v_1$ excludes $c$ from the feasible set of SOCP, the
	optimal solution is infeasible for OPF and the relaxation is not exact.
}
\label{fig:2busC}
\end{figure}
\vspace{0.2in}

To state the sufficient condition for a general radial network, 
recall from \cite[Section VI]{Low2014a} the linear approximation of BFM
for radial networks obtained by setting $\ell_{jk}=0$ in \eqref{eq:gdf}: 
for each $s$
\begin{subequations}
\bq
S^{\text{lin}}_{jk}(s) & = & \sum_{i \in \mathbb T_j} s_i
\\
v^{\text{lin}}_{j}(s) & = & v_0 + 2 \sum_{(i,k) \in \mathbb P_j} 
					\text{Re} \left( z_{ik}^H  S_{ik}^{\text{lin}}(s) \right)
\eq
\label{eq:gldf}
\end{subequations}
where $\mathbb T_j$ denotes the subtree at node $j$, including $j$, and $\mathbb P_j$ denotes the
set of links on the unique path from $j$ to $0$.
The key property we will use is, from \cite[Lemma 13 and Remark 9]{Low2014a}:
\bq
S_{jk} \ \leq \ S_{jk}^\text{lin}(s) & \text{ and } & v_j \ \leq \ v_j^\text{lin}(s)
\label{eq:linbounds}
\eq
Define the $2\times 2$ matrix function
\bq
A_{jk}(S_{jk}, v_j) & := & I - \frac{2}{v_j}\, z_{jk} 
			\left( S_{jk} \right)^T
\label{eq:defA}
\eq
where $z_{jk} := [ r_{jk} \ \, x_{jk} ]^T$ is the line impedance  and $S_{jk} := [ P_{jk}\ \, Q_{jk} ]^T$
is the branch power flows,
both taken as 2-dimensional real vectors so that $z_{jk} \left( S_{jk} \right)^T$ is a $2\times 2$
matrix with rank less or equal to 1.
The matrices $A_{jk}(S_{jk}, v_j)$ describe how changes in the real and reactive power flows 
propagate towards the root node 0; see comments below. 
Evaluate the Jacobian matrix $A_{jk}(S_{jk}, v_j)$ at the boundary values:
\bq
\underline A_{jk} & := &  A_{jk} \left( \left[ S_{jk}^\text{lin}(\overline s) \right]^+\!\!, \  \underline v_j \right) 
\ \, = \ \, I - \frac{2}{\underline v_j}\, z_{jk} 
			\left( \left[ S^{\text{lin}}_{jk}(\overline s) \right]^+ \right)^T
\label{eq:defAunder}
\eq
Here $\left(\left[ a \right]^+ \right)^T$ is the row vector $\left[ [a_1]^+ \  [a_2]^+ \right]$ 
with $[a_j]^+ := \max\{0, a_j\}$.

For a radial network, for $j\neq 0$, every link $j\rightarrow k$ identifies a unique node $k$ and therefore, to simplify
notation, we refer to a link interchangeably by $(j,k)$ or $j$ and use 
$A_j$, $\underline A_j$, $z_j$ etc. in place of $A_{jk}$, $\underline A_{jk}$, $z_{jk}$ etc.
respectively.

Assume
\bee
\item[B1: ] The cost function is $C(x) := \sum_{j=0}^n C_j \left( \text{Re}\, s_j \right)$ with $C_0$ 
	strictly increasing.  There is no constraint on $s_0$.
\item[B2: ] The set $\mathbb S_j$ of injections satisfies $v_j^{\text{lin}}(s) \leq \overline{v}_j$, $j\in N$,
	where $v_j^{\text{lin}}(s)$ is given by \eqref{eq:gldf}.
\item[B3: ] For each leaf node $j\in N$ let the unique path from $j$ to 0 have $k$ links and be denoted by 
		$\mathbb P_j := ( (i_k, i_{k-1}),  \dots, (i_1, i_0) )$ with $i_k = j$ and $i_0 = 0$.
		Then $\underline A_{i_t} \cdots \underline A_{i_{t'}} \, z_{i_{t'+1}} > 0$
		for all $1 \leq t \leq t' < k$.
\eee
The following result is proved in \cite{Gan-2014-BFMt-TAC}.
\begin{theorem} 
\label{thm:bfmsocp.2}
Suppose $\tilde G$ is a tree and B1--B3 hold.  Then OPF-socp \eqref{eq:SOCP-radialBFM} is exact.
\end{theorem}
\vspace{0.1in}

We now comment on the conditions B1--B3.
B1 requires that the cost functions $C_j$ depend only on the injections $s_j$.  For instance,
if $C_j \left( \text{Re}\, s_j \right) = p_j$, then the cost is total active power loss over the network.
It also requires that $C_0$ be strictly increasing but makes no assumption on $C_j, j>0$.
Common cost functions such as line loss or generation cost usually satisfy B1.
If $C_0$ is only nondecreasing, rather than strictly increasing, in $p_0$ then
 B1--B3 still guarantee that all optimal solutions of OPF \eqref{eq:OPFbfm}
are (effectively) optimal for OPF-socp \eqref{eq:SOCP-radialBFM}, but OPF-socp may
not be exact, i.e., it may have an optimal solution that maintains strict inequalities in 
\eqref{eq:gdf.3}.   In this case the proof of Theorem \ref{thm:bfmsocp.2} can be used
to recursively construct from it another optimal solution that attains equalities in \eqref{eq:gdf.3}.

 B2 is affine in the injections $s := (p, q)$.
It enforces the upper bounds on voltage magnitudes because of \eqref{eq:linbounds}.

 B3 is a technical assumption and has a simple interpretation: the
branch power flow $S_{jk}$ on all branches should move in the same direction.
Specifically, given a marginal change in the complex power on line $j\rightarrow k$, 
the $2\times 2$ matrix $\underline A_{jk}$ is (a lower bound on) the Jacobian and 
describes the effect of this marginal change on the complex power on the line
immediately upstream from line $j\rightarrow k$.
  The product of $\underline A_{i}$ in
B3 propagates this effect upstream towards the root.
 B3 requires that a small change, positive or negative, in the power flow on a line 
affects \emph{all} upstream branch powers in the same direction.
This seems to hold with a significant margin
 in practice; see \cite{Gan-2014-BFMt-TAC} for examples from real systems.

Theorem \ref{thm:bfmsocp.2} unifies and generalizes some earlier results in 
\cite{Li-2012-BFMt, Gan-2012-BFMt, Gan-2013-BFMt-CDC}.
The sufficient conditions in these papers have the following simple and
practical interpretation: OPF-socp is exact provided either
\bi
\item there are no reverse power flows in the network, or
\item if the $r/x$ ratios on all lines are equal, or 
\item if the $r/x$ ratios increase in the downstream direction from the 
	substation (node 0) to the leaves then there are no reverse real power flows, or
\item if the $r/x$ ratios decrease in the downstream direction then
	there are no reverse reactive power flows.
\ei

The exactness of SOCP relaxation does not require convexity, i.e., 
the cost $C(x) = \sum_{j=0}^n C_j (\text{Re} s_j)$ need not be a convex function
and the injection regions $\mathbb S_j$ need not be convex sets.
Convexity allows polynomial-time computation. 
Moreover when it is convex the exactness of SOCP relaxation also implies the
uniqueness of the optimal solution, as the following result from \cite{Gan-2014-BFMt-TAC}
shows.
\begin{theorem}
\label{thm:uniqueBFM}
Suppose $\tilde G$ is a tree.
Suppose the costs $C_j$, $j=0, \dots, n$, are convex functions and the injection
regions $\mathbb S_j$, $j=1, \dots, n$, are convex sets.  If the relaxation OPF-socp
\eqref{eq:SOCP-radialBFM} is exact then its optimal solution is unique.
\end{theorem}
\vspace{0.1in}

Consider the model of \cite{Baran1989a} for radial networks,
which is \eqref{eq:gdf} with the inequalities in \eqref{eq:gdf.3} replaced by 
equalities.   Let $\mathbb X$ denote an equivalent feasible set of OPF,\footnote{There
is a bijection between $\mathbb X$ and the feasible set of OPF \eqref{eq:OPFbfm} 
 (when \eqref{eq:constraints} are placed by \eqref{eq:bfmopf.s})
\cite{Farivar-2013-BFM-TPS, Low2014a}.
}
i.e., those $x\in \mathbb R^{3(m+n+1)}$ that satisfy \eqref{eq:gdf}, \eqref{eq:cnstrs}
and attain equalities in \eqref{eq:gdf.3}.  
The proof of Theorem \ref{thm:uniqueBFM} reveals that, for radial networks, 
the feasible set $\mathbb X$ has a ``hollow'' interior.
\begin{corollary}
\label{coro:nonconvex}
If $\hat{x}$ and $\tilde x$ are distinct solutions in $\mathbb X$
then no convex combination of $\hat{x}$ and $\tilde x$ can be in $\mathbb X$.
In particular $\mathbb X$ is nonconvex.
\end{corollary}
\vspace{0.1in}

This property is illustrated vividly in several numerical examples 
for mesh networks in 
\cite{Hiskens-1995-Nonlinearity, Hiskens-2001-OPFboundary-TPS, LesieutreHiskens-2005-OPF-TPS, Hill-2008-OPFboundary-TPS}.

\subsection{Angle differences}

The sufficient conditions in 
\cite{Zhang2013, LavaeiTseZhang2012, LamZhang2012} require that the 
voltage angle difference across each line be small.
We explain the intuition using a result in \cite{LavaeiTseZhang2012} for
an OPF problem where $|V_j|$ are fixed for all $j\in N^+$ and reactive powers
are ignored.   
Under these assumptions, as long as the voltage angle difference 
is small, the power flow solutions form a locally convex surface that is the Pareto
front of its relaxation.  This implies that the relaxation is exact.
This geometric picture is apparent in earlier work on the geometry 
of power flow solutions, see e.g. \cite{Hiskens-1995-Nonlinearity}, and underlies the
intuition that the dynamics of a power system is usually benign until it is pushed 
towards the boundary of its stability region.  The geometric insight in
Figures \ref{fig:2busNoC} and \ref{fig:2busC} for BFM and later in this subsection 
for BIM says that, when it is far away from the boundary, 
the local convexity structure also facilitates exact relaxation. 
Reactive power is considered in 
\cite[Theorem 1]{LamZhang2012} with fixed $|V_j|$ where, with an additional constraint
on the lower bounds of reactive power injections 
that ensure these lower bounds are not tight, it is proved that if 
the original OPF problem is feasible then its SDP relaxation is exact.
The case of variable $|V_j|$ without reactive power is considered in
\cite[Theorem 7]{LavaeiTseZhang2012} but the simple geometric structure
is lost.

Recall that $y_{jk} = g_{jk} - \ii b_{jk}$ with $g_{jk}>0, b_{jk}>0$.  
Let $V_j = |V_j|\, e^{\ii \theta_j}$ and suppose $|V_j|$ are given.
Consider:
\begin{subequations}
\bq
\min_{p, P, \theta} & \!\!\!\! & \!\!\!\!  C(p)
\\
\text{subject to}  & \!\!\!\! & \!\!\!\!
 	\underline{p}_j \leq p_j \leq \overline{p}_j, \qquad\quad\ \   j\in N^+
\\
& \!\!\!\!  & \!\!\!\!  		
	\underline{\theta}_{jk} \leq \theta_{jk} \leq \overline{\theta}_{jk},  \quad (j,k)\in E
\\
& \!\!\!\! & \!\!\!\!
 	p_j = \sum_{k: k\sim j} P_{jk}, \qquad \quad\ \   j\in N^+
\\
& \!\!\!\! & \!\!\!\!
	P_{jk} = |V_j|^2 g_{jk} - |V_j| |V_k| g_{jk}\cos\theta_{jk} 
	+ |V_j| |V_k| b_{jk}\sin\theta_{jk},
	\quad    (j,k)\in E
\label{eq:pfOPF.1e}
\eq
\label{eq:pfOPF.1}
\end{subequations}
where $\theta_{jk} := \theta_j - \theta_k$ are
the voltage angle differences across lines $(j,k)$.

We comment on the constraints on angles $\theta_{jk}$ in \eqref{eq:pfOPF.1}.
When the voltage magnitudes $|V_i|$ are fixed, constraints on real power flows,
branch currents, line losses, as well as stability constraints
 can all be represented in terms of $\theta_{jk}$.  
Indeed a line flow constraint of the form $|P_{jk}| \leq \overline{P}_{jk}$ becomes 
a constraint on $\theta_{jk}$ using the expression for $P_{jk}$ in \eqref{eq:pfOPF.1e}.
A current constraint of the form $|I_{jk}| \leq \overline{I}_{jk}$ is also a constraint
on $\theta_{jk}$ since $|I_{jk}|^2 = |y_{jk}| ( |V_j |^2 + |V_k|^2 - 2 |V_jV_k| \cos \theta_{jk})$.
The line loss over $(j,k)\in E$ is equal to $P_{jk} + P_{kj}$ which is again
a function of $\theta_{jk}$.
Stability typically requires $|\theta_{jk}|$ to stay within a small threshold.
Therefore given constraints on branch power or current flows, losses,
and stability, appropriate bounds $\underline{\theta}_{jk}, \overline{\theta}_{jk}$
can be determined in terms of these constraints, assuming $|V_j|$ are fixed.  

We can eliminate the branch flows $P_{jk}$
and  angles $\theta_{jk}$ from  \eqref{eq:pfOPF.1}.
Since $|V_j|, j\in N^+$, are fixed we
assume without loss of generality that $|V_j|=1$ pu.
Define the injection region
\bq
\mathbb P_\theta & \!\!\!\!\! :=  \!\!\!\!\! & \left\{ p\in \mathbb R^n \left|  
		p_j = \!\!\! \sum_{k: k\sim j} \!\! \left( g_{jk} -  g_{jk}\cos\theta_{jk} + b_{jk}\sin\theta_{jk} \right), \, j\in N^+,
 \  
	\underline{\theta}_{jk} \leq \theta_{jk} \leq \overline{\theta}_{jk},
	 \, (j,k) \in E  \right. \right\}
\label{eq:defP}
\eq
Let $\mathbb P_p := \{p\in \mathbb R^n \, |\,  \underline{p}_j \leq p_j \leq \overline{p}_j,  j\in N\}$.
Then \eqref{eq:pfOPF.1} is:
\\
\noindent
\textbf{OPF:}
\bq
\underset{p}{\min} \ C(p)  &  \text{subject to} &    p \in \mathbb P_\theta \cap \mathbb P_p
\label{eq:pfOPF}
\eq
This problem is hard because the set $\mathbb P_\theta$ is nonconvex.
To avoid triviality we assume OPF \eqref{eq:pfOPF} is feasible.
For a set $A$ let conv$\, A$ denote the convex hull of $A$.
Consider the following problem that relaxes the nonconvex feasible set 
$\mathbb P_\theta \cap \mathbb P_p$ of \eqref{eq:pfOPF} to a convex superset:
\\
\noindent
\textbf{OPF-socp:}
\bq
\underset{p}{\min} \ C(p)  &  \text{subject to} &    p \in \text{conv}(\mathbb P_\theta) \, \cap \, \mathbb P_p
\label{eq:pfOPF-socp}
\eq
We will show below that \eqref{eq:pfOPF-socp} is indeed an SOCP.  It is
said to be \emph{exact} if every optimal solution of \eqref{eq:pfOPF-socp}
lies in $\mathbb P_\theta \cap \mathbb P_p$ and is therefore
also optimal for \eqref{eq:pfOPF}.

We say that a point $x\in A \subseteq \mathbb R^n$ is a \emph{Pareto optimal point} 
in $A$ if there does not exist another $x'\in A$ such that $x' \leq x$ with at least one
strictly smaller component $x_j' < x_j$.  The \emph{Pareto front of $A$}, 
denoted by $\mathbb O(A)$,  is the set of all Pareto optimal points in $A$.  
The significance of $\mathbb O(A)$ is that, for any increasing function, its
minimizer, if exists, is necessarily in $\mathbb O(A)$ whether $A$ is convex
or not.   If $A$ is convex then 
$x^\text{opt}$ is a Pareto optimal point in $\mathbb O(A)$ if and only if there is a
nonzero vector $c := (c_1, \dots, c_n) \geq 0$ such that $x^\text{opt}$ is a 
minimizer of $c^T x$ over $A$ \cite[pp.179--180]{Boyd2004}.

Assume
\begin{enumerate}
\item[C1:] $C(p)$ is {strictly} increasing in each $p_j$.
\item [C2:] For all $(j,k)\in E$, $-\tan^{-1}\frac{b_{jk}}{g_{jk}} < \underline{\theta}_{jk} \leq \overline{\theta}_{jk}
< \tan^{-1}\frac{b_{jk}}{g_{jk}}$.
\eee
The following result, proved in \cite{LavaeiTseZhang2012, LamZhang2012},
says that \eqref{eq:pfOPF-socp} is exact provided  $\theta_{jk}$ are suitably bounded.  
\begin{theorem} 
\label{thm:pfOPF}
Suppose $G$ is a tree and C1--C2 hold.    
\bee
\item  $\mathbb P_\theta \cap \mathbb P_p  \ = \ 
	\mathbb O( \text{conv}(\mathbb P_\theta)\, \cap\, \mathbb P_p  )$.
\item  The problem \eqref{eq:pfOPF-socp} is indeed an SOCP.  Moreover
	it is exact.
\eee
\end{theorem}
\vspace{0.1in}

C1 is needed to ensure every optimal solution of OPF-socp \eqref{eq:pfOPF-socp} is optimal
for OPF \eqref{eq:pfOPF}.  If $C(p)$ is nondecreasing but 
not {strictly} increasing in all $p_j$, then
$\mathbb P_\theta \cap \mathbb P_p  \subseteq
	\mathbb O( \text{conv}(\mathbb P_\theta)\, \cap\, \mathbb P_p  )$ 
and OPF-socp may not be exact according to our definition.
Even in that case it is possible to recover an optimal solution of OPF from any optimal
solution of OPF-socp.

Theorem \ref{thm:pfOPF} is illustrated in Figures \ref{fig:bim2busNoC} and 
\ref{fig:bim2busC}.
\begin{figure}[htbp]
\centering
	\includegraphics[width=0.45 \textwidth]{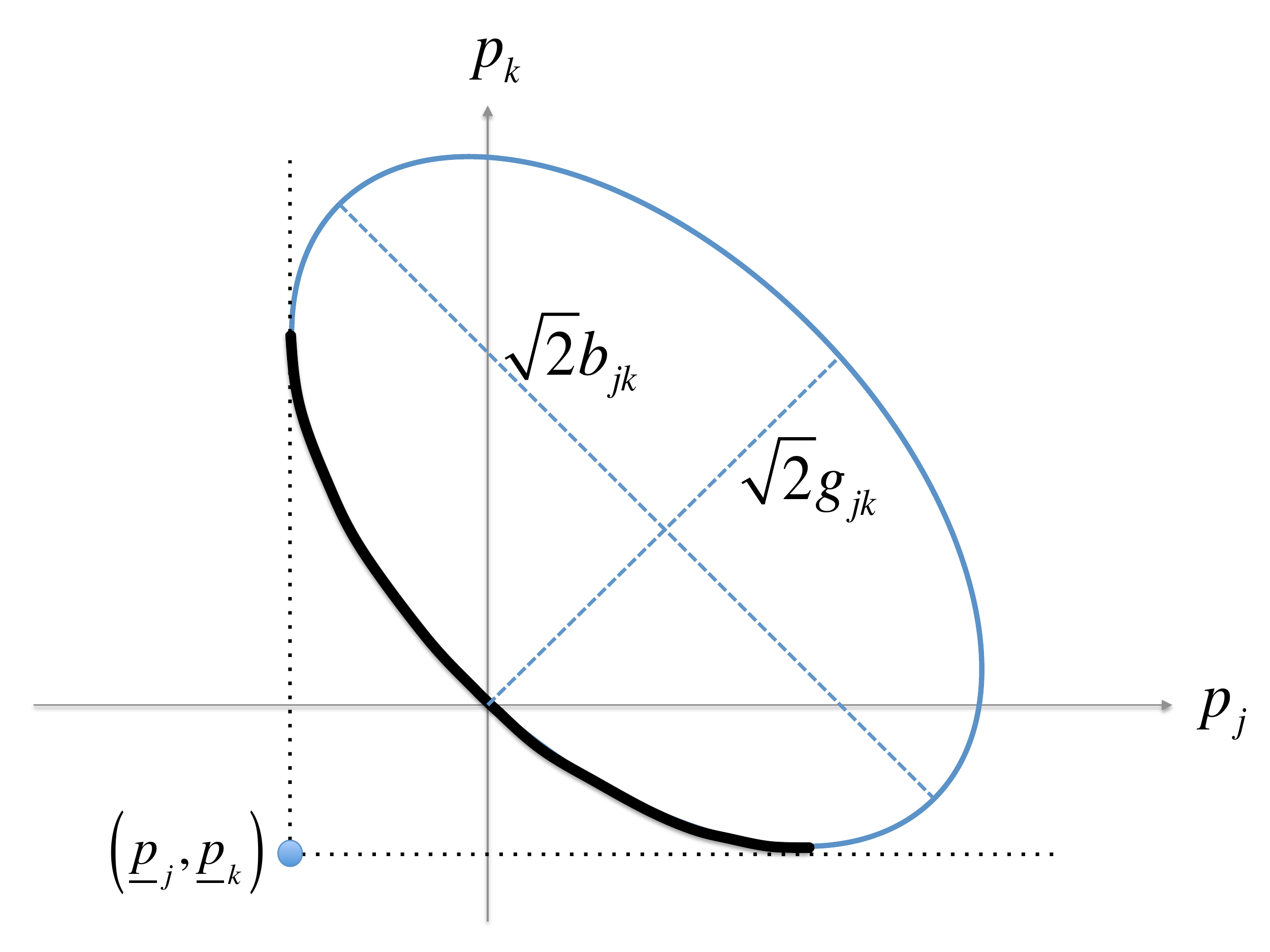}
\caption{Feasible set of OPF \eqref{eq:pfOPF} for a two-bus network without 
	any constraint when $|V_j|$ are fixed and reactive powers are ignored.
	It is an ellipse without the interior, hence nonconvex.  
	OPF-socp \eqref{eq:pfOPF-socp} includes the interior of the ellipse and 
	is hence convex. If the cost function $C$ is strictly increasing in $(p_j, p_k)$
	then the Pareto front of the SOCP feasible set will lie on the lower part of 
	the ellipse, $\mathbb O (\mathbb P_\theta) = \mathbb P_\theta$,
	and hence OPF-socp is exact.
}
\label{fig:bim2busNoC}
\end{figure}
\begin{figure}[htbp]
\centering
\subfigure [Exact relaxation with constraint] {
	\includegraphics[width=0.4 \textwidth]{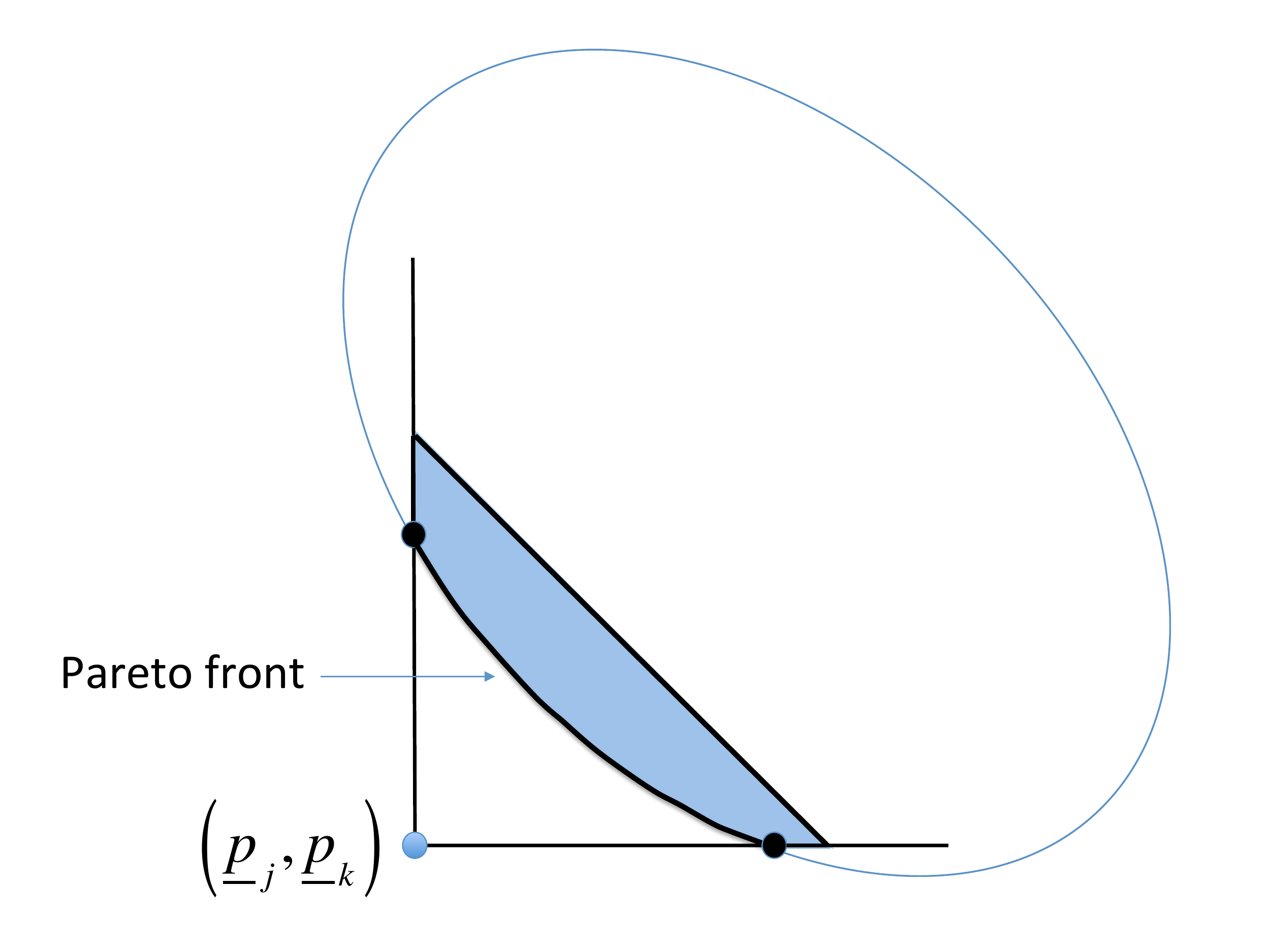}
	}
\subfigure [Inexact relaxation with constraint] {
	\includegraphics[width=0.4 \textwidth]{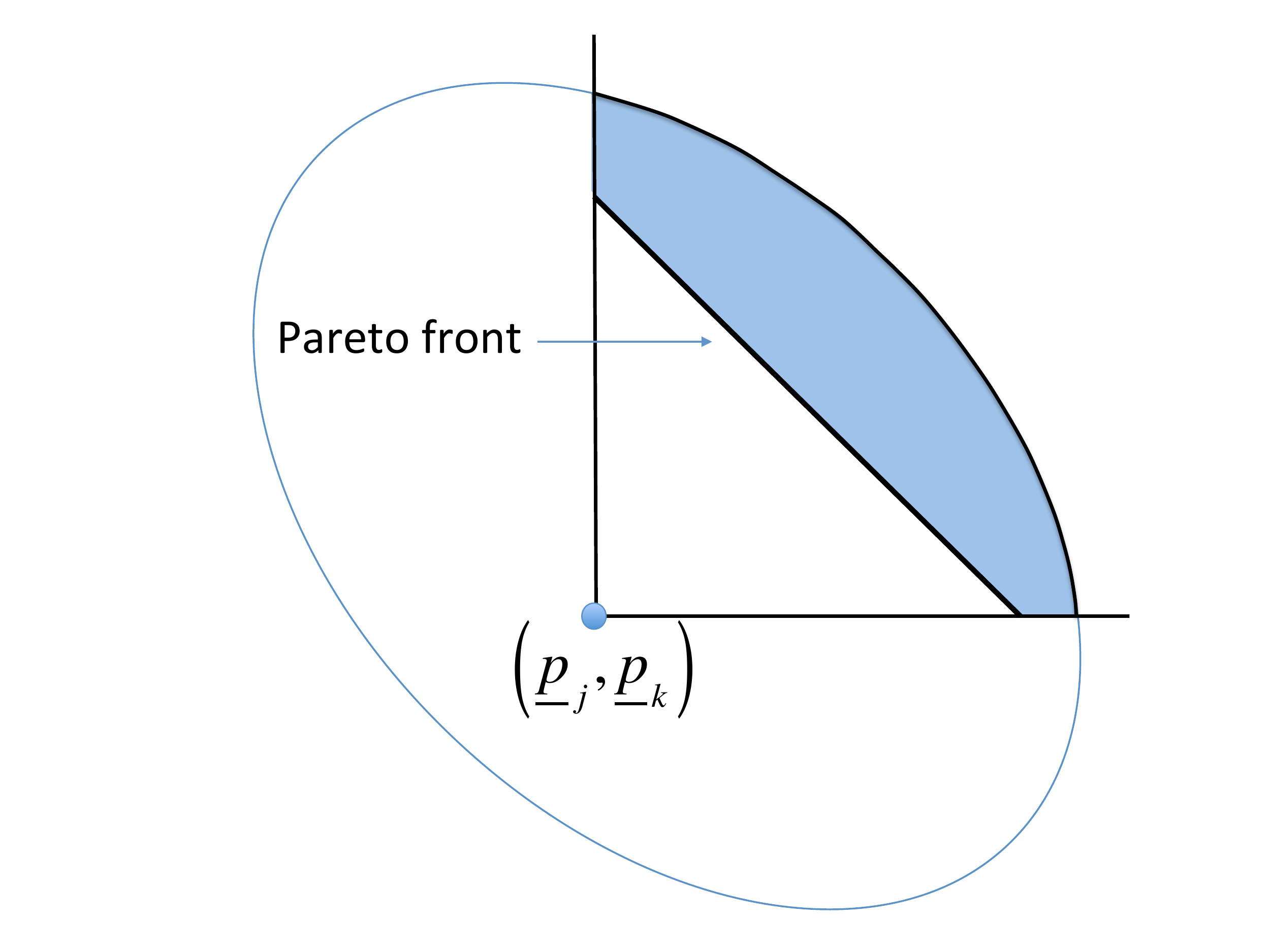}
	}
\caption{With lower bounds
$\underline p$ on power injections, the feasible set of OPF-socp \eqref{eq:pfOPF-socp}
is the shaded region.
(a) When the feasible set of OPF \eqref{eq:pfOPF} is restricted to the lower half of 
	the ellipse (small $|\theta_{jk}|$), the Pareto front
	remains on the ellipse itself,
	$\mathbb P_\theta \cap \mathbb P_p  =
	\mathbb O( \text{conv}(\mathbb P_\theta)\, \cap\, \mathbb P_p  )$,
	and hence the relaxation is exact.
(b) When the feasible set of OPF includes upper half of the ellipse (large $|\theta_{jk}|$),
	the Pareto front may not lie on the ellipse if $\underline p$ is large, making
	the relaxation not exact.
}
\label{fig:bim2busC}
\end{figure}
As explained in the caption of Figure \ref{fig:bim2busNoC}, if there are no 
constraints then  SOCP relaxation \eqref{eq:pfOPF-socp} is exact under 
condition C1.  It is clear from the figure that upper bounds on power injections
do not affect exactness whereas lower bounds do.
The purpose of condition C2 is to restrict the angle $\theta_{jk}$ in order to eliminate the
upper half of the ellipse from $\mathbb P_\theta$.  As explained in the caption of 
Figure \ref{fig:bim2busC}, under C2, 
$\mathbb P_\theta \cap \mathbb P_p  =
	\mathbb O( \text{conv}(\mathbb P_\theta)\, \cap\, \mathbb P_p  )$
 and hence the relaxation is exact.  Otherwise it may not.
\vspace{0.2in}

When the network is not radial or $|V_j|$ are not constants, then the feasible set can be 
much more complicated
than ellipsoids \cite{Hiskens-2001-OPFboundary-TPS, LesieutreHiskens-2005-OPF-TPS, Hill-2008-OPFboundary-TPS}.
Even in such  settings the Pareto fronts might still coincide, though the simple
geometric picture is lost.
See \cite{Hiskens-1995-Nonlinearity} for a numerical example on an Australian system or 
\cite{Bose-2014-BFMe-TAC} on a three-bus mesh network.

\subsection{Equivalence}

Since BIM and BFM are equivalent, the results on exact SOCP relaxation and
uniqueness of optimal solution apply in both models.
Recall the linear bijection $g$ from BIM to BFM defined in 
\cite[end of Section V]{Low2014a} by ${x}  = g(W_G)$ where 
\bqn
S_{jk} &:= &   y_{jk}^H \left( [W_G]_{jj} - [W_G]_{jk} \right),   
			\qquad\qquad\qquad\qquad\qquad\ \, \,  j\rightarrow k \in \tilde E
\label{eq:g.1}
\\
\ell_{jk} & \!\!\!   := \!\!\!   &   |y_{jk}|^2 \left( [W_G]_{jj} + [W_G]_{kk} - [W_G]_{jk} - [W_G]_{kj} \right),
			\qquad\ j \rightarrow k \in \tilde E
\label{eq:g.2}
\\
v_j & := &  [W_G]_{jj}, 
			\qquad\qquad\qquad\qquad\qquad\qquad\qquad\qquad\qquad\ \,    j\in N^+
\label{eq:g.3}
\\
s_j & := & \sum_{k: j\sim k} y_{jk}^H \left( [W_G]_{jj} - [W_G]_{jk} \right), 
			\qquad\qquad\qquad\qquad\qquad\ \, \,  j\in N^+
\eqn
The mapping $g$ allows us to directly apply Theorem \ref{thm:uniqueBFM} to BIM.
We summarize all the results for type A and type B conditions for radial networks.
\footnote{To apply type C conditions to BFM, one needs to translate
the angles $\theta_{jk}$ to the BFM variables 
$x := (S, \ell, v, s)$ through $\beta_{jk}(x)$, though this will introduce additional 
nonconvex constraints into OPF of the form 
$\underline \theta_{jk} \leq \beta_{jk}(x) \leq \overline \theta_{jk}$.
}
\begin{theorem}
Suppose $G$ and $\tilde G$ are trees.  Suppose conditions A1--A2', or A3--A4, or B1--B3
hold.  Then
\bee
\item \emph{BIM:} SOCP relaxation \eqref{eq:bimOPF-socp} is exact.  Moreover if $C(W_G)$
	is convex in $([W_G]_{jj}, [W_G]_{jk})$ then the optimal solution is unique.

\item \emph{BFM:} SOCP relaxation \eqref{eq:bfmOPF-socp} is exact.  Moreover if
	$C(x) := \sum_j C_j(p_j)$ is convex in $p$ 
	then the optimal solution is unique.
\eee
\end{theorem}
\vspace{0.1in}
Since both the SDP and the chordal relaxations are equivalent to the SOCP
relaxation for radial networks, these results apply to SDP and chordal relaxations as well.

\section{Mesh networks}
\label{sec:exact-mesh}

In this section we summarize a result of 
\cite[Part II]{Farivar-2013-BFM-TPS} on mesh networks with phase shifters
and of \cite[Part I]{Farivar-2013-BFM-TPS}, \cite{Rantzer2011, Gan-2014-OPFDC-TPS} 
on dc networks when all voltages are nonnegative.

To be able to recover an optimal solution of OPF from an optimal solution 
$W_G^\text{socp} / x^\text{socp}$ of SOCP relaxation, $W_G^\text{socp} / x^\text{socp}$
must satisfy both a local condition and a global cycle condition
(\eqref{eq:cyclecond.2} for BIM and \eqref{eq:cyclecond.1} for BFM); see the 
definition of exactness in Section \ref{sec:opf}.  
The conditions of Section \ref{sec:exact-radial} 
guarantee that every SOCP optimal solution will satisfy the local condition 
(i.e., $W_G^\text{socp}$ is $2\times 2$ psd rank-1 and $x^\text{socp}$ attains
equalities in \eqref{eq:mdf.3}), \emph{whether the network is radial or mesh}, 
but do not guarantee that it satisfies the cycle condition.
For radial networks, the cycle condition is vacuous and therefore
the conditions of Section \ref{sec:exact-radial} 
are sufficient for SOCP relaxation to be exact.    The result of 
\cite[Part II]{Farivar-2013-BFM-TPS} implies that these conditions are sufficient
also for a mesh network that has tunable phase shifters at strategic locations.

Similar conditions also extend to dc networks where
all variables are real and the voltages are assumed nonnegative.

\subsection{AC networks with phase shifters}

For BFM the conditions of Section \ref{sec:exact-radial}
guarantee that every optimal solution of OPF-socp \eqref{eq:bfmOPF-socp}
attains equalities in \eqref{eq:mdf.3} but may or may not satisfy the
cycle condition \eqref{eq:cyclecond.1}. 
If it does then it can be uniquely mapped to an
optimal solution of OPF \eqref{eq:OPFbfm}, according to 
\cite[Theorem 2]{Farivar-2013-BFM-TPS}.
If it does not then the solution is not physically implementable because it does not
satisfy the power flow equations (Kirchhoff's laws).
For a radial network the reduced incidence matrix $B$ in \eqref{eq:cyclecond.1}
is $n\times n$ and invertible and hence every optimal solution of the SOCP
relaxation that attains equalities in \eqref{eq:mdf.3} always satisfies the 
cycle condition \cite[Theorem 4]{Farivar-2013-BFM-TPS}.
This is not the case for a mesh network where $B$ is $m \times n$ with $m>n$.

It is proved in \cite[Part II]{Farivar-2013-BFM-TPS} however that if the network
has tunable phase shifters then any SOCP solution that attains equalities in 
\eqref{eq:mdf.3} becomes implementable even if the solution does not satisfy
the cycle condition.  
This extends the sufficient conditions A1--A2', or A3--A4, or B1--B3, or C0--C1 from
radial networks to this type of mesh networks.

For BIM the effect of phase shifter is equivalent to introducing a free variable $\phi_c$ in
\eqref{eq:cyclecond.2} for each basis cycle $c$ so that the cycle condition can 
always be satisfied for any $W_G$.
The results presented here however start with a simple power flow model 
\eqref{eq:bfmps} for networks 
with phase shifters.  This model makes transparent the effect of the spatial distribution 
of phase shifters and how they impact the exactness of SOCP relaxation 
 and can be useful in other contexts, such as the design of a network of FACTS 
 (Flexible AC Transmission Systems) devices.

\vspace{0.05in}
\noindent
\emph{BFM with phase shifters.}
We consider an idealized phase shifter
that  only shifts the phase angles of the sending-end voltage and current across a line,
and has no impedance nor limits on the shifted angles.
  Specifically consider
an idealized phase shifter parametrized by $\phi_{jk}$ across line $j\rightarrow k$ as shown in
Figure \ref{fig:PST}.
\begin{figure}[htbp]
\centering
\includegraphics[scale=0.21]{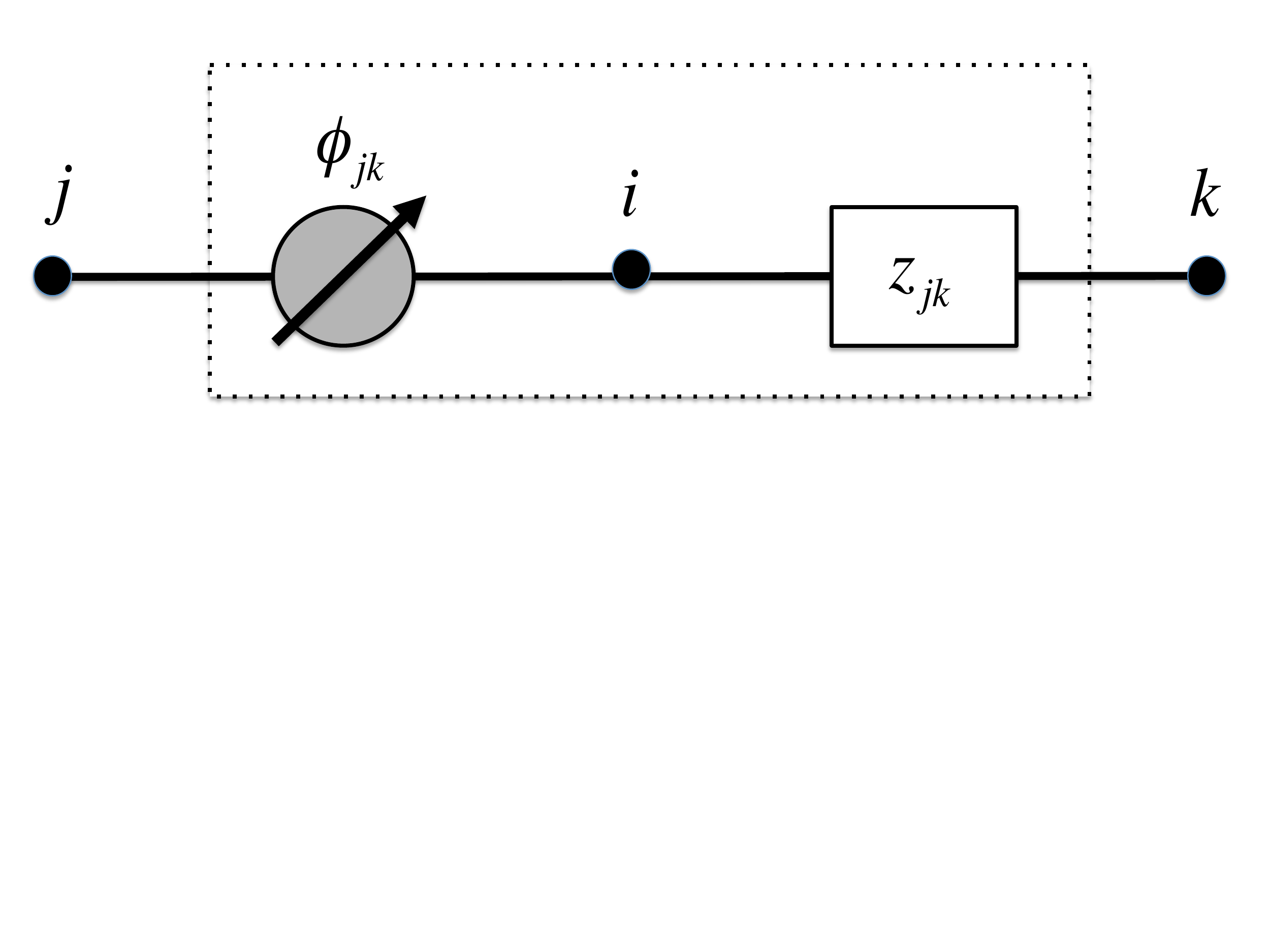}
\vspace{-0.9in}
\caption{\label{fig:PST} Model of a phase shifter in line $j\rightarrow k$.}
\end{figure}
As before let $V_j$ denote the sending-end voltage at node $j$.  Define $I_{jk}$ to be the 
{\em sending-end} current
leaving node $j$ towards node $k$.    Let $i$ be the point between the phase shifter $\phi_{jk}$
and line impedance $z_{jk}$.   Let $V_i$ and $I_i$ be the voltage at $i$ and the current from $i$
to $k$ respectively.  Then the effect of an idealized phase shifter, parametrized by $\phi_{jk}$,
is summarized by the following modeling assumptions:
\bqn
V_i \, = \, V_j\ e^{\ii \phi_{jk}} & \text{and} &
I_{i} \, =\, I_{jk}\ e^{\ii \phi_{jk}}
\eqn
The power transferred from nodes $j$ to $k$ is still (defined to be)
$S_{jk} :=  V_j I_{jk}^H$, which is equal to the power $V_i I_{i}^H$ from nodes $i$ to $k$
since the phase shifter is assumed to be lossless.
Applying Ohm's law across $z_{jk}$, we define the {\em branch flow model with phase
shifters} as the following set of equations:
\begin{subequations}
\bq
\!\!\!\!\!\!\!\!\!\!\!
 \sum_{k: j\rightarrow k} \!\! S_{jk}   & \!\!\! = \!\!\!
 	 &  \!\! \sum_{i: i\rightarrow j} \!\! \left( S_{ij} - z_{ij} |I_{ij}|^2 \right) + s_j,
		\qquad j\in N^+
\label{eq:bfmps.3}
\\
\!\!\!\!\!\!\!\!\!\!\!
I_{jk} & \!\!\! = \!\!\! & y_{jk} \left( V_j - V_k \ e^{- \ii \phi_{jk}} \right) \!,
		\quad\  j\rightarrow k \in \tilde E
\label{eq:bfmps.1}
\\
\!\!\!\!\!\!\!\!\!\!\!
S_{jk} & \!\!\! = \!\!\! & V_j I_{jk}^H, \qquad\qquad\qquad\quad  \ j\rightarrow k \in \tilde E
\label{eq:bfmps.2}
\eq
\label{eq:bfmps}
\end{subequations}
Without phase shifters ($\phi_{jk}=0$), \eqref{eq:bfmps} reduces to BFM \eqref{eq:bfm}.
Let $\tilde x := (S, I, V, s) \in \mathbb{C}^{2(m+n+1)}$ denote the variables in \eqref{eq:bfmps}.
Let $x := (S, \ell, v, s) \in \mathbb R^{3(m+n+1)}$ denote the variables in 
SOCP relaxation \eqref{eq:bfmOPF-socp}.  
These variables are related through the mapping $x = h(\tilde x)$ where 
$\ell_{jk} = |I_{jk}|^2$ and $v_j = |V_j|^2$.  In particular, given any solution $\tilde x$
of \eqref{eq:bfmps}, $x := h(\tilde x)$ satisfies \eqref{eq:mdf} with equalities in \eqref{eq:mdf.3}.

\vspace{0.05in}
\noindent
\emph{Cycle condition.}
If every line has a phase shifter then the cycle condition changes from \eqref{eq:cyclecond.1}
to: given any $x$ that satisfies \eqref{eq:mdf} with equalities in \eqref{eq:mdf.3},
\bq
\!\!\!\!\!\!
\exists (\theta, \phi)\in \mathbb R^{n+m} &\!\!\!\!  \text{such that}  \!\!\!\!  & 
		B\theta  =  \beta(x)  - \phi \!\!\!  \mod 2\pi
\label{eq:cyclecond.ps}
\eq
It is proved in \cite[Part II]{Farivar-2013-BFM-TPS} that, 
given any $x$ that attains equalities in \eqref{eq:mdf.3}, there always exists a
$\theta$ in $(-\pi, \pi]^n$ and a $\phi$ in $(-\pi, \pi]^m$ that solve \eqref{eq:cyclecond.ps}.
Moreover phase shifters are needed only on lines not in a spanning tree.

\vspace{0.05in}
\noindent
\emph{Exact SOCP relaxation.}
Recall the OPF problem \eqref{eq:OPFbfm}
where the feasible set $\mathbb{\tilde X}$ without phase shifters is:
\bqn
\mathbb{\tilde X}  & \!\!\!\! := \!\!\!\! &  \{ \tilde{x}  \ \vert \ \tilde{x} \text{ satisfies } 
				\eqref{eq:bfmps} \text{ with } \phi = 0 \text{ and } \eqref{eq:bfmopf} \}
\eqn
Phase shifters on every line enlarge the feasible set to: 
\bqn
\overline{\mathbb X} & \! := &\!\!  \left\{ \tilde x \ \vert \ \tilde{x} 
		\text{ satisfies \eqref{eq:bfmps} for some $\phi$ and \eqref{eq:bfmopf}}  \right\}
\eqn
Given any spanning tree $T$ 
of $\tilde G$, let ``$\phi\in T^{\perp}$'' be the shorthand for ``$\phi_{jk}=0$ for all $(j,k)\in T$'', 
i.e., $\phi$ involves only phase shifters in lines not in the spanning tree $T$. 
Fix any $T$.  Define the feasible set when there are phase shifters only on lines
outside $T$:
\bqn
\overline{\mathbb X}_T \!\!\! &\!\!\!  :=  \!\!\! &\!\!\! \{ \tilde x \ \vert \ \tilde{x} 
		\text{ satisfies \eqref{eq:bfmps} for some $\phi \in T^\perp$ and \eqref{eq:bfmopf}} \}
\eqn
Clearly $\tilde{\mathbb X} \subseteq \overline{\mathbb X}_T \subseteq \overline{\mathbb X}$.
Define the (modified) OPF problem where there is a phase shifter on every line:
\\
\noindent
\textbf{OPF-ps:}
\bq
\min_{\tilde x, \phi}\ C(x) & \text{subject to} & \tilde x \in \overline{\mathbb X}, \ \phi\in \mathbb R^m
\label{eq:opf4a.b}
\eq
and that where there are phase shifters only outside $T$:
\\
\noindent
\textbf{OPF-$T$:}
\bq
\min_{\tilde x, \phi}\  C(x) & \text{subject to} &
	  \tilde x \in \overline{\mathbb X}_T, \ \phi \in T^\perp
\label{eq:opf5a.b}
\eq
Let $C^\text{opt}$, $C^\text{ps}$, and $C^{T}$ denote respectively the optimal values of 
OPF \eqref{eq:OPFbfm}, OPF-ps \eqref{eq:opf4a.b}, and OPF-$T$ \eqref{eq:opf5a.b}.
Clearly $C^\text{opt} \geq C^T \geq C^\text{ps}$ since
$\mathbb{\tilde X} \subseteq \overline{\mathbb X}_T \subseteq \overline{\mathbb X}$.
Solving OPF \eqref{eq:OPFbfm}, OPF-ps \eqref{eq:opf4a.b}, or OPF-$T$ \eqref{eq:opf5a.b} is difficult because their feasible sets are nonconvex.

Recall the following sets defined in \cite{Low2014a} for networks without phase shifters:
\bqn
\mathbb X^+ & \!\!\!\! := \!\!\!\! & \{ x \ | \ x \text{ satisfies \eqref{eq:bfmopf} and \eqref{eq:mdf}} \}
\\
\mathbb X_{nc} & \!\!\!\! := \!\!\!\! & \{ x \ | \ x \text{ satisfies \eqref{eq:bfmopf} and 
					\eqref{eq:mdf} with equalities in \eqref{eq:mdf.3}} \}
\\
\mathbb X & \!\!\!\! := \!\!\!\! & \{ x \ | \ x \in \mathbb X_{nc} \text{ and satisfies the cycle
			condition \eqref{eq:cyclecond.1}} \}
\eqn
Note that $\mathbb X$ is defined by the cycle condition without phase shifters $(\phi = 0$ in 
\eqref{eq:cyclecond.ps}).
As explained in \cite[Theorem 9]{Low2014a}, $\mathbb X$ is equivalent to
the feasible set $\mathbb{\tilde X}$ of OPF \eqref{eq:OPFbfm}.  Hence 
$\mathbb{\tilde X} \equiv \mathbb X \subseteq \mathbb X_{nc} \subseteq \mathbb X^+$.
A key result of \cite[Part II]{Farivar-2013-BFM-TPS} is
\begin{theorem}
\label{thm:eqps}
Fix any spanning tree $T$ of $\tilde G$.  Then
$\overline{\mathbb X}_T \ = \ \overline{\mathbb X} \ \equiv  \ \mathbb X_{nc}$.
\end{theorem}
\vspace{0.1in}

The implication of Theorem \ref{thm:eqps} is that, for a mesh network, when 
a solution of SOCP relaxation \eqref{eq:bfmOPF-socp} attains equalities in
\eqref{eq:mdf.3} (i.e., it is in $\mathbb X_{nc}$), 
then it can be implemented with an appropriate setting of phase
shifters even when the solution does not satisfy the cycle condition 
\eqref{eq:cyclecond.1}.    
Define the problem:
\\
\noindent
\textbf{OPF-nc:}
\bq
\min_{x}\  C(x) & \text{subject to} & x \in \mathbb X_{nc}
\label{eq:opfnc}
\eq
Let $C^\text{nc}$ and $C^\text{socp}$ denote respectively the 
optimal values of  OPF-nc \eqref{eq:opfnc} and OPF-socp \eqref{eq:bfmOPF-socp}.
Theorem \ref{thm:eqps} then implies
\begin{corollary}
\label{coro:eqps}
Fix any spanning tree $T$ of $\tilde G$.  Then
\bee
\item $\mathbb{\tilde X} \subseteq \overline{\mathbb X}_T \ = \ \overline{\mathbb X} \ \equiv  \ \mathbb X_{nc} \ \subseteq \ \mathbb X^+$.
\item $C^\text{opt} \geq C^T = C^\text{ps} = C^\text{nc} \geq C^\text{socp}$.
\eee
\end{corollary}

Hence if an optimal solution $x^\text{socp}$ of OPF-socp \eqref{eq:bfmOPF-socp}
attains equalities in  \eqref{eq:mdf.3} then $x^\text{socp}$ 
solves the problem OPF-nc \eqref{eq:opfnc}.   If it also satisfies the cycle 
condition \eqref{eq:cyclecond.1} then $x^\text{socp} \in \mathbb X$ and it
can be mapped to a unique optimal of OPF \eqref{eq:OPFbfm}.
Otherwise, 
$x^\text{socp}$ can be implemented through an appropriate phase shifter setting $\phi$
and it attains a cost that lower bounds the optimal cost of the original
OPF without tunable phase shifters.
Moreover this benefit can be attained with phase shifters
only outside an arbitrary spanning tree $T$ of $\tilde G$.
The result can help determine if a network with a given set of
phase shifters can be convexified and, if not, where additional phase shifters
are needed for convexification \cite[Part II]{Farivar-2013-BFM-TPS}.

Corollary \ref{coro:eqps} also implies that, if SOCP is exact, then phase shifters cannot
further reduce the cost.   This can help determine when phase shifters
 provide benefit to  system operations.

Hence phase shifters in strategic locations make a mesh network 
behave like a radial network
as far as convex relaxation is concerned.  The results of Section \ref{sec:exact-radial} then imply
\begin{corollary}
Suppose conditions {A1--A2', or A3--A4, or B1--B3, or C1--C2} hold.  
Then any optimal solution of OPF-socp \eqref{eq:bfmOPF-socp}
solves OPF-ps \eqref{eq:opf4a.b} and OPF-$T$ \eqref{eq:opf5a.b}.
\end{corollary}

\subsection{DC networks}

In this subsection we consider purely resistive dc networks,  
i.e., the impedance $z_{jk}  = r_{jk} = y_{jk}^{-1}$, the power injections 
$s_j = p_j$, and the voltages $V_j$ are real.   
We assume all voltage magnitudes are strictly positive.  Formally:
\bee
\item [D0:] Replace \eqref{eq:bimopf.2} and \eqref{eq:mdf.2} by 
	$0 < \underline V_j \leq V_j \leq \overline V_j$, $j\in N^+$,
	and replace \eqref{eq:opfW.2} by 
	$0 < \underline V_j^2 \leq [W_G]_{jj} \leq \overline V_j^2$, $j\in N^+$. 
\eee

\noindent
\emph{Type A conditions.}
Condition D0 immediately implies that the cycle condition \eqref{eq:cyclecond.1}
in BFM is satisfied by every feasible $x$ of OPF-socp \eqref{eq:bfmOPF-socp}, 
for
\bqn
\beta_{jk}(x) & := & \angle \left(v_j - z_{jk}^H S_{jk} \right)
\ \ = \  \  \angle \left( v_j - r_{jk} \left( r_{jk}^{-1}V_j (V_j - V_k) \right)  \right)
\ \ = \  \  0
\eqn
A3--A4 guarantee that any optimal solution of OPF-socp attains
equality in \eqref{eq:mdf.3} for general mesh networks. 
Hence \cite[Theorem 7]{Low2014a} and Theorem \ref{thm:bfmsocp.1} imply
\begin{corollary}
Suppose A3--A4 and D0 hold.  Then OPF-socp \eqref{eq:bfmOPF-socp} is exact.
\end{corollary}
\vspace{0.1in}


For BIM, consider an OPF as a QCQP \eqref{eq:OPFbim.2} 
where all the matrices are real and symmetric.   
Even though all the QCQP matrices in \eqref{eq:OPFbim.2} satisfy
condition A2', Corollary \ref{coro: OPFSeparatingLine}
 is not directly applicable as its proof constructs a complex (rather than real)
 $V$ from an optimal solution of 
OPF-socp.   However if there are no lower bounds on the power injections,
then only $\Phi_j$ are involved in the QCQP so all their off-diagonal entries 
are negative.   It is then observed in \cite{Rantzer2011} that 
\cite[Theorem 3.1]{KimKojima2003} directly implies (without needing D0)
\begin{corollary}
\label{coro:dcqcqp}
Suppose A1 and A4 hold.   Then OPF-sdp \eqref{eq:bimOPF-sdp} and
OPF-socp \eqref{eq:bimOPF-socp} are exact.
\end{corollary}


\vspace{0.1in}
\noindent
\emph{Type B conditions.}
The following result is proved in \cite{Gan-2014-OPFDC-TPS}.
Consider: 
\bee
\item [B1':] The cost function is $C(x) := \sum_{j=0}^n C_j \left( \text{Re}\, s_j \right)$ 
	with $C_j$ strictly increasing for all $j\in N^+$.  There is no constraint on $s_0$.
\item [B2':] $\overline V_1 = \overline V_2 = \cdots = \overline V_n$; 
	$\mathbb S_j = [\underline p_j, \overline p_j]$ with $\underline p_j < 0$, $j\in N$.
\item [B2'':] $\overline V_j = \infty$ for $j\in N$.
\eee
\begin{theorem}
\label{thm:dcmesh.2}
Suppose at least one of the following holds:
\bi
\item B1, B2'' and D0; or
\item B1', B2' and D0.
\ei
Then OPF-socp \eqref{eq:bimOPF-socp} with the additional constraints 
$W_{jk}\geq 0$, $(j,k)\in E$, is exact.
If, in addition, the problem is convex then its optimal solution is unique.
\end{theorem}
\vspace{0.1in}
It is possible to enforce B2'' by an affine constraint on the power injections,
similar to (but different from) condition B2 for radial networks; see 
\cite{Gan-2014-OPFDC-TPS} for details.  See also \cite{Tan2014} for
a result on the uniqueness of SOCP relaxation.

\subsection{General AC networks}

Unfortunately no sufficient conditions for exact semidefinite relaxation for
general mesh networks are yet known. 
There are type A conditions
on power injections for exact relaxation only for special cases: 
a lossless cycle or lossless cycle with one chord \cite{Zhang2013},
or a weakly cyclic network (where every line belongs to at most one cycle)
of size 3 \cite{Madani2013}.

We close by mentioning three recent approaches for global optimization
of OPF when the relaxations in this tutorial fail.
First, higher-order semidefinite relaxations on the Lesserre hierarchy for polynomial 
optimization \cite{Lasserre2001} have been applied to solving OPF 
when SDP relaxation fails
\cite{JoszPanciatici2013, Molzahn2014a, Molzahn2014b, Ghaddar2014}.
By going up the hierarchy, the relaxations become  tighter and
their solutions approach a global optimal of the original polynomial
optimization \cite{Lasserre2001, Parrilo2003}.
This however comes at the cost of significantly higher runtime.
Techniques are proposed in \cite{Molzahn2014b, Ghaddar2014}
to reduce the problem sizes, e.g., by exploiting sparsity or adding redundant 
constraints \cite{Waki2006, Lasserre2009, Ghaddar2014}
or applying higher-order relaxations only on (typically small) subnetworks 
where constraints are violated \cite{Molzahn2014b}.

Second, a branch-and bound algorithm is proposed in \cite{Phan2012} 
where a lower bound is computed from the Lagrangian dual of OPF and 
the feasible set subdivision is based on rectangular or ellipsoidal bisection.  
The dual problem is solved using a subgradient algorithm.  Each iteration
of the subgradient algorithm requires minimizing the 
Lagrangian over the primal variables.  This minimization
is separable into two subproblems, one being a convex
subproblem and the other having a nonconvex quadratic objective.
The latter subproblem turns out to be a trust-region problem
that has a closed-form solution.
It is proved in \cite{Phan2012} that the proposed algorithm converges
to a global optimal.
This method is extended in \cite{Gopalakrishnan2012} to include more constraints
and alternatively use SDP relaxation for lower bounding the cost.

Finally a new approach is proposed in \cite{Hijazi2013}
based on convex quadratic relaxation of OPF in polar coordinates.

\section{Conclusion}
\label{sec:conc2}

We have summarized the main sufficient conditions for exact semidefintie
relaxations of OPF as listed in Tables \ref{table:SummaryRadial} and 
\ref{table:SummaryMesh}.   For radial networks these conditions  suggest that SOCP relaxation (and hence SDP and 
chordal relaxations) will likely be exact in practice.  This is corroborated 
by significant numerical experience.   For mesh networks they are applicable only for
special cases: networks that have tunable phase shifters or dc networks
where all variables are real and voltages are nonnegative.   
Even though counterexamples exist where SDP/chordal relaxation is not exact 
for AC mesh networks 
numerical experience seems to suggest that SDP/chordal relaxation tends
to be exact in many cases.  Sufficient conditions that
guarantee exact relaxation for AC mesh networks however remain 
elusive.   The main difficulty is in designing relaxations of the
cycle condition \eqref{eq:cyclecond.2} or \eqref{eq:cyclecond.1}.


\newpage

\section{Appendix: proofs}
\label{app:proofs}

We prove all the main results here.  

\subsection{Proof of Theorem \ref{thm:exact4qcqp} and Corollary \ref{coro:exact4qcqp}:  linear separability}

The proof is from an updated version of \cite{Bose-2012-QCQPt}.
It is equivalent to the argument of \cite{Sojoudi2013} and simpler than
the original duality proof in \cite{Bose-2012-QCQPt}.

\begin{proof}[Proof of Theorem \ref{thm:exact4qcqp}]
Fix any partial matrix $W_G$ that is feasible for SOCP \eqref{eq:socp}.  
We will construct an $x \in \mathbb C^{n+1}$ that satisfies
\bqn
x^H C_l x & \leq & \text{tr } C_l W_G, \ \ \ l = 0, 1, \dots, L
\eqn
i.e., $x$ is feasible for QCQP \eqref{eq:qcqp} and has an equal or lower cost than $W_G$.
Since the minimum cost of QCQP is lower bounded by that of its SOCP relaxation
this means that an optimal solution $x\in \mathbb C^{n+1}$ of QCQP \eqref{eq:qcqp} 
can be obtained from every optimal solution $W_G$ of SOCP \eqref{eq:socp}.

Now $W_G(j,k) \succeq 0$ for every $(j,k)\in E$ implies that $[W_G]_{jj}\geq 0$ for all $j\in N$ and
\bqn
[W_G]_{jj} \, [W_G]_{kk} & \geq & \left| [W_G]_{jk} \right|^2,  \ \ (j,k)\in E
\eqn
 %
Suppose first that $[W_G]_{jj} [W_G]_{kk} = |[W_G]_{jk}|^2$ for all $(j,k)\in E$, i.e., $W_G$ is
$2\times 2$ psd rank-1.  
We will construct an $x\in \mathbb C^{n+1}$ that is feasible for QCQP and has an equal cost.
To construct such an $x$ let $|x_j| := \sqrt{[W_G]_{jj}}$, $j\in N^+$.  
Recall that $G$ is a (connected) tree with node 0 as its root.  Let $\angle x_0 := 0$.  
Traversing the tree starting from the root the angles can be successively assigned:
given $\angle x_j$ at one end of a link $(j,k)$, let $\angle x_k := \angle x_j - \angle [W_G]_{jk}$ at the 
other end.   Then for $l=0, 1, \dots, L$ we have
\bqn
x^H C_l x & = & \sum_{j, k} [C_l]_{jk} \, x_j x_k^H \ = \ \text{tr } C_l W_G
\eqn
Hence $x$ is feasible for QCQP \eqref{eq:qcqp} and has the same cost as $W_G$.

Next suppose $[W_G]_{jj} [W_G]_{kk} > |[W_G]_{jk}|^2$ for some $(j,k)$, i.e., $W_G$ is 
$2\times 2$ psd but not $2\times 2$  rank-1. 
We will 
\bee
\item Construct an $\hat{W}_G$ that is $2\times 2$ psd rank-1.
\item Show that A2 implies 
	\bq
	\text{tr } C_l \hat{W}_G & \leq & \text{tr } C_l W_G, \ \ l = 0, 1, \dots, L
	\label{eq:CWineq}
	\eq
\eee
Then an $x\in \mathbb C^{n+1}$ can be constructed from $\hat W_G$ as in the case above
and step 2 ensures that for $l = 0, 1, \dots, L$
\bqn
x^H C_l x \ = \ \text{tr } C_l \hat{W}_G \ \leq \ \text{tr } C_l W_G
\eqn
i.e., $x$ is feasible for QCQP \eqref{eq:qcqp} and has an equal or lower cost than $W_G$.

To construct such an $\hat{W}_G$ let $[\hat W_G]_{jj} = [W_G]_{jj}$, $j\in N^+$.  For $(j,k)\in E$ let
\bqn
[\hat W_G]_{jk} - [W_G]_{jk} & =: & r_{jk} e^{- \ii \left(\frac{\pi}{2} - \alpha_{jk} \right)}
\eqn
for some $r_{jk}>0$ to be determined and $\alpha_{jk}$ in  assumption A2.
For $\hat{W}_G$ to be $2\times 2$ psd rank-1 we need to choose $r_{jk}>0$ such that 
$[\hat W_G]_{jj} [\hat W_G]_{kk} = \left| [\hat W_G]_{jk} \right|^2$ for all $(j,k)\in E$, i.e.,
\bqn
[W_G]_{jj}\, [W_G]_{kk} & = & \left| [W_G]_{jk} + r_{jk} e^{- \ii \left(\frac{\pi}{2} - \alpha_{jk} \right)} \right|^2
\eqn
or
\bqn
r_{jk}^2 + 2b\, r_{jk} - c & = & 0
\eqn
where
\bqn
b & := & \text{Re}\left( [W_G]_{jk} \, e^{\ii \left(\frac{\pi}{2} - \alpha_{jk} \right)} \right)
\\
c & := & [W_G]_{jj}\, [W_G]_{kk} - \left| [W_G]_{jk} \right|^2 \ > \ 0
\eqn
Therefore setting $r_{jk} := \sqrt{b^2 + c} - b > 0$ yields an $\hat W_G$ that is $2\times 2$ psd rank-1.

To show that $\hat W_G$ is feasible for SOCP \eqref{eq:socp} and has an equal or lower cost than
$W_G$, we have for $l=0, 1, \dots, L$,
\bqn
\text{tr } C_l \hat W_G - \text{tr } C_l W_G
& = & \text{tr } C_l \, \left( \hat W_G - W_G \right)
\\
& = & \sum_{(j,k)\in E} [C_l]_{jk} \left( [\hat W_G]_{jk} - [W_G]_{jk} \right)^H
\\
& = & 2\, \sum_{j<k} \text{Re} \left( [C_l]_{jk} \cdot r_{jk} \, e^{\ii \left(\frac{\pi}{2} - \alpha_{jk} \right)}  \right)
\\
& = & 2\, \sum_{j<k} \left| [C_l]_{jk} \right| \, r_{jk} \ \cos \left( \angle [C_l]_{jk} + \frac{\pi}{2} - \alpha_{jk} \right) 
\\
& \leq & 0
\eqn
where the last inequality follows because assumption A2 implies
\bqn
\frac{\pi}{2} & \leq \ \, \angle [C_l]_{jk} + \frac{\pi}{2} - \alpha_{jk} \ \, \leq & \frac{3\pi}{2}
\eqn
and therefore $\cos \left( \angle [C_l]_{jk} + \frac{\pi}{2} - \alpha_{jk} \right) \leq 0$.
This completes the proof.
\end{proof}

\begin{proof}[Proof of Corollary \ref{coro:exact4qcqp}]
A1 implies that the objective function of SOCP \eqref{eq:socp} is strictly convex
and hence has a unique optimal solution.
Suppose $W_G$ is an optimal solution of SOCP \eqref{eq:socp} but 
$[W_G]_{jj} [W_G]_{kk} > |[W_G]_{jk}|^2$ for some $(j,k)$, i.e., $W_G$ is 
$2\times 2$ psd but not $2\times 2$ psd rank-1.   Then the above constructs 
another feasible solution $\hat{W}_G$ with equal cost.  This contradicts the uniqueness
of the optimal solution of SOCP \eqref{eq:socp}, and hence $W_G$ must be 
$2\times 2$ psd rank-1.
\end{proof}

\subsection{Proof of Theorem \ref{thm:bfmsocp.1}: no injection lower bounds (BFM)}
\label{subsec:pfbfmsocp.1}

The proof  is from \cite[Part I]{Farivar-2013-BFM-TPS}.  
\begin{proof}
Fix any optimal solution $x := (S, \ell, v, s) \in \mathbb R^{3(m+n+1)}$ of OPF-socp in the branch
flow model.  Since the network is radial, the cycle condition is vacuous and we only need to show that
$x$ attains equality in \eqref{eq:mdf.3} on all lines $j\rightarrow k \in \tilde E$.   For the sake of 
contradiction assume this is violated on $j\rightarrow $k, i.e., 
\bq
v_j \ell_{jk} & > & |S_{jk}|^2
\label{eq:pf.2}
\eq
We will construct an $\hat x$ that is feasible for OPF-socp and attains a strictly lower cost, contradicting 
that $x$ is optimal.

For an $\epsilon>0$ to be determined below, consider the following $\hat x$ obtained by modifying only
the current $\ell_{jk}$ and power flows $S_{jk}$ on line $j\rightarrow k$ and the injections $s_j, s_k$
at two ends of the line: 
\bqn
\hat \ell_{jk} & := & \ell_{jk} - \epsilon
\\
\hat S_{jk} & := & S_{jk} - z_{jk} \epsilon / 2
\\
\hat s_j & := & s_j - z_{jk} \epsilon / 2
\\
\hat s_k & := & s_k - z_{jk} \epsilon / 2
\eqn
and $\hat v \ := \ v$, $\hat \ell_{il} \ := \ \ell_{il}$ and $\hat S_{il} \ := \ S_{il}$ for $(i,l) \neq (j,k)$, 
 $\hat s_i := s_i$ for $i\neq j, k$.
 By assumption A3 the objective function $C(x)$ is strictly increasing in $\ell$ and hence $\hat x$ has
 a strictly lower cost than $x$.    It suffices to show that there exists an $\epsilon>0$ such that $\hat x$
 is feasible for OPF-socp, i.e., $\hat x$ satisfies \eqref{eq:mdf} and \eqref{eq:bfmopf}.
 
Assumption A4 ensures that $\hat x$ satisfies \eqref{eq:bfmopf}.  Further $\hat x$ satisfies \eqref{eq:mdf.1} 
at buses $i\neq j, k$, and satisfies \eqref{eq:mdf.2} and \eqref{eq:mdf.3} over lines $(i,l) \neq (j,k)$.   
We now show that $\hat x$ also satisfies \eqref{eq:mdf.1} at buses $j$ and $k$ and satisfies
\eqref{eq:mdf.2} and \eqref{eq:mdf.3} over line $(j,k)$.

For \eqref{eq:mdf.1} at bus $j$, we have (adopting the graph orientation where every link
points away from node 0):
\bqn
\sum_{l: j\rightarrow l} \hat S_{jl} & = & 
\sum_{l\neq k: j\rightarrow l} S_{jl}  + \left( S_{jk} - z_{jk}\epsilon /2 \right)
\ \ = \ \ 
S_{ij} - z_{ij} \ell_{ij} + s_j - z_{jk}\epsilon /2
\ \, = \, \  \hat S_{ij} - z_{ij} \hat \ell_{ij} + \hat s_j 
\eqn
as desired.  
For \eqref{eq:mdf.1} at $k$, we have
\bqn
\sum_{l: k\rightarrow l} \hat S_{kl} & = & \sum_{l: k\rightarrow l} S_{kl} 
\ = \ 
S_{jk} - z_{jk} \ell_{jk} + s_k 
\ \ = \ \  \hat S_{jk} - z_{jk} \hat \ell_{jk} + \hat s_k 
\eqn
as desired.
For \eqref{eq:mdf.2} over line $(j,k)$, we have
\bqn
\hat v_j - \hat v_k & = & v_j - v_k 
\ \ = \ \  2\, \text{Re} \left(z_{jk}^H S_{jk} \right) - |z_{jk}|^2 \ell_{jk}
\ \ = \ \  2\, \text{Re} \left(z_{jk}^H \hat S_{jk} \right) - |z_{jk}|^2 \hat \ell_{jk}
\eqn
as desired.
For \eqref{eq:mdf.3} over line $(j,k)$, we have 
\bqn
\hat v_j \hat \ell_{jk}  - \left| \hat S_{jk} \right|^2 & = & 
- \frac{\left| z_{jk} \right|^2}{4} \epsilon^2 \ - \,  \left( \! v_j -  \text{Re}\left( z_{jk}^H S_{jk} \right) \right)\epsilon
\ + \  \left( \ell_{jk}  v_j - \left| S_{jk} \right|^2 \right) 
\eqn
Hence \eqref{eq:pf.2} implies that we can always choose an $\epsilon>0$ 
such that $\hat v_j \hat \ell_{jk}  = \left| \hat S_{jk} \right|^2$.
This completes the proof.
\end{proof}
\vspace{0.1in} 

If the cost function $C(x)$ in A3 is only nondecreasing, rather than strictly increasing,
in $\ell$, then A3--A4 still guarantee that all optimal solutions of OPF \eqref{eq:OPFbfm}
are optimal for OPF-socp \eqref{eq:bfmOPF-socp}, but OPF-socp \eqref{eq:bfmOPF-socp} 
may have optimal solutions $x$ that maintain strict inequalities in \eqref{eq:mdf.3}.  
Even in this case, however, the above proof constructs from $x$
an optimal solution $\hat x$ of OPF-socp that attains equalities in
\eqref{eq:mdf.3} from which an optimal solution $\tilde x$ of OPF \eqref{eq:OPFbfm}
can be recovered.

\subsection{Proof of Theorem \ref{thm:bfmsocp.2}: voltage upper bounds}

The proof here is from \cite{Gan-2014-BFMt-TAC} with a slightly different presentation.   
Given an optimal solution $x$ that maintains a strict inequality in \eqref{eq:mdf.3}, 
the proof  in Section \ref{subsec:pfbfmsocp.1} of Theorem \ref{thm:bfmsocp.1} by contradiction 
constructs another feasible solution $\hat x$ that incurs a strictly smaller cost, contradicting
the optimality of $x$.   The modification is over
a single line over which $x$ maintains a strict inequality in \eqref{eq:mdf.3}.
The proof of Theorem \ref{thm:bfmsocp.2} is also by contradiction but, unlike that of 
Theorem \ref{thm:bfmsocp.1},
the construction of $\hat x$ from $x$ involves modifications on multiple lines, 
propagating from the line that is closest to bus 0 where \eqref{eq:mdf.3} holds with strict 
inequality all the way to bus 0.   The proof relies crucially 
on the recursive structure of the branch flow model \eqref{eq:gdf}.

\begin{proof}[Proof of Theorem \ref{thm:bfmsocp.2}]
To simplify notation we only prove the theorem
for the case of a linear network representing a primary feeder without laterals.
The proof for a general tree network follows the same idea but with more cumbersome notations;
see \cite{Gan-2014-BFMt-TAC} for details. 
We adopt the graph orientation where every
link points \emph{towards} the root node 0.
The notation for the linear network is explained in Figure \ref{fig:LinearNk} 
(recall that we refer to a link $j\rightarrow k$ by $j$ and index the associated
variables $z_{jk}, S_{jk}, \ell_{jk}$ with $j$).\footnote{Note that
$m$ in this subsection does not denote the number of edges in $\tilde G$, which is $n$.}
\begin{figure}[htbp]
\centering
\includegraphics[width=0.55\textwidth]{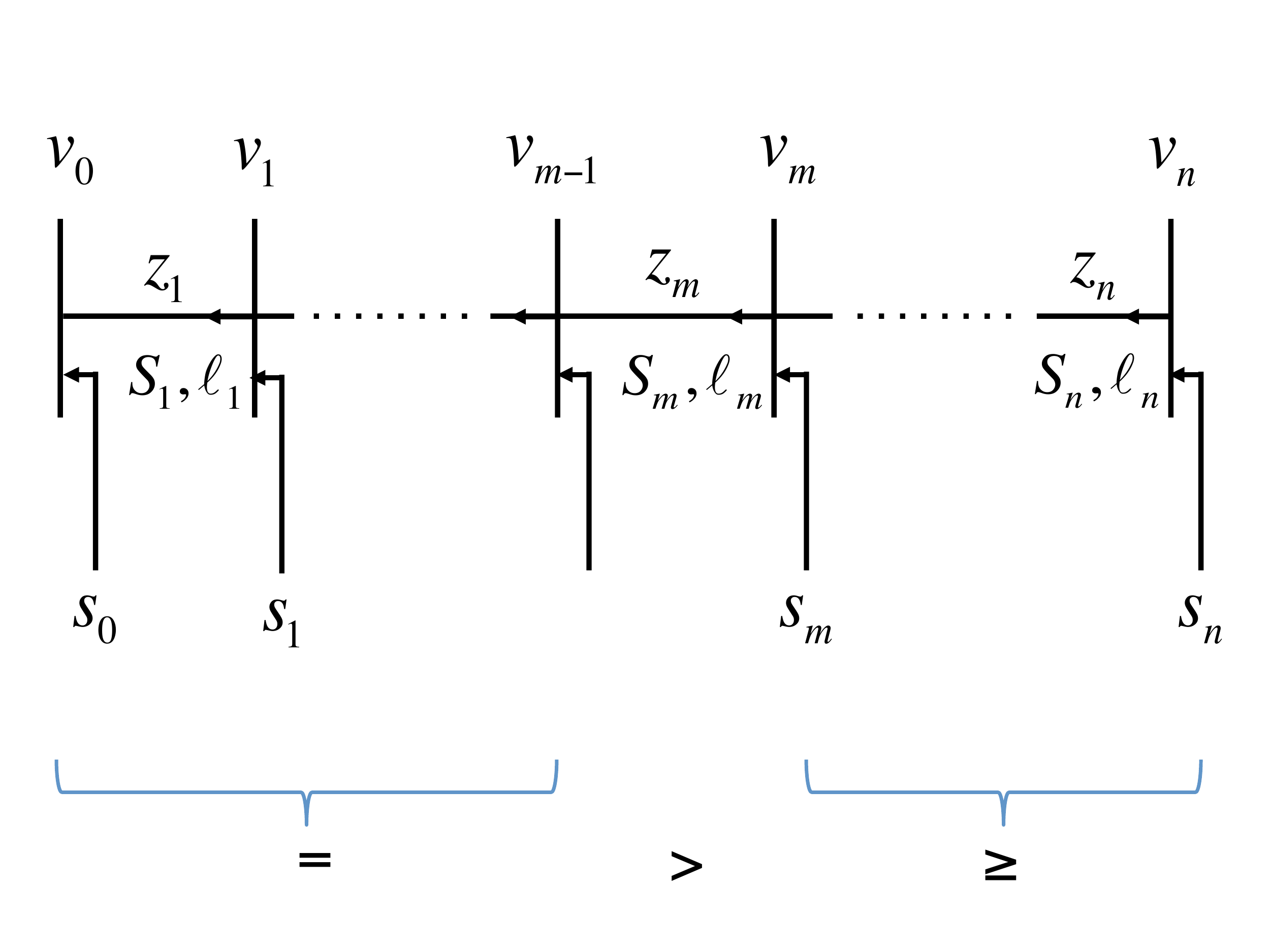}
\caption{Linear network and notations.  Line $m$ in the proof is the line closest to bus 0
	where the inequality in \eqref{eq:lndf.socp} is strict, i.e., \eqref{eq:lndf.socp} holds
	with equality at lines $j=1, \dots, m-1$, strict inequality at line $m$, and inequality 
	at lines $j=m+1, \dots, n$.
}
\label{fig:LinearNk}
\end{figure}
With this notation the branch flow model \eqref{eq:gdf} is the following recursion:
\begin{subequations}
\bq
{S}_{j-1} & = &  {S}_{j} - z_{j} \ell_{j} + s_{j-1}, \qquad\qquad\quad\! j = 1, \dots, n
\label{eq:lndf.1}
\\
v_{j-1} & = & v_j - 2\, \text{Re} \left(z_{j}^H  S_{j} \right) + |z_{j}|^2 \ell_{j}, \quad j = 1, \dots, n
\label{eq:lndf.2}
\\
v_j \ell_{j}  & = & |S_{j}|^2, \qquad\qquad\qquad\qquad\quad\ j = 1, \dots, n
\label{eq:lndf.3}
\\
S_n & = & s_n, \qquad  S_0 := 0
\label{eq:lndf.4}
\eq
\label{eq:lndf}
\end{subequations}
where $v_0$ is given.    The SOCP relaxation of \eqref{eq:lndf.3} is:
\bq
v_j \ell_{j}  & \geq & |S_{j}|^2, \qquad j = 1, \dots, n
\label{eq:lndf.socp}
\eq
OPF on the linear network then becomes ($s_0$ is unconstrained by assumption B1):
\\
\noindent
\textbf{OPF:}
\begin{subequations}
\bq
 \underset{x} {\text{min}}   & &  C(x) := \sum_{j=0}^n C_j\left( \text{Re } s_j \right)
\\
 \text{subject to} & &  x \text{ satisfies \eqref{eq:cnstrs} and \eqref{eq:lndf} }
\eq
\label{eq:lnOPF}
\end{subequations}
and its SOCP relaxation becomes:
\\
\noindent
\textbf{OPF-socp:}
\begin{subequations}
\bq
\underset{x} {\text{min}}    & &  C(x)  :=  \sum_{j=0}^n C_j\left( \text{Re } s_j \right)
\nonumber
\\
\text{subject to}  & & x \text{ satisfies  \eqref{eq:cnstrs} and }
				 \text{\eqref{eq:lndf}  with \eqref{eq:lndf.3} replaced by \eqref{eq:lndf.socp}}
\eq
\label{eq:lnSOCP}
\end{subequations}
For the linear network assumption B3 reduces:
\bee
\item [B3': ] $\underline A_j \cdots \underline A_{k} \ z_{k+1} > 0$ for $1 \leq j \leq k < n$.
\eee

Our goal is to prove OPF-socp \eqref{eq:lnSOCP} is exact, i.e., every optimal solution 
of \eqref{eq:lnSOCP} attains equality in \eqref{eq:lndf.socp} and hence is also optimal
for OPF \eqref{eq:lnOPF}.   Suppose
on the contrary that there is an optimal solution $x := (S, \ell, v, s)$ of OPF-socp
\eqref{eq:lnSOCP} that violates \eqref{eq:lndf.3}.
We will construct another feasible point $\hat x := (\hat S, \hat \ell, \hat v, \hat s)$ of 
OPF-socp \eqref{eq:lnSOCP} that has a strictly lower cost than $x$, contradicting 
the optimality of $x$.

Let $m := \min\, \{ j\in N \ |\ v_j \ell_j > |S_j|^2 \}$ be the closest link from bus 0 where \eqref{eq:lndf.3}
is violated; see Figure \ref{fig:LinearNk}.
Pick any $\epsilon_m \in (0, \ell_m - |S_m|^2/v_m]$ and construct $\hat x$ as follows:
\bee
\item $\hat s_j := s_j$ for $j\neq 0$.
\item For $\hat S, \hat \ell, \hat s_0$:
	\bi
	\item For $j = n, \dots, m+1$: $\hat S_j := S_j$ and $\hat \ell_j := \ell_j$.
	\item For $j = m$: $\hat S_m := S_m$ and $\hat \ell_m := \ell_m - \epsilon_m$.
	\item For $j = m-1, \dots, 1$: 
		\bqn
		\hat S_j & := & \hat S_{j+1} - z_{j+1} \hat \ell_{j+1} + \hat s_j
		\\
		\hat \ell_j & :=  & \frac{|\hat S_j|^2}{v_j}
		\eqn
	\item $\hat s_0 := - \hat S_1 + z_1 \hat \ell_1$.
	\ei
\item $\hat v_0 := v_0$.   For $j=1, \dots, n$, 
	\bqn
	\hat v_j & := & \hat v_{j-1} +  2\, \text{Re} \left(z_{j}^H  \hat S_{j} \right) - |z_{j}|^2 \hat \ell_{j}
	\eqn
\eee
Notice that the denomintor in $\hat \ell_j$ is defined to be 
$v_j$, not $\hat v_j$.  This decouples the recursive construction of $(\hat S_j, \hat \ell_j)$
and $\hat v_j$ so that the former propagates from bus $n$ towards bus 1 while the latter
propagates in the opposite direction.   

By construction $\hat x$ satisfies \eqref{eq:lndf.1},  \eqref{eq:lndf.2}, \eqref{eq:lndf.4},
and \eqref{eq:constraints}.   We  only have to prove that $\hat x$ satisfies \eqref{eq:boxv}
and \eqref{eq:lndf.socp}.   Hence the proof of Theorem \ref{thm:bfmsocp.2} is complete
after Lemma \ref{lemma:dv} is established, which asserts that $\hat x$ is feasible and has
a strictly lower cost under assumptions B1, B2, B3'.
\begin{lemma}
\label{lemma:dv}
Under the conditions of Theorem \ref{thm:bfmsocp.2} $\hat x$ satisfies
\bee
\item $C(\hat x) \ < \ C(x)$.
\item $\hat v_j \hat \ell_j \ \geq \  \left| \hat S_j \right|^2$, $j\in N$.
\item $\underline{v_j} \ \leq \ \hat v_j \ \leq  \ \overline v_j$, $j\in N$.
\eee
\end{lemma}

\vspace{0.2in}
To simplify the notation redefine $S_0 := -s_0$ and $\hat S_0 := - \hat s_0$.
Then for $j\in N^+$ define $\Delta S_j := \hat S_j - S_j$ and $\Delta v_j := \hat v_j - v_j$.
The key result that leads to Lemma \ref{lemma:dv} is:
\bqn
\Delta S_j \ \geq \ 0  & \text{ and } & \Delta v_j \ \geq \ 0
\eqn
The first inequality is stated more precisely in Lemma \ref{lemma:dS} and proved 
after the proof of Lemma \ref{lemma:dv}.
\begin{lemma}
\label{lemma:dS}
Suppose $m>1$ and B3' holds.   
Then $\Delta S_j \geq 0$ for $j \in N^+$ with $\hat S_j > S_j$ for $j = 0, \dots, m-1$.
In particular $\hat s_0 < s_0$.
\end{lemma}
We now prove the second inequality together with Lemma \ref{lemma:dv}
assuming Lemma \ref{lemma:dS} holds.
\vspace{0.2in}

\begin{proof}[Proof of Lemma \ref{lemma:dv}]
1) 
If $m=1$ then, by construction, $\hat s_0 = s_0 - z_1 \epsilon_1 < s_0$ since $z_1>0$.
If $m>1$ then $\hat s_0 < s_0$ by Lemma \ref{lemma:dS}.  
Since $\hat s = s$ and $\hat s_0 < s_0$ we have
\bqn
C(\hat x) - C(x) & = & \sum_{j=0}^n \left(  C_j\left( \text{Re }\hat s_j \right) - C_j\left( \text{Re }s_j \right) \right)
\ \, = \ \, C_0\left( \text{Re }\hat s_0 \right) - C_0\left( \text{Re }s_0 \right) 
\ \  < \ \  0
\eqn
as desired, since $C_0$ is strictly increasing.

2) To avoid circular argument we will first prove using Lemma \ref{lemma:dS}
\bq
\hat v_j & \geq & v_j, \qquad\quad j\in N
\label{eq:dv>0}
\eq
We will then use this and Lemma \ref{lemma:dS} to prove $\hat v_j \hat \ell_j \geq |\hat S_j|^2$
for all $j\in N$.  This means that $\hat x$ satisfies 
\eqref{eq:lndf.1}, \eqref{eq:lndf.2}, and \eqref{eq:lndf.socp}.   We can then use 
\cite[Lemma 13]{Low2014a} and assumption B2 to prove 
$\underline{v_j} \ \leq \ \hat v_j \ \leq  \ \overline v_j$, $j\in N$.

To prove \eqref{eq:dv>0}, note that 
both $\hat v$ and $v$ satisfy \eqref{eq:lndf.2} and hence we have, for $j = 1, \dots, n$,
\bq
\Delta v_{j-1} & = & \Delta v_j - 2\, \text{Re}\left( z_j^H \Delta S_j \right) + |z_j|^2 \Delta \ell_j
\qquad
\label{eq:dv}
\eq
where $\Delta \ell_j := \hat \ell_j - \ell_j$.  
From \eqref{eq:lndf.1} we have 
\bqn
z_j \Delta \ell_j & = & \Delta S_j - \Delta S_{j-1} + \Delta s_{j-1}
\eqn
where $\Delta s_0 := \hat s_0 - s_0 < 0$ and $s_{j-1} = 0$ for $j > 1$.
Multiplying both sides by $z_j^H$ and noticing that both sides must be real, we conclude
\bqn
|z_j|^2 \Delta \ell_j & = & \text{Re } \left( z_j^H \Delta S_j - z_j^H \Delta S_{j-1} + z_j^H \Delta s_{j-1} \right)
\eqn
Substituting into \eqref{eq:dv} we have for $j = 1, \dots, n$
\bqn
\Delta v_j - \Delta v_{j-1} & \!\!\! = \!\!\! & \text{Re } z_j^H \Delta S_j
			\ + \ \text{Re }  z_j^H \Delta S_{j-1}  
			 \ - \ \text{Re } z_j^H \Delta s_{j-1}
\eqn
But Lemma \ref{lemma:dS} implies that $\text{Re } z_j^H \Delta S_j = r_j\, \Delta P_j + x_j\, \Delta Q_j \geq 0$.
Similarly every term on the right-hand side is nonnegative and hence 
\bqn
\Delta v_j &  \geq & \Delta v_{j-1} \qquad \text{ for } j = 1, \dots, n
\eqn
implying that $\Delta v_j \geq \Delta v_0 = 0$, proving \eqref{eq:dv>0}.

We now use \eqref{eq:dv>0} to prove the second assertion of the lemma.
By construction, for $j=m+1, \dots, n$,
\bqn
\hat \ell_j & = & \ell_j  \ \geq \  \frac{|S_j|^2}{v_j} \ \geq \ \frac{|\hat S_j|^2}{\hat v_j}
\eqn
as desired, since $\hat S_j = S_j$ and $\hat v_j \geq v_j$.   
Similarly  \eqref{eq:lndf.socp} holds for $\hat x$ for $j = m$ because of the
choice of $\epsilon_m$.   For $j = 1, \dots, m-1$, $\hat v_j \geq v_j$ again implies
\bqn
\hat \ell_j & = & \frac{|\hat S_j|^2}{v_j} \ \geq\  \frac{|\hat S_j|^2}{\hat v_j}
\eqn

3) The relation \eqref{eq:dv>0} means 
\bqn
\hat v_j  & \geq & v_j \ \geq\  \underline v_j, \qquad j\in N
\eqn
Assumption B2 and \cite[Lemma 13]{Low2014a} (see also Remark 6 of \cite{Low2014a})
imply that 
\bqn
\hat v_j & \leq & v_j^{\text{lin}}(s) \  \leq  \ \overline v_j, \qquad j\in N
\eqn
This proves $\hat x$ satisfies  \eqref{eq:boxv} and completes the proof of Lemma \ref{lemma:dv}.
\end{proof}
\vspace{0.2in}


The remainder  of this subsection is devoted to proving the key result Lemma
\ref{lemma:dS}.

\begin{proof}[Proof of Lemma \ref{lemma:dS}]
By construction $\Delta S_j = 0$ for $j=m, \dots, n$.   
To prove $\Delta S_j > 0$ for $j = 0, \dots, m-1$, the key idea is to derive
a recursion on $\Delta S_j$ in terms of the Jacobian matrix $A_j(S_j, v_j)$.
The intuition is that, when the branch current $\ell_m$ is reduced by $\epsilon_m$ to
$\hat \ell_m$, loss on line $m$ is reduced and all upstream branch powers $S_j$ will be 
increased to $\hat S_j$ as a consequence.

This is proved in three steps, of which we now give an informal overview.
 First we derive a recursion \eqref{eq:DS} on $\Delta S_j$.
This motivates a collection of linear dynamical systems $w$ in \eqref{eq:ds1} 
that contains the process $(\Delta S_j$, $j=0, \dots, m-1)$ as a specific trajectory.
Second we construct another collection of linear dynamical systems $\underline w$
in \eqref{eq:ds2} such that assumption B3' implies $\underline w>0$.
Finally we  prove an expression for the process $w - \underline w$ that
shows $w \geq \underline w$ (in Lemmas \ref{lemma:p1p2}, \ref{lemma:dw}, \ref{lemma:dw+}).  
This then implies $\Delta S = w \geq \underline w > 0$.
We now make these steps precise.

Since both $x$ and $\hat x$ satisfy  \eqref{eq:lndf.1} and $\hat s_j = s_j$ 
for all $j\in N$ we have (with the redefined $\Delta S_0 := -(\hat s_0 - s_0)$)
\bq
\Delta {S}_{j-1} & = &  \Delta {S}_{j} - z_{j} \Delta \ell_{j},  \quad j = 1, 2, \dots, n  \qquad
\label{eq:pf.4}
\eq
where $\Delta \ell_j := \hat \ell_j - \ell_j$.  
For $j=1, \dots, m-1$ both $x$ and $\hat x$ satisfy
\eqref{eq:lndf.3}.   For these $j$, fix any $v_j \geq \underline v_j$ and consider $\ell_j := \ell_j(S_j)$ as functions of the real pair $S_j := (P_j, Q_j)$:
\bqn
\ell_j(S_j) & := & \frac{P_j^2 + Q_j^2}{v_j}, \quad j = 1, \dots, m-1
\eqn
whose Jacobian are the row vectors:
\bqn
\frac{\partial \ell_j}{\partial S_j}(S_j) & = & \frac{2}{v_j} [P_j \ \ Q_j] \ = \ \frac{2}{v_j} S_j^T
\eqn
The mean value theorem implies for $j=1, \dots, m-1$
\bqn
\Delta \ell_j & = & \ell_j(\hat S_j) - \ell_j(S_j) \ = \ \frac{\partial \ell_j}{\partial S_j}(\tilde{S}_j) \, \Delta S_j
\eqn
where $\tilde{S}_j := \alpha_j S_j + (1-\alpha_j) \hat S_j$ for some $\alpha_j \in [0, 1]$.
Substituting it into \eqref{eq:pf.4} we obtain the recursion, for $j=1, \dots, m-1$,
\begin{subequations}
\bq
\Delta {S}_{j-1} & = & \tilde A_j \, \Delta S_j
\label{eq:DS1}
\\
\Delta S_{m-1} & = & \epsilon_m \, z_m \ > \ 0
\label{eq:DS2}
\eq
\label{eq:DS}
\end{subequations}
where the $2\times 2$ matrix $\tilde A_j$ is the matrix function $A_j(S_j, v_j)$ 
defined in \eqref{eq:defA} evaluated at $(\tilde S_j, v_j)$:
\bq
\tilde A_j \ := \ A_j (\tilde S_j, v_j) & := &  I - \frac{2}{v_j} z_j \tilde{S}_j^T
\label{eq:defAtilde}
\eq
which depends on $(S_j, \hat S_j)$ through $\tilde{S}_j$.

Note that $\tilde A_j$ and $\Delta S_j$ are {not} independent since both are defined
in terms of $(S_j, \hat S_j)$, and therefore strictly speaking
\eqref{eq:DS} does not specify a \emph{linear} system.  Given an optimal solution $x$ 
of the relaxation OPF-socp and our modified solution $\hat x$, however, the sequence of matrices
$\tilde A_j$, $j=1, \dots, m-1$, are fixed.   We can therefore consider the following collection
of discrete-time linear time-varying systems (one for each $\tau$),
whose state at time $t$ (going backward in time) is $w(t; \tau)$,
when it starts at time $\tau \geq t$ in the initial state $z_{\tau+1}$: 
for each $\tau$ with $0 < \tau < m$,
\begin{subequations}
\bq
w(t-1; \tau)  &\!\!\! = \!\!\! & \tilde A_t\, w(t; \tau), \ \   t = \tau, \tau-1, \dots, 1 \qquad\ \
\\
w(\tau; \tau) &\!\!\! = \!\!\! & z_{\tau+1}
\eq
\label{eq:ds1}
\end{subequations}
Clearly $\Delta S_j = \epsilon_m\, w(j; m-1)$.  Hence, to prove 
$\Delta S_j > 0$, it suffices to prove $w(j; m-1)>0$ for all $j$ with $0 \leq j \leq m-1$.

To this end we compare the system $w(t; \tau)$ with the following collection of linear 
time-variant systems:  for each $\tau$ with $0 < \tau < m$,
\begin{subequations}
\bq
\underline w(t-1; \tau)  &\!\!\! = \!\!\! & \underline A_t\, \underline w(t; \tau), \ \   t = \tau, \tau-1, \dots, 1 \qquad\ \ 
\\
\underline w(\tau; \tau) &\!\!\! = \!\!\! & z_{\tau+1}
\eq
\label{eq:ds2}
\end{subequations}
where $\underline A_t$ is defined in \eqref{eq:defAunder} and reproduced here:
\bq
\underline A_{t} & := &  A_{t} \left( \left[ S_{t}^\text{lin}(\overline s) \right]^+\!\!, \  \underline v_t \right) 
\ \ = \ \ I - \frac{2}{\underline v_t}\, z_{t} 
			\left( \left[ S^{\text{lin}}_{t}(\overline s) \right]^+ \right)^T
\label{eq:defAunder.2}
\eq
Note that $\underline A_t$ are  \emph{independent of} the OPF-socp solution
$x$ and our modified solution $\hat x$.
Then assumption B3' is equivalent to
\bq
\underline w(t; \tau) & > & 0 \quad \text{ for all } 0 \leq t \leq \tau < m
\label{eq:pf.3}
\eq

We now prove, in  Lemmas \ref{lemma:p1p2}, \ref{lemma:dw}, \ref{lemma:dw+}, 
that $w(t; \tau) \geq \underline w(t; \tau)$ and hence B3' implies 
$\Delta S_j = \epsilon_m\, w(j; m-1) \geq \epsilon_m\, \underline w(j; m-1) > 0$,
establishing Lemma \ref{lemma:dS}.

\vspace{0.1in}
\begin{lemma}
\label{lemma:p1p2}
For each $t = m-1, \dots, 1$
\bqn
\tilde A_t - \underline A_t & = &  2 \ z_t \, \delta_t^T
\eqn
for some 2-dimensional vector  $\delta_t \geq 0$.
\end{lemma}
\begin{proof}[Proof of Lemma \ref{lemma:p1p2}]
Fix any $t = m-1, \dots, 1$.  
We have $S_t \leq S_t^\text{lin}(s)$ from \eqref{eq:constraints}.
Even though we have not yet proved $\hat S_t$ is feasible for OPF-socp
we know $\hat S_t$ satisfies \eqref{eq:lndf.1} by construction of $\hat x$.
The same argument as in \cite[Lemma 13(2)]{Low2014a} then shows  
$\hat S_t \leq S_t^\text{lin}(s)$.  Hence $\tilde S_t := \alpha_t S_t + (1 - \alpha_t) \hat S_t$,
$\alpha_t \in [0, 1]$, satisfies $\tilde S_t \leq S_t^\text{lin}(s)$.  Hence
\bq
\tilde S_t \ \leq \ S_t^\text{lin}(s) \ \leq \ S_t^\text{lin}(\overline s) \ \leq \ 
		\left[ S_t^\text{lin}(\overline s) \right]^+
\label{eq:tildeSt}
\eq
Using the definitions of $\tilde A_t$ in \eqref{eq:defAtilde} and 
$\underline A_t$ in \eqref{eq:defAunder.2} we have
$\tilde A_t  - \underline A_t = 2\, z_t \delta_t^T$ where 
\bqn
\delta_t^T & := & \left[ \frac{\left[ P_t^\text{lin}(\overline s) \right]^+}{\underline v_t} - \frac{\tilde P_t}{v_t}  \ \ \ \
				 \frac{\left[ Q_t^\text{lin}(\overline s) \right]^+}{\underline v_t} - \frac{\tilde Q_t}{v_t}  \right]
\eqn
Then \eqref{eq:tildeSt} and $v_t \geq \underline v_t$ impy that $\delta_t \geq 0$.
\end{proof}
\vspace{0.1in}

For each $\tau$ with $0<\tau < m$ define the scalars $a(t; \tau)$ in terms of the 
solution $\underline w(t; \tau)$ of \eqref{eq:ds2} and $\delta_t$ in Lemma \ref{lemma:p1p2}:
\bq
a(t; \tau)  & := & 2\, \delta_t^T \underline w(t; \tau) \ \ > \ \  0
\label{eq:defa}
\eq
\begin{lemma}
\label{lemma:dw}
Fix any $\tau$ with $0< \tau <m$.  For each $t = \tau, \tau-1, \dots, 0$ we have
	\bqn
	w(t; \tau) - \underline w(t; \tau) & = & \sum_{t'=t+1}^\tau a(t'; \tau) \, w(t; t'-1)
	\eqn
\end{lemma}

\begin{proof}[Proof of Lemma \ref{lemma:dw}]
Fix a $\tau$ with $0< \tau <m$.  We now prove the lemma by induction on $t = \tau, \tau-1, \dots, 0$.
The assertion holds for $t = \tau$ since $w(\tau; \tau) - \underline w(\tau; \tau) = 0$.  Suppose it holds
for $t$.   Then for $t-1$ we have from \eqref{eq:ds1} and \eqref{eq:ds2}
\bqn
w(t-1; \tau) - \underline w(t-1; \tau)
& = & \tilde A_t\, w(t; \tau) - \underline A_t\, \underline w(t; \tau)
\\
& = & 
\left( \tilde A_t - \underline A_t \right) \, \underline w(t; \tau)  \ + \  \tilde A_t \left( w(t; \tau) - \underline w(t; \tau) \right)
\\
& = &
a(t; \tau)\, z_t \ + \ \sum_{t'=t+1}^\tau a(t'; \tau) \ \tilde A_t\, w(t; t'-1)
\\
& = &
a(t; \tau)\, z_t \ + \ \sum_{t'=t+1}^\tau a(t'; \tau) \, w(t-1; t'-1)
\\
&  = &
\sum_{t'=t}^\tau a(t'; \tau) \, w(t-1; t'-1)
\eqn
where the first term on the right-hand side of the third equality follows from Lemma \ref{lemma:p1p2} 
and the definition of $a(t; \tau)$ in \eqref{eq:defa}, and the second term from the induction hypothesis.
The last two equalities follow from \eqref{eq:ds1}.
\end{proof}
\vspace{0.1in}

\begin{lemma}
\label{lemma:dw+}
Suppose B3' holds.   Then for each $\tau$ with $0< \tau <m$ and
each $t = \tau, \tau-1, \dots, 0$,
\bq
w(t; \tau) & \geq & \underline w(t; \tau) \ > \ 0
\label{eq:dw+}
\eq
\end{lemma}

\begin{proof}[Proof of Lemma \ref{lemma:dw+}]
We prove the lemma by induction on $(t, \tau)$.   
\bee
\item \emph{Base case:} For each $\tau$ with $0<\tau<m$, \eqref{eq:dw+} holds for $t = \tau$,
	i.e., for $t$ such that $\tau - t = 0$.
\item \emph{Induction hypothesis:} For each $\tau$ with $0<\tau<m$, suppose \eqref{eq:dw+} holds 
	for $t \leq \tau$ such that $0 \leq \tau - t \leq k-1$.
\item \emph{Induction:} We will prove that, for each $\tau$ with $0<\tau<m$, \eqref{eq:dw+} holds 
	for $t \leq \tau$ such that $0 \leq \tau - t \leq k$.    For $t = \tau - k$ we have from Lemma
	\ref{lemma:dw}
	\bqn
	w(t; \tau) - \underline w(t; \tau) & = & \sum_{t'=t+1}^\tau a(t'; \tau) \, w(t; t'-1)
	\eqn
	But each $w(t; t'-1)$ in the summands satisfies $w(t; t'-1) \geq \underline w(t; t'-1)$ by the
	induction hypothesis.  Hence, since $a(t'; \tau)>0$, 
	\bqn
	w(t; \tau) - \underline w(t; \tau) & \geq & \sum_{t'=t+1}^\tau a(t'; \tau) \, \underline w(t; t'-1)
	\ \ > \ \ 0
	\eqn
	where the last inequality follows from \eqref{eq:pf.3} and \eqref{eq:defa}.
\eee
This completes our induction proof.
\end{proof}
\vspace{0.2in}

Lemma \ref{lemma:dw+} implies, for $j=0, \dots, m-1$, 
$\Delta S_j = \epsilon_m\, w(j; m-1) > 0$.  
This completes the proof of Lemma \ref{lemma:dS}.
\end{proof}

This completes the proof of Theorem \ref{thm:bfmsocp.2} for the linear network.
For a general tree network the proof is almost identical, except with more cumbersome notations, 
by focusing on a path from the root to a first link over which \eqref{eq:gdf.3} holds with strict inequality; see
\cite{Gan-2014-BFMt-TAC}.
\end{proof}

\subsection{Proof of Theorem \ref{thm:uniqueBFM}: uniqueness of SOCP solution}
The proof is from \cite{Gan-2014-BFMt-TAC}.
\begin{proof}
Suppose $\hat x$ and $\tilde x$ are distinct optimal solutions of the relaxation OPF-socp
\eqref{eq:SOCP-radialBFM}.
Since the feasible set of OPF-socp is convex the point $x := (\hat x + \tilde x)/2$ is also
feasible for OPF-socp.   Since the cost function $C$ is convex and both $\hat x$ and $\tilde x$ 
are optimal for OPF-socp \eqref{eq:SOCP-radialBFM}, $x$ is also optimal for 
\eqref{eq:SOCP-radialBFM}.  The exactness of OPF-socp \eqref{eq:SOCP-radialBFM}
then implies that $x$ attains equality in 
\eqref{eq:gdf.3}.   We now show that $\hat x = \tilde x$.
 
Since $v_j \ell_{jk} = |S_{jk}|^2$ we have for $j\neq 0$
\bqn
\frac{1}{4} ( \hat v_j + \tilde v_j ) (\hat \ell_{jk} + \tilde \ell_{jk})
	& = & \frac{1}{4} \left| \hat S_{jk} + \tilde S_{jk} \right|^2
\eqn
Substituting $\hat v_j \hat \ell_{jk} = |\hat S_{jk}|^2$ and $\tilde v_j \tilde \ell_{jk} = |\tilde S_{jk}|^2$
yeilds
\bq
\hat v_j \tilde \ell_{jk} + \tilde v_j \hat \ell_{jk} 
& = & 2\ \text{Re} \left( \hat S_{jk}^H \tilde S_{jk} \right)
\label{eq:vell.1}
\eq
The right-hand side satisfies
\bq
2\ \text{Re} \left( \hat S_{jk}^H \tilde S_{jk} \right) & \leq & 2 \,  | \tilde S_{jk}| | \hat S_{jk}|
\label{eq:vell.2}
\eq
with equality if and only if $\angle \hat S_{jk} = \angle \tilde S_{jk}$ (mod $2\pi$).
The left-hand side of \eqref{eq:vell.1} is
\bq
\hat v_j \tilde \ell_{jk} + \tilde v_j \hat \ell_{jk}  & = & 
\hat v_j \frac{| \tilde S_{jk}|^2}{\tilde v_j} \, + \, 
\tilde v_j \frac{| \hat S_{jk}|^2}{\hat v_j}
\ \ = \ \ \frac{1}{\eta_j}\left( \eta_j^2 | \tilde S_{jk}|^2 + | \hat S_{jk}|^2 \right)
\ \ \geq \ \    2\,  | \tilde S_{jk}| | \hat S_{jk}|
\label{eq:vell.3}
\eq
with equality if and only if $\eta_j |\tilde S_{jk}| = |\hat S_{jk}|$,
where for $j=1, \dots, n$
\bqn
\eta_j & := & \frac{\hat v_j}{\tilde v_j}
\eqn
Combining \eqref{eq:vell.1}--\eqref{eq:vell.3} implies that equalities are attained
in both \eqref{eq:vell.2} and \eqref{eq:vell.3}.   Hence
\bq
\eta_j \tilde S_{jk} \ = \ \hat S_{jk}
& \text{and} & 
\eta_j \tilde \ell_{jk} \ = \ \hat \ell_{jk}
\label{eq:vell.4}
\eq

Define $\eta_0 := \hat v_0 / \tilde v_0 = 1$.   Then for each line 
$j\rightarrow k \in \tilde E$ we have, using \eqref{eq:gdf.2},
\bqn
\eta_k & = & \frac{\hat v_k}{\tilde v_k}
	\ = \ \frac{\hat v_j - 2\, \text{Re} (z_{jk}^H \hat S_{jk}) + |z_{jk}|^2 \hat \ell_{jk}}
		   {\tilde v_j - 2\, \text{Re} (z_{jk}^H \tilde S_{jk}) + |z_{jk}|^2 \tilde \ell_{jk}}
\\
& = & \frac{\eta_j \left(\tilde v_j - 2\, \text{Re} (z_{jk}^H \tilde S_{jk}) + |z_{jk}|^2 \tilde \ell_{jk} \right)}
		   {\tilde v_j - 2\, \text{Re} (z_{jk}^H \tilde S_{jk}) + |z_{jk}|^2 \tilde \ell_{jk}}
\ = \ \eta_j
\eqn
where the last equality follows from \eqref{eq:vell.4}.
This implies, since the network graph $\tilde G$ is connected, 
that $\eta_j = \eta_0 = 1$ for all $j\in N^+$, i.e.
$\hat v_j = \tilde v_j$, $j\in N^+$.

We have thus shown that $\hat S = \tilde S$, $\hat \ell = \tilde \ell$, $\hat v = \tilde v$, 
and hence, by \eqref{eq:gdf.1}, $\hat s = \tilde s$, i.e., $\hat x = \tilde x$.
This completes the proof.
\end{proof}

\subsection{Proof of Corollary \ref{coro:nonconvex}: hollow feasible set}

\begin{proof}
To prove Corollary \ref{coro:nonconvex}, note that
the optimality of $x$ is only used to ensure that $x$ attains equalities
in \eqref{eq:gdf.3}.   The equalities \eqref{eq:vell.1} hold for
\emph{any} convex combination $x$ of $\hat x$ and $\tilde x$.
Hence the proof of Theorem \ref{thm:uniqueBFM} shows that if $\hat x$ and $\tilde x$ are
distinct solutions of the branch flow model in $\mathbb X$
then no convex combination of $\hat x$ and $\tilde x$ can be in $\mathbb X$,
implying in particular that $\mathbb X$ is nonconvex.
\end{proof}

\subsection{Proof of Theorem \ref{thm:pfOPF}: angle difference}

The proof follows that in \cite{LavaeiTseZhang2012}.
We first prove the case of two buses and then extend it to a tree network.

\subsubsection*{Case 1: two-bus network}

Consider two buses $j$ and $k$ connected by a line with 
admittance $y_{jk} = g_{jk} - \ii b_{jk}$ with $g_{jk}>0, b_{jk}>0$.
Since $p_j = P_{jk}$ and $p_k = P_{kj}$ we will work with $P := (P_{jk}, P_{kj})$.
Now
\begin{subequations}
\bq
\!\!\!\!\!
P_{jk} \ := \ P_{jk}(\theta_{jk}) &\!\!\! := \!\!\! & g_{jk} - g_{jk} \cos\theta_{jk} + b_{jk}\sin\theta_{jk}
\\
\!\!\!\!\!
P_{kj} \ := \ P_{kj}(\theta_{jk}) &\!\!\! := \!\!\! & g_{jk} - g_{jk} \cos\theta_{jk} - b_{jk}\sin\theta_{jk}
\eq
\label{eq:Ptheta1}
\end{subequations}
where $\theta_{jk} := \theta_j - \theta_k$, or in vector form
\bq
P - g_{jk} \textbf{1} & = & 
	A \begin{bmatrix} \cos\theta_{jk} \\ \sin\theta_{jk} \end{bmatrix}
\label{eq:Ptheta2}
\eq
where $\textbf{1} := [1\ 1]^T$ and $A$ is the positive definite matrix:
\bqn
A & := & \begin{bmatrix}
		-g_{jk}  &  b_{jk}   \\   -g_{jk}  &  -b_{jk}
		\end{bmatrix}
\eqn
This is an ellipse that passes through the origin as shown in 
Figure \ref{fig:geometry} since\footnote{Recall 
that an ellipsoid in $\mathbb R^k$ 
(without the interior) are the points $x\in \mathbb R^k$ that satisfy
\bqn
x^T M^{-1} x & = & 1
\eqn
for some positive definite matrix $M > 0$.   The $k$ principal axes are
the $k$ eigenvectors of $M$.   
The expression \eqref{eq:Ptheta2} can be written as
\bq
1 = \left\| \begin{bmatrix}
	\cos\theta_{jk}   \\   \sin\theta_{jk}
	\end{bmatrix}  \right\|^2
& = & 
\hat{P}^T
\begin{bmatrix}
	\frac{1}{b_{jk}^2}  &  0    \\    0  &   \frac{1}{g_{jk}^2}
\end{bmatrix} 
 \hat{P}
\label{eq:defhatP}
\eq
where $\hat{P} \in \mathbb R^2$ is related to $P = (P_{jk}, P_{kj})$ by
\bqn
\begin{bmatrix}
	P_{jk}  \\  P_{kj}
\end{bmatrix} 
& = & 
\sqrt{2}\,
\begin{bmatrix}
	\cos 45^\circ   &   \sin 45^\circ  \\  -\sin 45^\circ   &   \cos 45^\circ
\end{bmatrix}
\cdot \hat{P}
\ + \ 
\begin{bmatrix}
	1  \\   1
\end{bmatrix}
\eqn
This says that $\hat{P}$ defined by \eqref{eq:defhatP}
 is a standard form ellipse centered at the origin
with its major axis of length $2b_{jk}$ on the $x$-axis  and its
minor axis of length $2g_{jk}$ on the $y$-axis.  
$P$ is the ellipse obtained from $\hat{P}$ by scaling
it  by $\sqrt{2}$, rotating it by $-45^\circ$, and shifting its center to $(g_{jk}, g_{jk})$,
as shown in Figure \ref{fig:geometry}.
}
\bq
(P-g_{jk}\textbf{1})^T \left( A A^T \right)^{-1} (P-g_{jk}\textbf{1}) & = & 1
\label{eq:Ptheta3}
\eq
Let $\pi_{jk}^\text{min}$ denote the minimum $P_{jk}(\theta_{jk})$ 
and $\pi_{kj}^\text{min}$ the minimum $P_{kj}(\theta_{jk})$ on the 
ellipse as shown in the figure.
They are attained when $\theta_{jk}$ takes the values
\bqn
\theta_{jk}^{\text{min}, jk} \ := \ -\tan^{-1}\frac{b_{jk}}{g_{jk}} & \text{ and } & 
\theta_{jk}^{\text{min}, kj} \ := \ \tan^{-1}\frac{b_{jk}}{g_{jk}}
\eqn
respectively.  This can be easily checked using \eqref{eq:Ptheta1} and
\bqn
\pi_{jk}^\text{min} & := & \min_{\theta\in [-\pi, \pi]} P_{jk}(\theta_{jk})
\\
\pi_{kj}^\text{min} & := & \min_{\theta\in [-\pi, \pi]} P_{kj}(\theta_{jk})
\eqn
\begin{figure}[htbp]
\centering
	\includegraphics[width=0.45\textwidth]{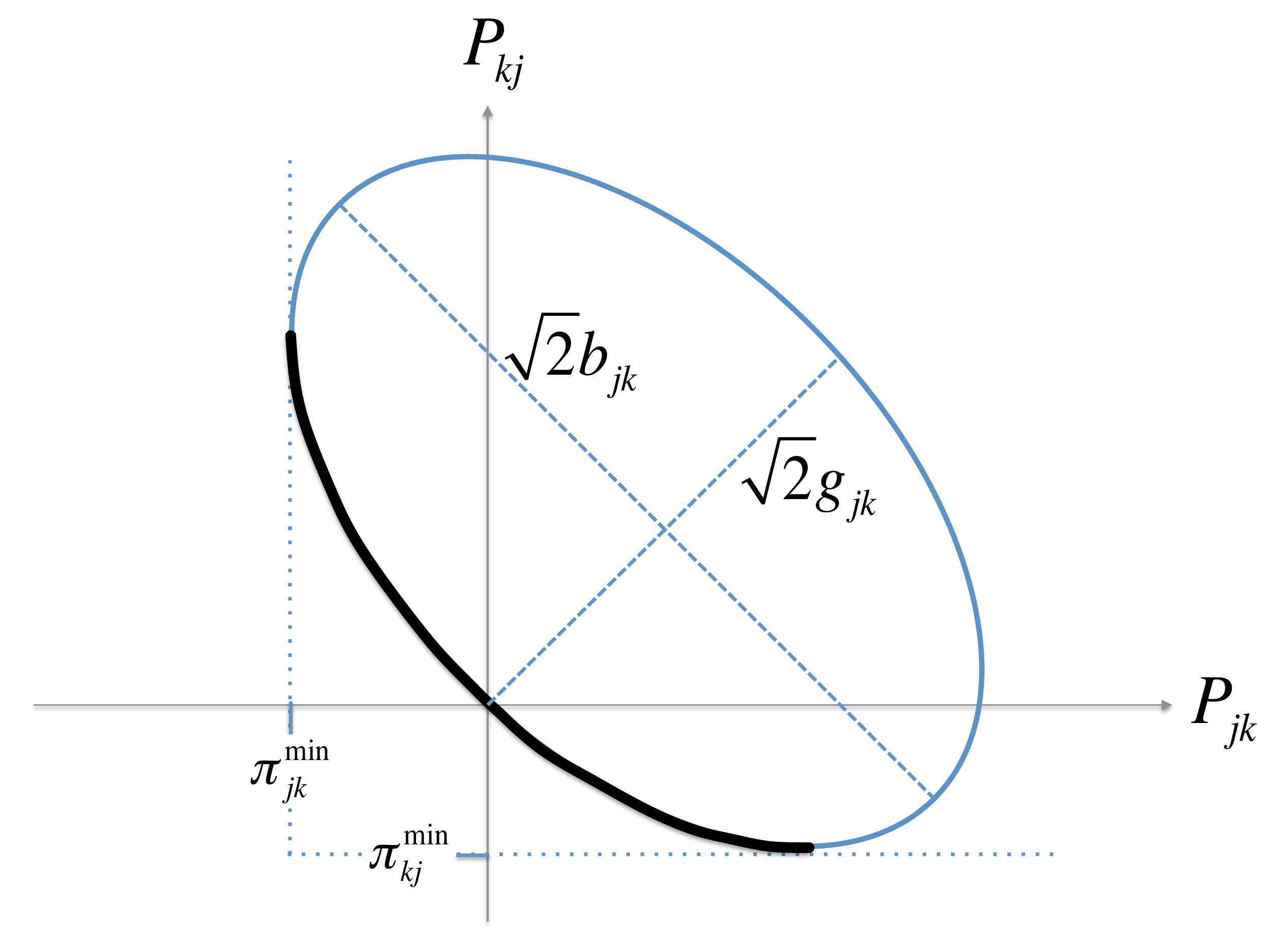}
\caption{The points $P(\theta_{jk}) := (P_{jk}(\theta_{jk}), P_{kj}(\theta_{kj}))$ 
is an ellipse as $\theta_{jk}$ varies in $[-\pi, \pi]$ with
$P=0$ when $\theta_{jk}=0$, 
$P_{jk}=\pi_{jk}^\text{min}$ when $\theta_{jk} = \theta_{jk}^{\text{min}, jk}$, and 
$P_{kj}=\pi_{kj}^\text{min}$ when $\theta_{jk} = \theta_{jk}^{\text{min}, kj}$.
}
\label{fig:geometry}
\end{figure}
The condition $\theta_{jk}^{\text{min}, jk} \leq \theta_{jk} \leq \theta_{jk}^{\text{min}, kj}$
in Theorem \ref{thm:pfOPF} restricts $\mathbb P_\theta$ to the darkened
segment of the ellipse in Figure \ref{fig:geometry}
where $\mathbb P_\theta$ coincides with the Parento front of its convex hull.

Recall the sets
\bqn
\mathbb P_\theta & := & \{\ p \ | \ p = P, P \text{ satisfies \eqref{eq:Ptheta2} for } 
		\underline{\theta}_{jk} \leq \theta_{jk} \leq \overline{\theta}_{jk} \ \}
\eqn
$\mathbb P_p := \{ p \, | \,  \underline p \leq p \leq \overline p\}$, and 
the feasible set $\mathbb P_\theta \cap \mathbb P_p$ of OPF \eqref{eq:pfOPF}.
It is clear from Figure \ref{fig:geometry} that the additional constraint
in $\mathbb P_p$ only restricts the feasible set $\mathbb P_\theta \cap \mathbb P_p$
to  a subset of $\mathbb P_\theta$, but does not change the property that 
the Pareto front of its convex hull coincides with the set itself.
\begin{lemma}
\label{lemma:PF2}
Under condition C1, for the two-bus network,
\bqn
\mathbb P_\theta & = & \mathbb O(\text{conv } \mathbb P_\theta)
\\
\mathbb P_\theta \cap \mathbb P_p  & = & 
	\mathbb O(\text{conv}( \mathbb P_\theta \cap \mathbb P_p))
\eqn
\end{lemma}

Lemma \ref{lemma:PF2} implies that the minimizers of any increasing function of $p$
over the convex set $\text{conv}( \mathbb P_\theta \cap \mathbb P_p)$
will lie in the nonconvex subset $\mathbb P_\theta \cap \mathbb P_p$ under condition C1.
The set $\text{conv}( \mathbb P_\theta \cap \mathbb P_p)$ however does not
have a simple algebraic representation.   Instead the superset
$\text{conv}( \mathbb P_\theta) \cap \mathbb P_p$, which is the feasible set of OPF-socp
\eqref{eq:pfOPF-socp}, is more amenable to computation.
These two sets are illustrated in Figure \ref{fig:geosocp}(a). 
\begin{figure}[htbp]
\centering
\subfigure [$\text{conv}( \mathbb P_\theta \cap \mathbb P_p)
	\subseteq \text{conv}( \mathbb P_\theta) \cap \mathbb P_p$]{
	\includegraphics[width=0.7\textwidth]{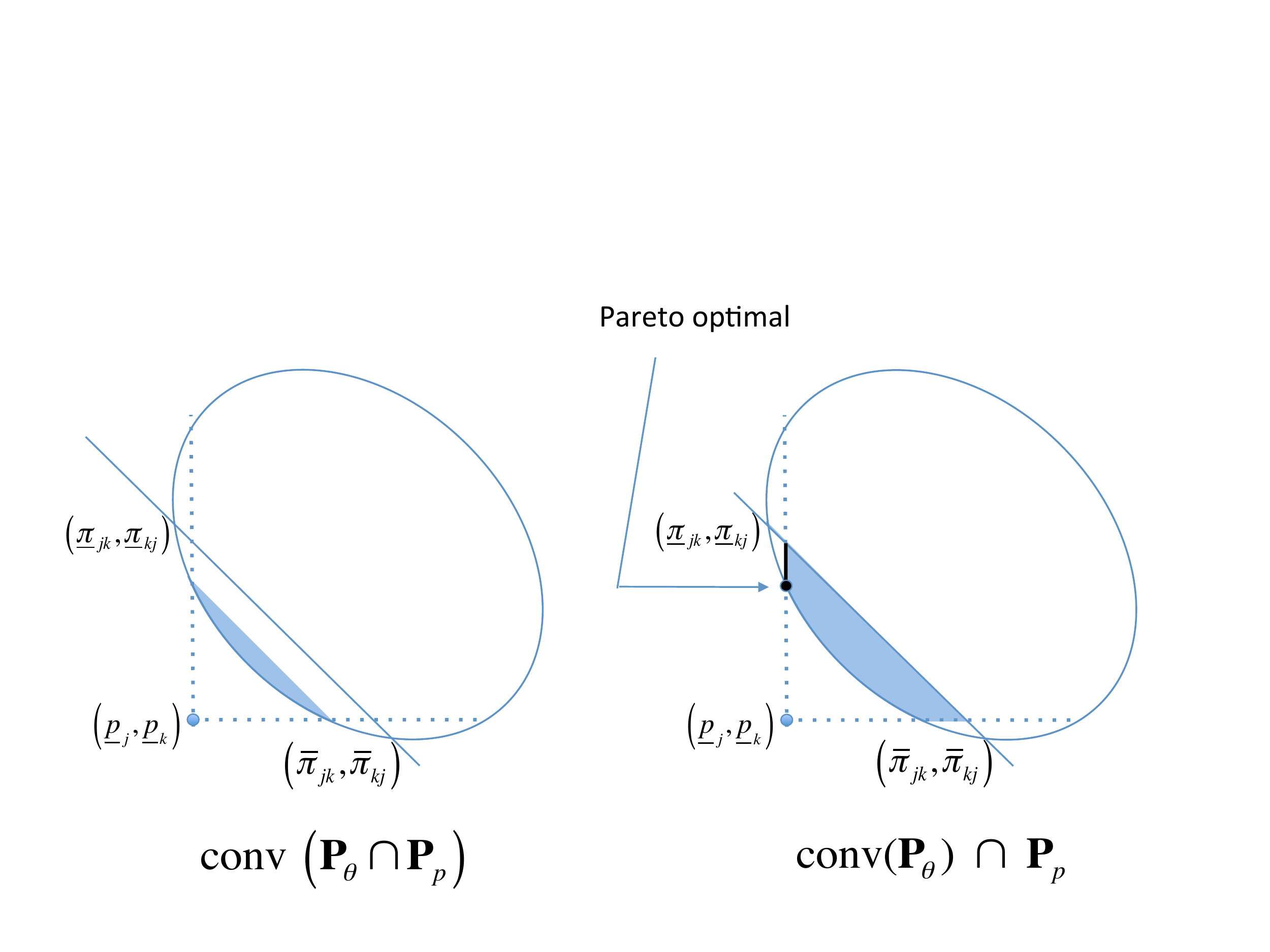}
	}
\\
\subfigure [conv$(\mathbb P_\theta)$] {
	\includegraphics[width=0.6\textwidth]{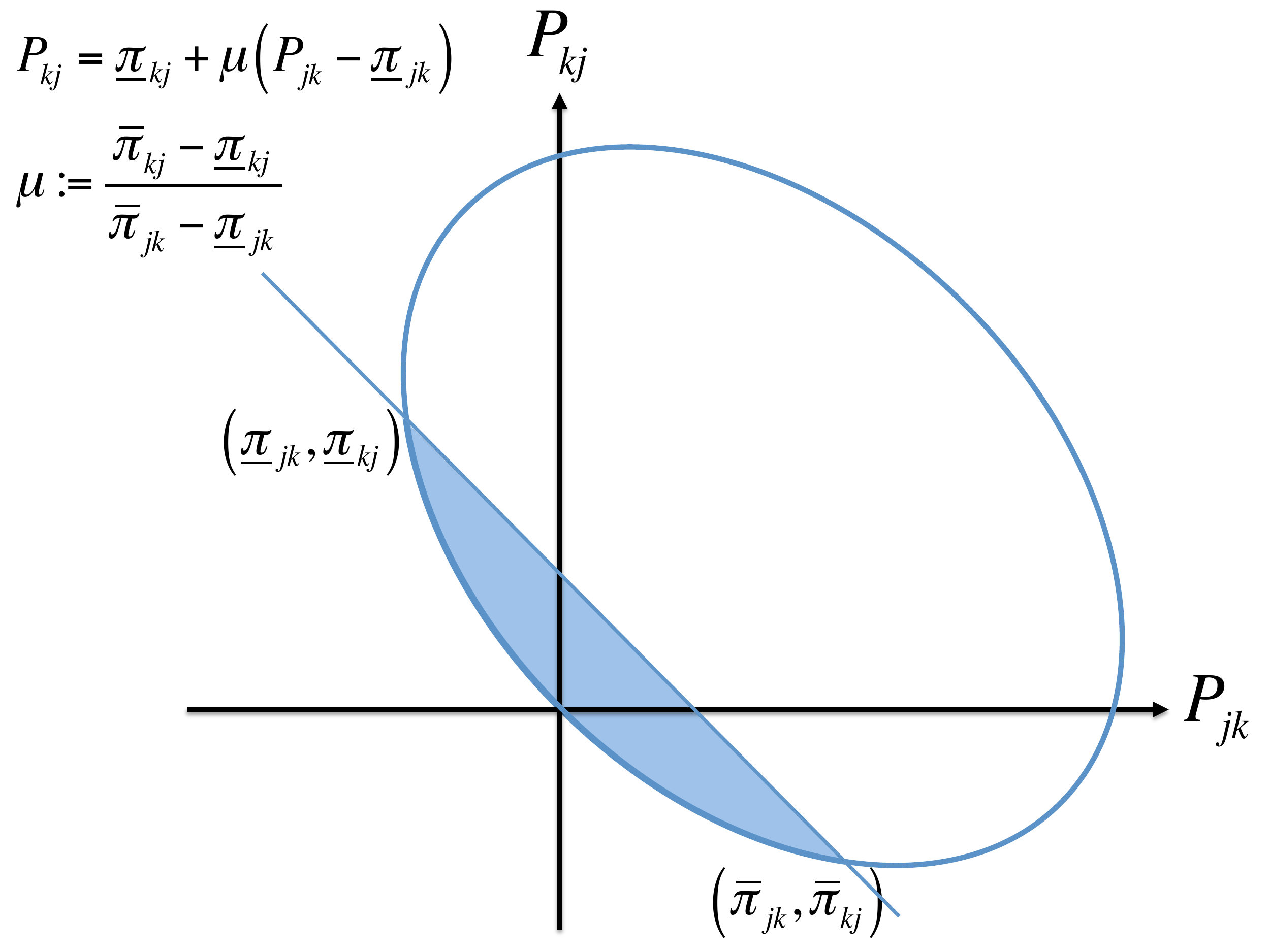}
	}
\caption{The feasible set $\text{conv}( \mathbb P_\theta) \cap \mathbb P_p$
of the SOCP relaxation for the 2-bus network is the intersection of a second-order
cone with an affine set.}
\label{fig:geosocp}
\end{figure}
The set $\text{conv}( \mathbb P_\theta) \cap \mathbb P_p$ 
has two important properties: under C1,
\bee
\item[(i)] It has the same Pareto front, i.e., 
$\mathbb O (\text{conv}( \mathbb P_\theta) \cap \mathbb P_p) 
	= \mathbb O (\text{conv}( \mathbb P_\theta \cap \mathbb P_p))
= \mathbb P_\theta \cap \mathbb P_p$ by Lemma \ref{lemma:PF2}.

\item[(ii)] It is the intersection of a second-order cone with an affine set.
\eee
\begin{remark}
\label{remark:defPO}
Strictly speaking, $\mathbb O (\text{conv}( \mathbb P_\theta) \cap \mathbb P_p) 
	\supseteq \mathbb O (\text{conv}( \mathbb P_\theta \cap \mathbb P_p))$
in (i) because when $P_{jk}=\underline p_j$,
the Pareto optimal points $(\underline p_j, P_{kj})$ are nonunique
where $P_{kj}$ can take any value on the darkened segment of the line
$P_{jk}=\underline p_j$ in Figure \ref{fig:geosocp}(a).
In this case we will regard only the point of intersection of 
$P_{jk}=\underline p_j$ and the ellipse as the unique Pareto optimal
point in $\mathbb O (\text{conv}( \mathbb P_\theta) \cap \mathbb P_p)$ and
ignore the other points since they are not feasible (do not lie on the ellipse).
The case of $P_{kj}=\underline p_k$ is handled similarly.
Then $\mathbb O (\text{conv}( \mathbb P_\theta) \cap \mathbb P_p) 
	= \mathbb O (\text{conv}( \mathbb P_\theta \cap \mathbb P_p))$
under this interpretation of Pareto optimal points.
This corresponds to, for our purposes, defining Pareto optimal points as
the set of minimizers of: 
\bqn
\min_{P\in\text{conv}(\mathbb P_\theta)} \ \ c^T P
\eqn
for some $c>0$, as opposed to nonzero $c\geq 0$
$(P_{jk}=\underline p_j$ corresponds to $c = (c_1, 0), c_1>0$).
This is why we require in condition C1 that $C(p)$ is \emph{strictly
increasing} in each $p_j$.
We will henceforth use this characterization of Pareto optimal points
unless otherwise specified.
\end{remark}

To see (ii), we use \eqref{eq:Ptheta3} to specify the set 
conv$(\mathbb P_\theta)$  as the intersection of a second-order cone 
with an affine set (see Figure \ref{fig:geosocp}(b)), as follows:\footnote{Note that
the first equation is a second-order cone
$t^2 \geq (P-g_{jk}\textbf{1})^T ( A A^T)^{-1} (P-g_{jk}\textbf{1})$ intersecting
with $t=1$.}
\bqn
1 & \geq & (P-g_{jk}\textbf{1})^T ( A A^T )^{-1} (P-g_{jk}\textbf{1}) 
\\
P_{kj} & \leq & \underline{\pi}_{kj} + 
\frac{\overline{\pi}_{kj} - \underline{\pi}_{kj}} {\overline{\pi}_{jk} - \underline{\pi}_{jk}}\, 
	(P_{jk} - \underline{\pi}_{jk}
\eqn
where $(\underline{\pi}_{jk}, \underline{\pi}_{kj}) := 
(P_{jk}(\underline{\theta}_{jk}), P_{kj}(\underline{\theta}_{jk}))$ and
$(\overline{\pi}_{jk}, \overline{\pi}_{kj}) := 
(P_{jk}(\overline{\theta}_{jk}), P_{kj}(\overline{\theta}_{jk}))$.
This implies that the problem \eqref{eq:pfOPF-socp} is indeed an SOCP for
the two-bus case.

The SOCP relaxation of OPF \eqref{eq:pfOPF} enlarges the feasible set
$\mathbb P_\theta \cap \mathbb P_p$ to the convex superset 
$\text{conv}(\mathbb P_\theta) \cap \mathbb P_p$.   
Under condition C1,
\emph{every} minimizer lies in its Pareto front and hence, by property (i),
in the original nonconvex feasible set $\mathbb P_\theta \cap \mathbb P_p$.
We have hence proved Theorem \ref{thm:pfOPF} for the two-bus case.
\vspace{0.2in}

\subsubsection*{Case 2: tree network}

Let $\mathbb F_\theta^{jk}$ denote the set of branch power flows on 
each line $(j,k)\in E$:
\bqn
\mathbb F_\theta^{jk} & := & \{\ (P_{jk}, P_{kj}) \ | \ 
	(P_{jk}, P_{kj}) \text{ satisfies \eqref{eq:Ptheta2} for } 
	\underline{\theta}_{jk} \leq \theta_{jk} \leq \overline{\theta}_{jk} \ \}
\eqn
Since the network is a tree, the set $\mathbb F_\theta$ of branch 
power flows on all lines is simply the product set:
\bq
\mathbb F_\theta & := &  \{ P:= (P_{jk}, P_{kj}, (j,k)\in E) \ | \ 
	P \text{ satisfies \eqref{eq:Ptheta2} for } 
	\underline{\theta}_{jk} \leq \theta_{jk} \leq \overline{\theta}_{jk}, (j,k)\in E \ \}
\nonumber
\\
& =  &  \underset{(j,k)\in E} {\prod} \ \, \mathbb F^{jk}_\theta
\label{eq:prodF}
\eq
because given any $(\theta_{jk}, (j,k)\in E)$ there is always a (unique)
$(\theta_j, j\in N^+)$ that satisfies $\theta_{jk} = \theta_j - \theta_k$.
(This is equivalent to the cycle condition \eqref{eq:cyclecond.1}.)
If the network has cycles then this is not possible for some vectors
$(\theta_{jk}, (j,k)\in E)$
and $\mathbb F_\theta$ is no longer a product set of $\mathbb F_\theta^{jk}$.

Since the power injections $p$ are related to the branch flows $P$
by $p_j = \sum_{k: j\sim k} P_{jk}$, the injection region \eqref{eq:defP} 
is a linear transformation of $\mathbb F_\theta$:
\bqn
\mathbb P_\theta & = & A \mathbb F_\theta
\eqn
for some $(n+1)\times 2m$ dimensional matrix $A$.
{Matrix $A$ has full row rank and it can 
be argued that there is a bijection between $P_\theta$ and $F_\theta$ using
the fact that the graph is a tree \cite{LavaeiTseZhang2012}.}
We can therefore freely work with either $p\in \mathbb P_\theta$ or the corresponding
$P\in \mathbb F_\theta$.

To prove the second assertion of Theorem \ref{thm:pfOPF}, note that the argument
for the two-bus case shows that conv$(\mathbb F_\theta^{jk})$, $(j,k)\in E$, is the intersection of a
second-order cone with an affine set.  This, together with Lemma \ref{lemma:misc} below,
the fact that $\mathbb F_\theta$ is a direct product of $\mathbb F_\theta^{jk}$ and 
the fact that $A$ is of full rank, imply that conv$(\mathbb P_\theta) \cap \mathbb P_p$ is
the intersection of a second-order cone with an affine set.   Hence \eqref{eq:pfOPF-socp}
is indeed an SOCP for a tree network.   Therefore it suffices to prove the first assertion of
Theorem \ref{thm:pfOPF}:
\bq
\mathbb P_\theta \cap \mathbb P_p & = &
\mathbb O	( \text{conv}(\mathbb P_\theta) \cap \mathbb P_p )
\label{eq:P=O}
\eq
because it implies that, under C1, \emph{every} minimizer of OPF-socp \eqref{eq:pfOPF-socp} 
lies in its Pareto front and hence is feasible and optimal for OPF \eqref{eq:pfOPF}
(see also Remark \ref{remark:defPO}).  Hence SOCP relaxation is exact.

We are hence left to prove \eqref{eq:P=O}.   Half of the equality follows from the following 
simple properties of Pareto front and convex hull.
\begin{lemma}
\label{lemma:misc}
Let $\mathbb B, \mathbb C \subseteq \mathbb R^k$ be arbitrary sets,  
$\mathbb D:= \{x\in \mathbb R^k | Mx\leq c\}$ be an affine set, and
$M$ a matrix and $b$ a vector of appropriate dimensions.  
\bee
\item[(1)] conv$(M\, \mathbb B) = M\, \text{conv}(\mathbb B)$
	and conv$(\mathbb B\times \mathbb C) = 
	\text{conv}(\mathbb B) \times \text{conv}(\mathbb C)$.
\item[(2)] Suppose $\mathbb B$ and $\mathbb C$ are convex
	and a point is Pareto optimal over a set if and only if it minimizes
	$c^T x$ over the set for some $c>0$.\footnote{In general,
	a point is Pareto optimal over a convex set if and only if it minimizes $c^Tx$ 
	over the set for some nonzero $c\geq 0$, as opposed to $c>0$.  In that case,
	$\mathbb O(\mathbb B\times \mathbb C) \supseteq 
	\mathbb O(\mathbb B) \times \mathbb O(\mathbb C)$; c.f. Remark \ref{remark:defPO}.
	}	
	Then $\mathbb O(M\, \mathbb B) = M\, \mathbb O(\mathbb B)$ and
	$\mathbb O(\mathbb B\times \mathbb C) = 
	\mathbb O(\mathbb B) \times \mathbb O(\mathbb C)$.
		
\item[(3)] If $\mathbb B = \mathbb O(\text{conv }\mathbb B)$ then
	$\mathbb B \cap \mathbb D \subseteq \mathbb O (\text{conv}(\mathbb B) \cap \mathbb D)$.
\eee
\end{lemma}
For ease of reference we prove Lemma \ref{lemma:misc} below.

The next lemma says that the feasible set of OPF \eqref{eq:pfOPF} is
a subset of the feasible set of its SOCP relaxation \eqref{eq:pfOPF-socp}.
\begin{lemma}
\label{lemma:subset}
$\mathbb P_\theta \cap \mathbb P_p \ \subseteq \ \mathbb O (
			\text{conv}(\mathbb P_\theta) \cap \mathbb P_p )$.
\end{lemma}
\begin{proof}[Proof of Lemma \ref{lemma:subset}]
We have
\bqn
\mathbb O	( \text{conv}(\mathbb P_\theta) )
& = & 
\mathbb O	( \text{conv}(A\, \mathbb F_\theta) )  
\\
& = &
\mathbb O	( A\ \text{conv}(\mathbb F_\theta)) 
\\
& = & 
\mathbb O\left( A \underset{(j,k)\in E} {\prod} 
	\text{conv}\left(\mathbb F^{jk}_\theta\right) \right)  
\\
& = & 
A \underset{(j,k)\in E} {\prod} 
	\mathbb O \left( \text{conv} \left(\mathbb F^{jk}_\theta\right) \right) 
\\
& = & 
A \underset{(j,k)\in E} {\prod} \mathbb F^{jk}_\theta 
\\
& = &
\mathbb P_\theta  
\eqn
where the second equality follows from Lemma \ref{lemma:misc}(1),
the third equality follows from \eqref{eq:prodF} and Lemma \ref{lemma:misc}(1),
the fourth equality follows from Lemma \ref{lemma:misc}(2), 
the fifth equality follows from Lemma \ref{lemma:PF2} where
$\mathbb F_\theta^{jk}$ plays the role of $\mathbb P_\theta$, 
and the last equality follows from \eqref{eq:prodF} and $\mathbb P_\theta = A\mathbb F_\theta$.
Lemma \ref{lemma:misc}(3) then implies the lemma.
\end{proof}

Lemma \ref{lemma:subset}  means that every optimal solution of OPF \eqref{eq:pfOPF} is an 
optimal solution of its SOCP \eqref{eq:pfOPF-socp}.  For exactness
of OPF-socp \eqref{eq:pfOPF-socp} we need the converse to hold as well.
The remainder of the proof is to show this is indeed true, proving \eqref{eq:P=O}.
\begin{lemma}
\label{lemma:supset}
$\mathbb P_\theta \cap \mathbb P_p \ \supseteq \ \mathbb O (
			\text{conv}(\mathbb P_\theta) \cap \mathbb P_p )$.
\end{lemma}

The proof of Lemma \ref{lemma:subset} shows that 
$\mathbb P_\theta = \mathbb O	( \text{conv}(\mathbb P_\theta) )$, so the
converse of Lemma \ref{lemma:misc}(3) would imply Lemma \ref{lemma:supset}.
Figure \ref{fig:PO} and the explanation in its caption, however, illustrate why the 
converse of Lemma \ref{lemma:misc}(3) generally does not hold.
\begin{figure}[htbp]
\centering
	\includegraphics[width=0.6\textwidth]{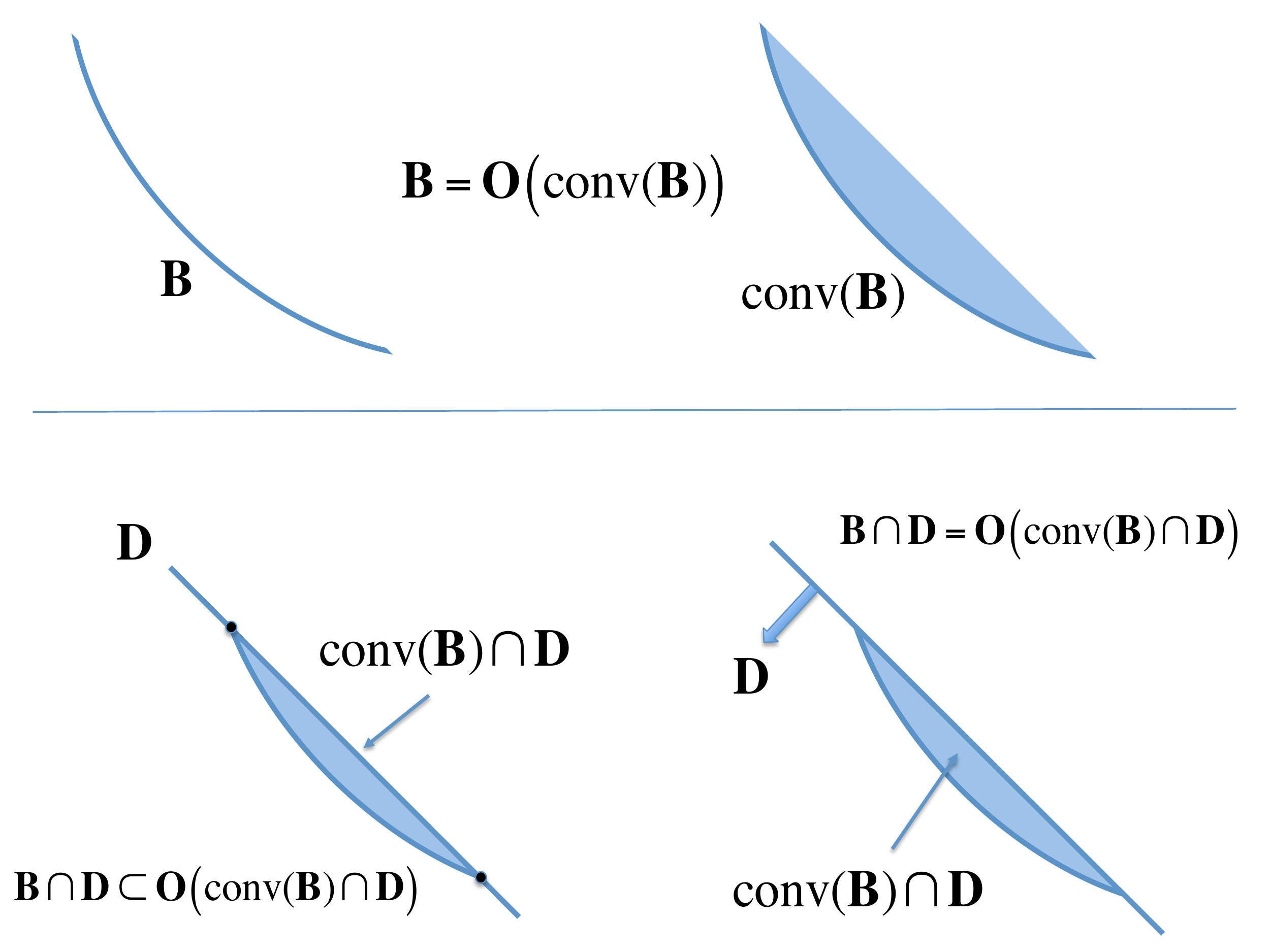}
\caption{The upper panel shows a set $\mathbb B$ and its convex hull
$\text{conv}(\mathbb B)$ with the property that 
$\mathbb B = \mathbb O(\text{conv}(\mathbb B))$.
The lower panel shows two affine sets $\mathbb D$.  On the left $\mathbb D$
is a hyperplane; $\mathbb B \cap \mathbb D$ consists of two intersection points 
and is a strict subset of $\mathbb O(\text{conv}(\mathbb B) \cap \mathbb D)$.
On the right $\mathbb D$ is a halfspace and 
$\mathbb B \cap \mathbb D = \mathbb O(\text{conv}(\mathbb B) \cap \mathbb D)$.
}
\label{fig:PO}
\end{figure}
To prove Lemma \ref{lemma:supset} we need to exploit the structure of 
$\mathbb P_\theta, \mathbb F_\theta, \mathbb P_p$.

\begin{proof}[Proof of Lemma \ref{lemma:supset}]

Take any point $p \in \mathbb O( \text{conv}(\mathbb P_\theta) \cap \mathbb P_p )$.
We now show that $p \in \mathbb P_\theta \cap P_p$.
By definition of Pareto optimality, $p$ is a minimizer of 
\bqn
\min_{\hat p\in \text{conv}(\mathbb P_\theta)} \ c^T \hat p
&\  \text{ subject to }\  & \underline{p}\ \leq  \ \hat p \  \leq \  \overline{p}
\eqn
for some $c> 0$.
This minimization is equivalent to:
\bqn
\min_{\alpha_j, \hat p_j}  &  & c^T \sum_j\  \alpha_j \hat p_j
\\
\text{subject to} & &
\alpha_j\ \geq \ 0, \ \ \sum_j \alpha_j \ = \ 1, \ \ \hat p_j  \in \mathbb P_\theta
\\
& & \underline{p} \ \, \leq \ \, \sum_j \alpha_j \, \hat p_j \  \, \leq \  \, \overline{p}
\eqn
We can uniquely express $p$ and $p_j$ in terms of branch flows
in $\mathbb F_\theta$,
$p = AP$ and $\hat p_j = A\hat P_j$.  Then $P$ is
in $\text{conv}(\mathbb F_\theta)$ and a minimizer of
\bqn
\min_{\hat P \in \text{conv}(\mathbb F_\theta)} \ c^T A \hat P 
&\  \text{ subject to } \  & \underline{p} \ \leq  \  A \hat P \  \leq\  \overline{p}
\eqn
It suffices to prove that $P\in \mathbb F_\theta$, which then implies
that $p = AP \in \mathbb P_\theta \cap P_p$.

The Slater's condition holds for OPF \eqref{eq:pfOPF}.  
By strong duality there exist Lagrange
multipliers $\overline{\lambda} \geq 0$ and $\underline{\lambda} \geq 0$
such that $P$ is a minimizer of the Lagrangian:
\bq
\min_{\hat P \in \text{conv}(\mathbb F_\theta)} & \!\!\!\! &
	\left( c^T + \overline{\lambda}^T - \underline{\lambda}^T  \right) A \hat P 
	\ - \ \overline{\lambda}^T\overline p \ + \ \underline\lambda^T \underline p
\label{eq:saddle}
\eq
If $\overline{c}:= c^T + \overline{\lambda}^T - \underline{\lambda}^T \geq 0$ 
and is nonzero then
$P\in \mathbb F_\theta$ since 
$\mathbb O(\text{conv}(\mathbb F_\theta)) = \mathbb F_\theta$.\footnote{The
minimization \eqref{eq:saddle} does not have the problem discussed in Remark \ref{remark:defPO}
because the feasible set is conv$(\mathbb F_\theta)$, not conv$(\mathbb F_\theta) \cap P_p$,
and hence $\mathbb O(\text{conv}(\mathbb F_\theta)) = \mathbb F_\theta$ for nonzero $c\geq 0$.}
We are left to deal with the case where either $\overline{c}=0$
(in which case every point in conv$(\mathbb F_\theta)$ is Pareto optimal)
or there exists a $j$ such that $\overline c_j< 0$.

Since $\mathbb F_\theta = \prod_{(j,k)\in E} \mathbb F_\theta^{jk}$, 
$P\in \mathbb F_\theta$ if and only if $(P_{jk}, P_{kj})\in \mathbb F_\theta^{jk}$.
Moreover \eqref{eq:saddle} becomes separable by Lemma \ref{lemma:misc}(1):
\bqn
& & 
\min_{\hat P \in \text{conv}(\mathbb F_\theta)} \ 
	\sum_{j\in N^+} \overline{c}_j \sum_{k: j\sim k} \hat P_{jk}
\ \ \equiv \ \
\sum_{(j,k)\in E}\ \  \min_{(\hat P_{jk}, \hat P_{kj}) \in \text{conv}(\mathbb F_\theta^{jk})} \
\left( \overline{c}_j \hat P_{jk} + \overline c_k \hat P_{kj} \right)
\eqn
This reduces the problem to the two-bus case:
\bqn
\min_{(\hat P_{jk}, \hat P_{kj}) \in \text{conv}(\mathbb F_\theta^{jk})} \
\left( \overline{c}_j \hat P_{jk} + \overline c_k \hat P_{kj} \right)
\eqn
If either $\overline c_j>0$ or $\overline c_k>0$ then it can be seen from
Figure \ref{fig:geosocp}(b) that the minimizer $(P_{jk}, P_{kj})$ is in
$\mathbb F_\theta^{jk}$.   
We now show that $(P_{jk}, P_{kj}) \in \mathbb F_\theta^{jk}$ even when
both $\overline c_j \leq 0$ and $\overline c_k \leq 0$.

Since $c>0$, any node $i$ with $\overline c_i \leq 0$ has  
$\underline\lambda_i >0$ and hence $p_i = \underline p_i$.
  Consider the biggest
subtree $T$ that contains link $(j,k)$ in which every node $i$ has 
$\overline c_i\leq 0$ and $p_i = \underline p_i$.  
Call a node $l$ in the subtree $T$ a \emph{boundary node}
if it is a leaf or connected to another node $l'$ outside $T$ where
$\overline c_{l'}>0$.   Without loss of generality, take one of the boundary nodes
as the root of the network graph and assume this is node 0.   
For each line $(l,i)$ in the graph, node $i$ is called
the \emph{parent} of node $l$ if $i$ lies in the unique path from $l$ to the root
node 0.
\begin{lemma}
\label{lemma:-vec}
$(P_{li}, P_{il}) \in \mathbb F_\theta^{li}$ for every link $(l,i)$ in the subtree $T$.
\end{lemma}
\begin{proof}[Proof of Lemma \ref{lemma:-vec}]
Consider any $\tilde P$ that satisfies 
$(\tilde P_{li}, \tilde P_{il}) \in \mathbb F_\theta^{li}$ for every link
$(l,i)\in T$ and $\underline p \leq \tilde p = A\tilde P \leq \overline p$.
We will first prove that, for every link $(l,i)\in T$, 
\bq
P_{li}\  \leq \  \tilde P_{li} & \text{and} &
P_{il}\  \geq \  \tilde P_{il} 
\label{eq:PtildeP}
\eq
We then use this to prove that 
$(P_{li}, P_{il}) = (\tilde P_{li}, \tilde P_{il}) \in \mathbb F_\theta^{li}$.

Consider first a boundary node $l$.  If $l$ is a leaf node then, 
since $\overline c_l\leq 0$,
$P_{li} = p_l = \underline p_l \leq \tilde p_l = \tilde P_{li}$.  Then, since 
$P_{li}\in \text{conv}(\mathbb F_\theta^{li})$ and $\tilde P_{li}\in \mathbb F_\theta^{li}$,
we have $P_{il} \geq \tilde P_{il}$; see Figure \ref{fig:-vec}(a).
Otherwise let $l'$ outside $T$ be a neighbor of $l$.  Since $\overline c_l\leq 0$
but $\overline c_{l'} > 0$, the minimization of 
$\overline c_{l'} \hat P_{l'l} + \overline c_l \hat P_{ll'}$ over conv$(\mathbb F_\theta^{ll'})$ means $P_{ll'} = \underline\pi_{ll'}$; see Figure \ref{fig:-vec}(b).
Hence  $P_{ll'} = \underline\pi_{ll'} \geq \tilde P_{ll'}$.
This holds for all neighbors $l'$ of $l$.  Hence
\bqn
P_{li} \ = \ p_l - \sum_{l'} P_{ll'} & \leq & \tilde p_l - \sum_{l'} \tilde P_{ll'} \ = \  \tilde P_{li}
\eqn
where $l'$ ranges over all neighbors (outside $T$) of $l$ except its parent $i$
in $T$.   From the region of possible values for $(P_{li}, P_{il})$ in 
Figure \ref{fig:-vec}(c), we conclude that $P_{il} \geq \tilde P_{il}$.
 Hence the claim is true for all links $(l,i)$ where $l$ is a boundary node.
\begin{figure}[htbp]
\centering
	\includegraphics[width=0.8\textwidth]{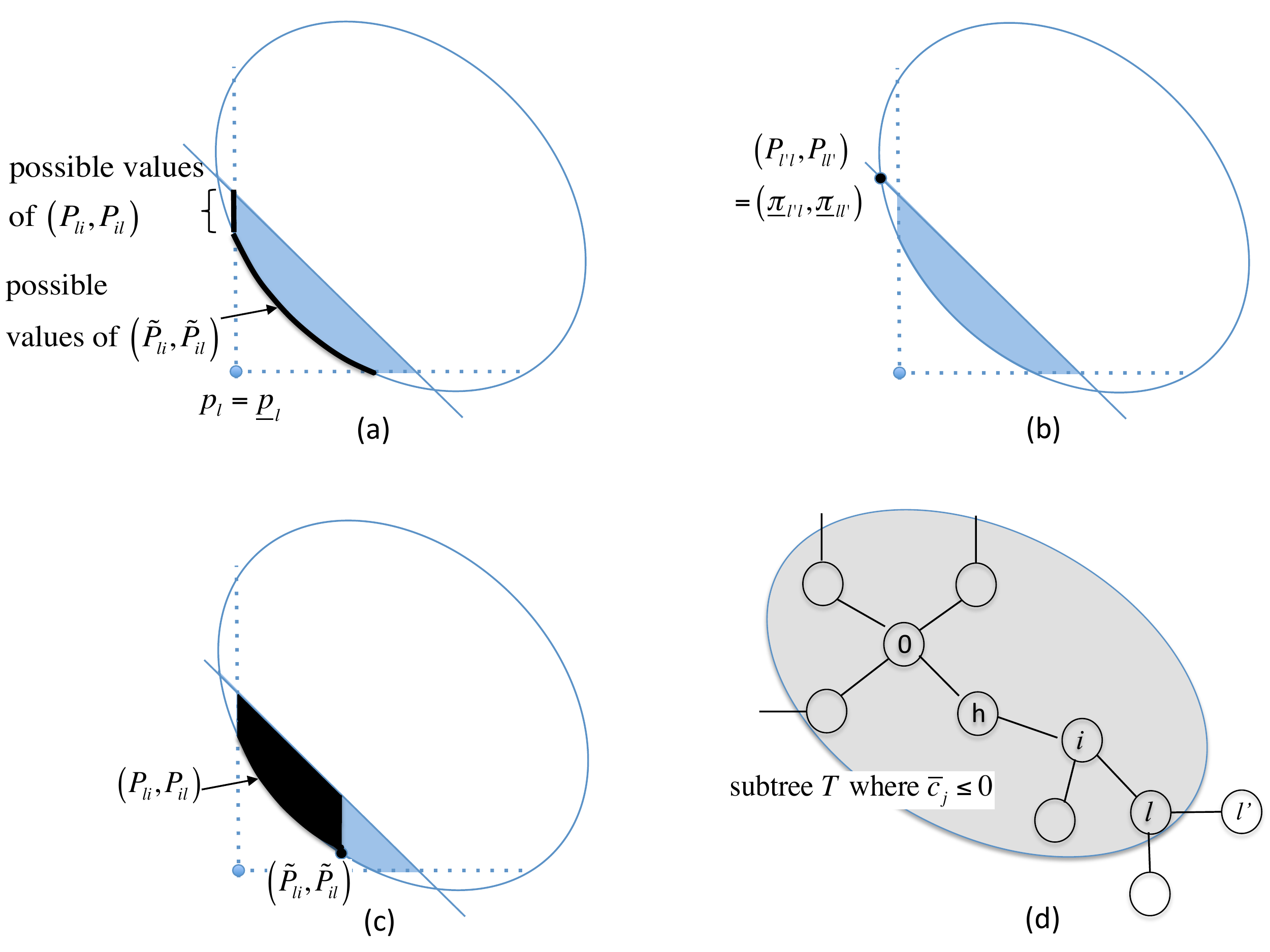}
\caption{Illustration for the proof of Lemma \ref{lemma:-vec}.}
\label{fig:-vec}
\end{figure}

Consider node $i$ one hop away from a boundary node towards root note 0
and let its parent be node $h$; see Figure \ref{fig:-vec}(d).
The above argument says that
$P_{il} \geq \tilde P_{il}$ for all neighbors $l$ of $i$ except its parent $h$.
This together with $p_i = \underline p_i$ (since $\overline c_i\leq 0$) implies
\bqn
P_{ih} \ = \  p_i - \sum_l P_{il}  & \leq & \tilde p_i - \sum_l \tilde P_{il} \ = \  \tilde P_{ih}
\eqn
and hence as before $P_{hi} \geq \tilde P_{hi}$.  Propagate towards the root
node 0 and \eqref{eq:PtildeP} follows by induction.

We now use \eqref{eq:PtildeP} to show that 
$(P_{li}, P_{il})\in \mathbb F_\theta^{li}$ for every link $(l,i)$ in the
subtree $T$.   Now \eqref{eq:PtildeP} implies that 
$P_{0l} \geq \tilde P_{0l}$ for all neighbors $l$ of 0.   Since node 0 has no 
parent, we have
\bqn
\sum_l P_{0l} \ = \ p_0  \ = \  \underline p_0 & \leq & \tilde p_0 \ = \ \sum_l \tilde P_{0l}
\eqn
implying $P_{0l} = \tilde P_{0l}$ for all neighbors $l$ of node 0.
This implies $P_{l0} = \tilde P_{l0}$; see Figure \ref{fig:-vec}(a) and (c).
Repeat this argument propagating from node 0 towards the boundary nodes
of the subtree $T$, and we conclude that
$(P_{li}, P_{il}) = (\tilde P_{li}, \tilde P_{il}) \in \mathbb F_\theta^{li}$ 
for every link $(l,i)$ in $T$.
This completes the proof of Lemma \ref{lemma:-vec}.
\end{proof}
This completes the proof of Lemma \ref{lemma:supset}.
\end{proof}

This completes the proof of Theorem \ref{thm:pfOPF}.

\vspace{0.2in}
\begin{proof}[Proof of Lemma \ref{lemma:misc}]
\bee
\item[(1)] Now $x\in \text{conv}(M\mathbb B)$ if and only if $x$ is a 
	finite convex combination of vectors in $M\mathbb B$, i.e., if and
	only if $x = \sum_j \alpha_j\, My_j = M\, \sum_j \alpha_j y_j$ for some 
	$y_j\in \mathbb B$, or equivalently, $x\in M\, \text{conv}(B)$.
	
	Similarly $x := (x^1,\ x^2) \in \text{conv}(\mathbb B \times \mathbb C)$ if and only if
	$(x^1,\  x^2) = \sum_j \alpha_j (x^1_j,\  x^2_j)$ for some $x^1_j\in \mathbb B$ and
	$x^2_j\in \mathbb C$ if and only if 
	$x^i = \sum_j \alpha_j x^i_j$, $i=1, 2$, i.e., 
	$x\in \text{conv}(\mathbb B) \times \text{conv}(\mathbb C)$.
	
\item[(2)] Now
	$x\in \mathbb O(M\mathbb B)$ if and only if there is a $c > 0$
	such that $x=\arg\min_{\hat{x}\in M\mathbb B} c^T\hat{x}$ 
	if and only if $x=My$ with
	$y=\arg\min_{\hat{y}\in \mathbb B} (M^Tc)^T \hat{y}$, i.e., 
	$y\in \mathbb O(\mathbb B)$ or equivalently $x\in M\mathbb O(\mathbb B)$.
	
	Similarly 
	$x := (x^1, \ x^2) \in \mathbb O(\mathbb B \times \mathbb C)$ if and only if
	$x$ solves, for some $c^1>0, c^2> 0$,
	\bqn
	\min_{(\hat x^{1}, \, \hat x^{2})\in 
			\mathbb B \times \mathbb C} (c^1)^T \hat x^{1} + (c^2)^T \hat x^{2}
	& \equiv & \min_{\hat x^{1}\in \mathbb B} (c^1)^T \hat x^{1} + 
			\min_{\hat x^{2}\in \mathbb C} (c^2)^T \hat x^{2}
	\eqn
	i.e., $x\in \mathbb O(\mathbb B) \times \mathbb O(\mathbb C)$.

\item[(3)] The key observation is that 
	$\mathbb B = \mathbb O(\text{conv}(\mathbb B))$, as opposed to
	$\mathbb B \supset \mathbb O(\text{conv}(\mathbb B))$, implies that
	\emph{every} point $x\in \mathbb B$ is a minimizer of $c^T\hat x$ over 
	$\text{conv}(\mathbb B)$ for some nonzero $c\geq 0$.  In particular every
	$x\in \mathbb B \cap \mathbb D$ is a minimizer of $c^T\hat x$ over 
	$\text{conv}(\mathbb B)$, for some nonzero $c\geq 0$, and hence is a
	minimizer over $\text{conv}(\mathbb B) \cap \mathbb D$.  This shows
	that $x\in \mathbb O(\text{conv}(\mathbb B) \cap \mathbb D)$, and hence
	$\mathbb B \cap \mathbb D \subseteq 
	\mathbb O(\text{conv}(\mathbb B) \cap \mathbb D)$.
\eee
\end{proof}

\subsection{Proof of Theorem \ref{thm:eqps}: mesh networks with phase shifters}

The proof follows that in \cite{Farivar-2013-BFM-TPS}.
\begin{proof}
We first prove $\overline{\mathbb X}_T \equiv  {\mathbb X}_{nc}$.  
It will then be clear that $\overline{\mathbb X}_T = \overline{\mathbb X}$.
To prove $\overline{\mathbb X}_T \equiv  {\mathbb X}_{nc}$,
we will exhibit a mapping $h: \overline{\mathbb X}_T \rightarrow \mathbb X_{nc}$ and
its inverse $h^{-1}$ and prove that $\tilde x\in  \overline{\mathbb X}_T$ if and only if
$x := h(\tilde x) \in \mathbb X_{nc}$, i.e., $\tilde x$ satisfies \eqref{eq:bfmps} if and only if
$x$ satisfies \eqref{eq:mdf} with equalities in \eqref{eq:mdf.3}.

Recall the function $h$ that, given any $\phi\in (-\pi, \pi]^m$ with $\phi\in T^\perp$, 
maps an $\tilde x \in \overline{\mathbb X}_T$ to a point $x \in \mathbb X_{nc}$:
\bqn
x \ := \ (S, \ell, v, s) \ := \ h(S, I, V, s) \ =: \ h(\tilde x) =: h(\tilde x; \phi)
\eqn
with $\ell_{jk} := | I_{jk} |^2$ and $v_j := |V_j|^2$.

We abuse (to simplify) notation and use $\theta$ to denote either an $n$-dimensional vector
$\theta := (\theta_j, j\in N)$ or an $(n+1)$-dimensional vector $\theta := (\theta_j, j\in N^+)$ 
with $\theta_0 := \angle V_0 := 0^\circ$, depending on the context.   
To construct an inverse of $h$, first consider, for each $\theta \in (-\pi, \pi]^{n+1}$,  
the mapping $\tilde h_\theta(S, \ell, v, s) = (S, I, V, s)$ 
from $\mathbb{R}^{3(m+n+1)}$ to $\mathbb{C}^{2(m+n+1)}$  by:
\begin{subequations}
\bq
V_j & := & \sqrt{v_j} \ e^{\ii \theta_j},  \qquad\qquad\qquad\ \ \   j \in N^+
\\
I_{jk} & := & \sqrt{\ell_{jk}}\ e^{\ii (\theta_j - \angle S_{jk})}, \qquad  j\rightarrow k \in \tilde E
\label{eq:bfs.3}
\eq
\label{eq:defhinv.2}
\end{subequations}
We now proceed in three steps:
(1) Prove that, given any $\beta(x)$, $x\in \mathbb X_{nc}$,
 there is a unique $(\theta(x), \phi(x))\in (-\pi, \pi]^{n+m}$ with $\phi(x)\in T^\perp$
that satisfies \eqref{eq:cyclecond.ps}.
(2) Prove that the function $h^{-1}(x) := \tilde h_{\theta(x)} (x)$ maps each 
$x \in {\mathbb X}_{nc}$
to an $\tilde x\in \overline{\mathbb X}_T$ that satisfies BFM with phase shifters \eqref{eq:bfmps};
(3) Prove that $h^{-1}$ as defined is indeed an inverse of $h$, establishing 
$\overline{\mathbb X}_T \equiv {\mathbb X}_{nc}$.

\noindent
\emph{Step 1: solution of \eqref{eq:cyclecond.ps} always exists.}
Fix an $x$ and the corresponding $\beta := \beta(x)$.
Write $\phi = [\phi_T^t\ \ \phi_{\perp}^t]^t$ and set $\phi_T = 0$.  Then
	\eqref{eq:cyclecond.ps} becomes
	\bq
	\begin{bmatrix}   B_T \\ B_{\perp}   \end{bmatrix}  \theta
	& = & 
	\begin{bmatrix}   \beta_T  \\  \beta_{\perp}   \end{bmatrix}  -  
	\begin{bmatrix}   0 \\  \phi_{\perp}   \end{bmatrix}
	+ 2\pi
	\begin{bmatrix}   k_T \\  k_{\perp}   \end{bmatrix}
	\label{eq:atb.4}
	\eq
	Hence a vector $({\theta}_*,  {\phi}_*, k_*)$ with $\theta_* \in (-\pi, \pi]^n$ and 
	${\phi}_*  \in T^\perp$ is a solution of \eqref{eq:atb.4} if and only if
	\bqn
		B_{\perp} B_T^{-1} \beta_T = \beta_{\perp} - [\phi_*]_{\perp} + 2\pi \left[ \hat{k}_* \right]_\perp
	\eqn
	where $\left[ \hat{k}_* \right]_\perp := [k_*]_\perp - B_{\perp} B_T^{-1} [k_*]_T$ is
	an integer vector.
	Clearly this can always be satisfied by choosing 
	\bq
	[\phi_*]_{\perp} - 2\pi \left[ \hat{k}_* \right]_\perp & = & \beta_{\perp} - B_{\perp} B_T^{-1}\beta_T
	\label{eq:phi}
	\eq
	Note that given $\theta_*$, $[k_*]_T$ is uniquely determined since $[\phi_*]_T = 0$,
	but $([\phi_*]_\perp, [k_*]_\perp)$ can be freely chosen to satisfy \eqref{eq:phi}.
	Hence we can choose the unique $[k_*]_\perp$ such that $[\phi_*]_\perp \in (-\pi, \pi]^{m-n}$.
	We have thus shown that there always exists a unique 	$({\theta}_*,  {\phi_*})$, 
	with $\theta_* \in (-\pi, \pi]^n$, $\phi_* \in (-\pi, \pi]^m$ and ${\phi}_*  \in T^\perp$, 
	that solves \eqref{eq:cyclecond.ps} for some $k_* \in \mathbb N^m$.
	Indeed this unique vector $(\theta_*, \phi_*)$ is given by
	\bqn
	\theta_* & = & \mathcal{P} \left( B_T^{-1}\beta_T  \right)
	\\
	\phi_*  & = &   \mathcal{P} \left( 
			\begin{bmatrix}
			0  \\  \beta_{\perp} - B_{\perp} B_T^{-1}\beta_T
			\end{bmatrix}
			\right)
	\eqn
	where $\mathcal{P}(\cdot)$ projects each component of a vector on to $(-\pi, \pi]$.

\noindent
\emph{Step 2: $\tilde h_{\theta(x)}(x)$ is in $\overline{\mathbb X}_T$.}
Fix any $x\in \mathbb X_{nc}$.   Let $(\theta(x), \phi(x))$ denote the unique vector in 
$(-\pi, \pi]^{n+m}$ with $\phi(x) \in T^\perp$ derived in Step 1 that solves \eqref{eq:cyclecond.ps}.
We claim that, given any $(\theta, \phi) \in (-\pi, \pi]^{n+m}$ with $\phi \in T^\perp$, 
$\tilde x := \tilde h_{\theta}(x)$ satisfies
BFM \eqref{eq:bfmps} with phase shifters  $\phi$ if and only if $\theta = \theta(x)$
and $\phi = \phi(x)$.
In that case, $\tilde x \in \overline{\mathbb X}_T$.
The argument is similar to the proof of \cite[Lemma 14]{Low2014a} without phase
shifters with the only change that $\tilde x$ here needs to satisfy \eqref{eq:bfmps.1} with possibly
nonzero $\phi$.

Specifically \eqref{eq:mdf.1} is equivalent to \eqref{eq:bfmps.3}; \eqref{eq:mdf.3} with equalities
and \eqref{eq:bfs.3} imply \eqref{eq:bfmps.2}.
For \eqref{eq:bfmps.1}, we have from \eqref{eq:mdf.2}, 
\bqn
|V_j|^2 & = & |V_i|^2 + |z_{ij}|^2 |I_{ij}|^2 - (z_{ij} S_{ij}^H + z_{ij}^H S_{ij})
\\
	& = & |V_i|^2 + |z_{ij}|^2 |I_{ij}|^2 - (z_{ij} V_i^H I_{ij} + z_{ij}^H V_i I_{ij}^H)
\\
	& = & |V_i - z_{ij} I_{ij}|^2
\eqn
Hence
\bq
|V_j\, e^{-\ii \phi_{ij}}| & = & |V_i - z_{ij} I_{ij}|
\label{eq:vje}
\eq
Since $(\theta(x), \phi(x))$ solves \eqref{eq:cyclecond.ps}, we have
\bqn
\theta_i - \theta_j & = & \angle \left( v_i - z_{ij}^H V_i I_{ij}^H \right) - \phi_{jk} +  2\pi k_{ij}
\ \, = \, \  \angle \ V_i \left(V_i - z_{ij} I_{ij} \right)^H \ - \ \phi_{ij} \ + \ 2\pi k_{ij}
\eqn
for some integer $k_{ij}$.
Hence
\bq
\theta_j \ - \ \phi_{ij}& = & \angle \left(V_i - z_{ij} I_{ij}\right)  \ - \ 2\pi k_{ij}
\label{eq:aVj}
\eq
But \eqref{eq:vje} and \eqref{eq:aVj} means 
\bqn
V_j\, e^{-\ii \phi_{ij}} & = & V_i - z_{ij} I_{ij}
\eqn
which is \eqref{eq:bfmps.1}, as desired.
Hence $\tilde x := \tilde h_{\theta(x)} (x) \in \overline{\mathbb X}_T$.

\noindent
\emph{Step 3: inverse of $h$.}   By definition of $h$, $h(\tilde x)$ is in $\mathbb X_{nc}$ for
every point $\tilde x\in \overline{\mathbb X}_T$.  
Let $h^{-1}(\cdot) := \tilde h_{\theta(\cdot)}(\cdot)$.
Step 2 shows that $h^{-1}(x)$ is in $\overline{\mathbb X}_T$ for every point $x\in \mathbb X_{nc}$.
Clearly $h(h^{-1}(x)) = h(\tilde h_{\theta(x)} (x)) = x$.  Hence $h$ and $h^{-1}$ are indeed inverses
of each other.
This establishes a bijection between $\overline{\mathbb X}_T$ and $\mathbb X_{nc}$, proving
their equivalence.

Finally, to show that $\overline{\mathbb X}_T = \overline{\mathbb X}$, write 
\eqref{eq:cyclecond.ps} as
	\bqn
	\begin{bmatrix}   B_T \\ B_{\perp}   \end{bmatrix}  \theta
	& = & 
	\begin{bmatrix}   \beta_T  \\  \beta_{\perp}   \end{bmatrix}  -  
	\begin{bmatrix}   \phi_T \\  \phi_{\perp}   \end{bmatrix}
	+ 2\pi
	\begin{bmatrix}   k_T \\  k_{\perp}   \end{bmatrix}
	\eqn
The same argument as in Step 1 (which is for the case with $\phi_T=0$) 
shows that, given any $\beta$, 
if the above equation has a solution $(\theta, \phi)$ then it has a solution with $\phi_T=0$.
Hence $\overline{\mathbb X}_T = \overline{\mathbb X}$. 

This completes the proof of Theorem \ref{thm:eqps}.
\end{proof}

\newpage
\bibliographystyle{unsrt}
\bibliography{PowerRef-201202}

\end{document}